\documentclass[11pt]{amsart}
\usepackage{amssymb, amsmath, latexsym, amsthm, amscd, fancyhdr, graphicx, xcolor, hyperref, bm}
\usepackage{nicefrac}
\def\R{\mathbb R}
\def\N{\mathbb N}

\def\Z{\mathbb Z}

\def\D{\mathbb D}

\def\H{\mathbb H}

\def\={\equiv}
\def\<{\langle}
\def\>{\rangle}
\def\eps{\varepsilon}

\def\rank{\operatorname{rank}}

\def\inv{^{-1}}

\def\supp{\operatorname{supp}}
\def\stab{\operatorname{Stab}}

\def\bms{m^{\operatorname{BMS}}}
\def\tbms{\tilde{m}^{\operatorname{BMS}}}
\def\r{\textbf{r}}

\def\1{\mathbf{1}}
\def\SO{\operatorname{SO}}

\def\norm#1{\left\Vert #1\right\Vert }
\def\opt1{\operatorname{T}^1}
\def\SL{\operatorname{SL}}

\def\t{\textbf{t}}
\def\s{\textbf{s}}

\def\calh{\mathcal{H}}

\def\ps{\mu^{\operatorname{PS}}}
\def\vps{\mu^{\operatorname{PS}-}}
\def\br{m^{\operatorname{BR}}}
\def\tbr{\tilde{m}^{\operatorname{BR}}}
\def\leb{\mu^{\operatorname{Leb}}}
\def\height{\operatorname{height}}
\def\inj{\operatorname{inj}}

\newtheorem{theorem}{Theorem}[section]

\newtheorem{proposition}[theorem]{Proposition}
\newtheorem{corollary}[theorem]{Corollary}
\newtheorem{lemma}[theorem]{Lemma}
\newtheorem{definition}[theorem]{Definition}

\newtheorem{remark}[theorem]{Remark}
\newtheorem{assumption}[theorem]{Assumption}

\title[Effective equidistribution]{Effective equidistribution of horospherical flows in infinite volume rank one homogeneous spaces}
\author[N. Tamam \and J. M. Warren]{Nattalie Tamam \and Jacqueline M. Warren}
\address[N. Tamam \and J. M. Warren]{Department of Mathematics, University of California, San Diego}

\begin{document}
	\maketitle
	\begin{abstract}
		We prove effective equidistribution of horospherical flows in $\SO(n,1)^\circ / \Gamma$ when $\Gamma$ is  geometrically finite and the frame flow is exponentially mixing for the Bowen-Margulis-Sullivan measure. We also discuss settings in which such an exponential mixing result is known to hold. As part of the proof, we show that the Patterson-Sullivan measure satisfies some friendly-like properties when $\Gamma$ is geometrically finite.
	\end{abstract}
	\tableofcontents
	\section{Introduction}\label{section; intro}
	
	The group $G = \SO(n,1)^\circ$ with $n\ge 2$, can be considered as the group of orientation preserving isometries of the hyperbolic space $\H^n$. Let $\Gamma \subseteq G$ be a geometrically finite and Zariski dense subgroup of $G$ with infinite covolume. 
	In this paper, we establish an effective rate of equidistribution of orbits under the action of a horospherical subgroup $U \subseteq G$ under a certain exponential mixing assumption (Assumption \ref{thm:effective mixing}).
	
	An early result on the equidistribution of horocyclic flows in $G/\Gamma$ for $G= \SL_2(\R)$ and $\Gamma$ a lattice was obtained by Dani and Smillie in \cite{DaniSmillie}. They proved that if $U=\left\{\begin{pmatrix}1 & t\\ 0 & 1 \end{pmatrix}:t\in \R\right\}$ and if $x$ does not have a closed $U$-orbit in $G/\Gamma$, then for every $f \in C_c(X)$, \begin{equation}\label{eqn; dani smillie thm}\lim\limits_{T \to \infty} \frac{1}{T} \int_0^T f(u_t x)dt = m(f),\end{equation} where $m$ denotes the normalized Haar probability measure on $X$. The lattice case is well-understood in general, thanks to Ratner's celebrated theorems on unipotent flows, \cite{Ratner}.
	
	Results such as these are not considered to be \emph{effective}, because they do not address the rate of convergence, and this is important in many applications. Burger proved effective equidistribution of horocyclic flows for $\SL_2(\R)/\Gamma$ when $\Gamma$ is a uniform lattice or convex cocompact with critical exponent at least $1/2$ in \cite{burger}. Sarnak proved effective equidistribution of translates of closed horocycles when $\Gamma$ is a non-uniform lattice in \cite{Sarnak}. More general results were obtained for non-uniform lattices using representation theoretic methods by Flaminio and Forni in \cite{FlaminioForni}, and also by Str\"ombergsson in \cite{Strombergsson}. The case where $\Gamma=\SL_2(\Z)$ was also obtained independently by Sarnak and Ubis in \cite{SarnakUbis}. The higher dimensional setting has recently been considered by Katz \cite{Katz} and McAdam \cite{McAdam}. McAdam proved equidistribution of abelian horospherical flows in $\SL_n(\R)/\Gamma$ for $n \ge 3$ when $\Gamma$ is a cocompact lattice or $\SL_n(\Z)$, and Katz proved equidistribution in greater generality when $\Gamma$ is a lattice in a semi-simple linear group without compact factors.
	
	In infinite volume, we cannot hope for a result such as equation (\ref{eqn; dani smillie thm}) for the Haar measure: by the Hopf ratio ergodic theorem, for almost every point,  $$\lim\limits_{T \to \infty} \frac{1}{T}\int_0^T f(u_t x)dt = 0.$$ This tells us that this is not the correct measure to consider. A key characteristic of the Haar probability measure in the lattice case is that it is the unique $U$-invariant ergodic Radon measure that is not supported on a closed $U$ orbit, \cite{DaniSmillie, Furstenberg}. By \cite{burger, roblin, winter}, the measure with this property in the infinite volume setting is the \emph{Burger-Roblin} (BR) measure, which is defined fully in \S\ref{section; notation}. The correct normalization will be given by the \emph{Patterson-Sullivan} (PS) measure, which is a geometrically defined measure on $U$ orbits. This is also defined in \S\ref{section; notation}.
	
	Maucourant and Schapira proved equidistribution of horocycle flows on geometrically finite quotients of $\SL_2(\R)$ in \cite{MauSchap}, and in \cite{joinings}, Mohammadi and Oh generalize these results to geometrically finite quotients of $\SO(n,1)^\circ$ for $n\ge2$, but these results are not effective. Oh and Shah also proved equidistribution on the unit tangent bundle of geometrically finite hyperbolic manifolds in \cite{OhShah}. In \cite{Edwards}, Edwards proves effective results for geometrically finite quotients of $\SL_2(\R)$. 
	
	In this paper, we extend these results to geometrically finite quotients of $\SO(n,1)^\circ$, under the assumption of exponential mixing of the frame flow for the \emph{Bowen-Margulis-Sullivan} (BMS) measure, which is defined in \S\ref{section; notation}. More explicitly, in \S\ref{section: proof of PS translates theorem}, \S\ref{section: proof of PS thm}, and \S\ref{section; long orbits} (but not \S\ref{sec:quantitative nondivergence} or \S\ref{subsection; friendly measures}), we will assume the following holds, where $\{a_s : s \in \R\}$ denotes the frame flow on $G/\Gamma$:

	
	
	\begin{assumption}[Exponential Mixing]\label{thm:effective mixing}
		There exist $c,\kappa>0$ and $\ell\in\N$ which depend only on $\Gamma$, such that for $\psi,\varphi\in C_c^\infty(G/\Gamma)$ and $s>0$,\[
		\left|\int_{X}\psi\left(a_{s}x\right)\varphi\left(x\right)d\bms\left(x\right)-\bms\left(\psi\right)\bms\left(\varphi\right)\right|<cS_\ell(\psi)S_\ell(\varphi)e^{-\kappa s}.\]
	\end{assumption}
	
	Assumption \ref{thm:effective mixing} is known to hold when $\Gamma$ is convex cocompact by \cite{mixing}. In \cite{Matrix coefficients}, Mohammadi and Oh prove such a result for geometrically finite $\Gamma$ under a spectral gap assumption (see Definition \ref{defn: spectral gap}), using decay of matrix coefficients. 
	Edwards and Oh recently proved effective mixing for the geodesic flow on the unit tangent bundle of a geometrically finite hyperbolic manifold when the critical exponent is larger than $(n-1)/2$ in \cite{EdwardsOh}.
	Further details about this assumption are discussed in \S\ref{sec:mixing}.
	
	We will need to restrict consideration to points satisfying the following geometric property, which means that the point does not travel into a cusp ``too fast''. Here, $d$ is a left-invariant Riemannian metric on $G/\Gamma$ that projects to the hyperbolic distance on $\H^n$.
	
	\begin{definition} \label{defn: diophantine} For $0<\eps<1$ and $s_0\ge 1$, we say that $x\in G/\Gamma$ with $x^- \in \Lambda(\Gamma)$ is \textbf{$(\eps,s_0)$-Diophantine} if for all $s\ge s_0$,
		\begin{equation*}
		d(\mathcal{C}_0,a_{-s}x)<(1-\eps)s,
		\end{equation*} where $\mathcal{C}_0$ is a compact set arising from the thick-thin decomposition, and is fully defined in \S\ref{section: thick thin decomposition}.
		We say that $x\in G/\Gamma$ with $x^- \in \Lambda(\Gamma)$ is \textbf{Diophantine} if $x$ is $(\eps,s_0)$-Diophantine for some $\eps$ and $s_0$.\end{definition} 
	
	Here, $\Lambda(\Gamma)$ denotes the set of limit points of $\Gamma$, and is defined fully in \S\ref{section; notation}, as is the notation $x^\pm$. In the case that $\Gamma$ is a lattice, the condition $x^- \in \Lambda(\Gamma)$ is always satisfied. Also, if $\Gamma$ is convex cocompact, every point $x \in G/\Gamma$ with $x^-\in \Lambda(\Gamma)$ will be Diophantine, because all limit points are \emph{radial} in this case (see \S\ref{section; notation}). 
	
	Note that $x$ is $(\eps,s_0)$-Diophantine if $(1-\eps)s$ is a bound on the asymptotic excursion rate of the geodesic $\left\{a_{-s}x\right\}$, i.e.
	\begin{equation} \label{eq:diphantine}
	\limsup_{s\rightarrow\infty} \frac{d(\mathcal{C}_0,a_{-s}x)}{s}\leq 1-\eps.
	\end{equation}
	
	Sullivan's logarithm law for geodesics when $\Gamma$ is geometrically finite with $\delta_\Gamma>(n-1)/2$ was shown in \cite{KelmerOh,StratmannVelani} (and is a strengthening of Sullivan's logarithm law for non-compact lattices (\cite[\S 9]{sullivan 2})), and implies that for almost all $x \in G/\Gamma$, 
	\begin{equation} \label{eq:logarithm law}
	\limsup\limits_{s\to\infty} \frac{d(\mathcal{C}_0,a_{-s}x)}{\log s} = \frac{1}{2\delta_\Gamma-k},
	\end{equation}
	where $k$ is the maximal cusp rank. 
	In \cite{KelmerOh}, Kelmer and Oh showed a strengthening of the above, considering excursion to individual cusps and obtaining a limit for the shrinking target problem of the geodesic flow. 
	Note also that the result stated in \cite{KelmerOh} is for $x\in\opt1(G/\Gamma$), but since the distance function there is assumed to be $K$-invariant, where $\H^n=K\backslash G$, and the set $\mathcal{C}_0$ is $K$-invariant as well (see \S\ref{section: thick thin decomposition}), we can deduce the form above. 
	
	It follows from \eqref{eq:logarithm law} that the limit on the left hand side of \eqref{eq:diphantine} is zero for almost every point $x\in G/\Gamma$ (with respect to the invariant volume measure) in this case.
	Moreover, for any $\eps$, the Hausdorff dimension of the set of directions in $\opt1(\H^n/\Gamma)$ around a fixed point in $\H^n/\Gamma$ that do not satisfy \eqref{eq:diphantine} is computed in \cite[Theorem 1]{MelianPestana}.
	For geometrically finite $\Gamma$, the Hausdorff dimension of the set of directions around a fixed point that do not satisfy \eqref{eq:diphantine} can be found in \cite{HillVelani,StratmannVelani}.

	The main goal of this paper is to establish the following two theorems. Here, $\br$ denotes the BR measure, $\bms$ denotes the Bowen-Margulis-Sullivan (BMS) measure, and $\ps$ denotes the PS measure. These measures are defined in \S\ref{section; notation}. Throughout the paper, the notation $$x \ll y$$ means there exists a constant $c$ such that $$x \le cy.$$ If a subscript is denoted, e.g. $\ll_\Gamma$, this explicitly indicates that this constant depends on $\Gamma$.
	
	Let $U = \{u_\t : \t \in \R^{n-1}\}$ denote the expanding horospherical flow. Let $B_U(r)$ denote the ball in $U$ of radius $r$ with the max norm on $\R^{n-1}$. See \S\ref{section; notation} for more details on notation.
	
	\begin{theorem}\label{thm; main PS} Assume that $\Gamma$ satisfies Assumption \ref{thm:effective mixing}.
		For any $0<\eps<1$ and $s_0\ge 1$, there exist constants $\ell = \ell(\Gamma) \in \N$ and $\kappa = \kappa(\Gamma, \eps)>0$ satisfying: for every $\psi \in C_c^\infty(G/\Gamma)$, there exists $c = c(\Gamma,\supp\psi)$ such that every $x\in G/\Gamma$ that is $(\eps,s_0)$-Diophantine, and for every $r \gg_{\Gamma,\eps} s_0$, \[\left|\frac{1}{\ps_x(B_U(r))} \int_{B_U(r)}\psi(u_\t x)d\ps_x(\t) - \bms(\psi)\right| \le c S_\ell(\psi) r^{-\kappa},\] where $S_\ell(\psi)$ is the $\ell$-Sobolev norm. 
	\end{theorem}
	
	For the Haar measure, we will prove the following equidistribution result:
	
	\begin{theorem}\label{thm; main} Assume that $\Gamma$ satisfies Assumption \ref{thm:effective mixing}.
		For any $0<\eps<1$ and $s_0\ge 1$, there exist $\ell=\ell(\Gamma)\in\N$ and $\kappa=\kappa(\Gamma,\eps)>0$ satisfying: for every $\psi \in C_c^\infty(G/\Gamma)$, there exists $c = c(\Gamma,\supp\psi)$ such that for every $x\in G/\Gamma$ that is $(\eps,s_0)$-Diophantine, and for all $r \gg_{\Gamma,\supp\psi,\eps} s_0$,\[\left|\frac{1}{\ps_x(B_U(r))} \int_{B_U(r)}\psi(u_\t x)d\t - \br(\psi)\right| \le c S_\ell(\psi)r^{-\kappa},\] where $S_\ell(\psi)$ is the $\ell$-Sobolev norm.
		
	\end{theorem}
	
	Note that the assumption that $x$ is Diophantine is required to obtain quantitative nondivergence results in \S\ref{sec:quantitative nondivergence}, which is key in proving the above theorems. The dependence on a Diophantine condition is necessary, and is analogous to known effective equidistribution results for when $\Gamma$ is a non-cocompact lattice (see \cite{McAdam,Strombergsson}).
	
	In \cite{distributions}, we apply the above result to obtain a quantitative ratio theorem for the distribution of orbits for $\Gamma$ acting on $U \backslash G$. This improves upon work of Maucourant and Schapira in \cite{MauSchap}.
	
	
	A key step towards proving Theorem \ref{thm; main PS} is the following, which is proved in \S\ref{section: proof of PS translates theorem}.
	
	\begin{theorem}\label{thm: PS translates thm} Assume that $\Gamma$ satisfies Assumption \ref{thm:effective mixing}.
		There exist $\kappa = \kappa(\Gamma)$ and $\ell = \ell(\Gamma)$ which satisfy the following: 
		for any $\psi \in C_c^\infty(X)$, there exists $c = c(\Gamma, \supp\psi)>0$ such that for any $f \in C_c^\infty(B_U(r))$, $0<r<1$, $x \in \supp\bms$, and $s\gg_{\Gamma} d(\mathcal{C}_0,x)$, we have 
		\begin{align*}
		\left|\int_U \psi(a_s u_\t x)f(\t)d\ps_x(\t) - \ps_x(f)\bms(\psi)\right| < cS_\ell(\psi)S_\ell(f)e^{-\kappa s}.
		\end{align*}
	\end{theorem}

	In \S\ref{section: proof of PS translates theorem}, we also prove an analogous statement for the Haar measure. Such a result is proven in \cite{Matrix coefficients} under a spectral gap assumption on $\Gamma$, but we show in this paper how to prove it whenever the frame flow is exponentially mixing.
	
	The proof will use similar techniques as in \cite{joinings, OhShah}; in particular, we will rely on Margulis' ``thickening trick'' from his thesis, \cite{MargulisThesis}. 
	
	In the proofs of our main theorems (Theorems \ref{thm; main PS} and \ref{thm; main}), we use partition of unity arguments. In particular, the bounds we get are on slightly bigger sets. As a result, we need an effective bound on the PS measure of a small neighborhood of a boundary of a ball relative to the PS measure of that ball. The following Theorem achieves this. It is shown using \cite[Lemma 3.8]{friendly} and \cite[Theorem 2]{StratmannVelani}:
	
		\begin{theorem}\label{thm: intro nicer friendly}
   There exists a constant $\alpha= \alpha(\Gamma)>0,$ such that for every $x \in G/\Gamma$ that is $(\eps,s_0)$-Diophantine, for every $0<s\le T^{\frac{\eps}{1-\eps}}$, every $0<\xi\ll_\Gamma 1,$ and every $T\gg_{\Gamma,\eps} s_0$,
	\[	\frac{\ps_{a_{-s}x}(B_U(\xi+T) )}{\ps_{a_{-s}x}(B_U(T))}-1\ll_\Gamma \xi^{\alpha}.\]
	\end{theorem}
	
	In the appendix, a stronger version is obtained under the assumption that all cusps of $G/\Gamma$ have maximal rank.

	This paper is organized as follows. In \S\ref{sec:mixing}, we discuss under what conditions Assumption \ref{thm:effective mixing} is known to hold. In \S\ref{section; notation}, we set out notation used in the article, and define the measures we will be using, along with proving some important facts about them. In \S\ref{sec:quantitative nondivergence}, we prove quantitative nondivergence of horospherical orbits of Diophantine points, which is needed in the following sections. In \S\ref{subsection; friendly measures}, we control the PS measure of the boundary of a set by proving Theorem \ref{thm: intro nicer friendly}. In \S\ref{section: proof of PS translates theorem}, we use Margulis' ``thickening trick'' to prove Theorem \ref{thm: PS translates thm} and an analogous result for the Haar measure, which are key in the proofs of Theorems \ref{thm; main PS} and \ref{thm; main}. In \S\ref{section: proof of PS thm}, we use quantitative nondivergence and Theorem \ref{thm: PS translates thm} to prove Theorem \ref{thm; main PS}. In \S\ref{section; long orbits}, we use Theorem \ref{thm: intro nicer friendly} and the Haar measure analogue of Theorem \ref{thm: PS translates thm} to prove Theorem \ref{thm; main}. Several technical details of the proof of Theorem \ref{thm: intro nicer friendly} are in the appendix, \S\ref{friendliness appendix}, and in \S\ref{section: appendix} we prove stronger statements hold in the setting that all cusps have maximal rank, because the PS measure is absolutely friendly (see \cite[Theorem 1.9]{friendly}).
	\newline

	\noindent\emph{Acknowledgements:} We would like to thank Amir Mohammadi for suggesting this problem, as well as for his support and guidance. We would also like to thank Wenyu Pan for bringing \cite{friendly} to our attention. We are also grateful to the anonymous referee for their insightful comments on an earlier version of this manuscript, which significantly improved the paper. The first author was partially supported by the Eric and Wendy Schmidt Fund for Strategic Innovation. The second author was supported in part by the National Science and Engineering Research Council of Canada (NSERC) PGSD3-502346-2017.

	\section{Known Exponential Mixing Results}\label{sec:mixing}
	Throughout the paper, we assume the existence of an exponential mixing result (see Assumption \ref{thm:effective mixing}). In this section we elaborate on the conditions under which such a result is known. 
	Here we assume that $\Gamma$ is a Zariski dense discrete subgroup of $G$.

	There is a natural action of $G$ on $\H^n$ and $\partial\H^n$, the hyperbolic $n$-space and its boundary, respectively. 
	Let $\Lambda(\Gamma)\subseteq\partial(\H^n)$ denote the limit set of $X$, i.e., the set of all accumulation points of $\Gamma z$ for some $z\in\H^n\cup \partial(\H^n)$.
	The \emph{convex core} of $X$ is the image in $X$ of the minimal convex subset of $\H^n$ which contains all geodesics connecting any two points in $\Lambda(\Gamma)$.
	We say that $\Gamma$ is \emph{convex cocompact} if the convex core of $\H^n /\Gamma$ is compact, and \emph{geometrically finite} if a unit neighborhood of the convex core of $\Gamma$ has finite volume. 

	For $\Gamma$ convex cocompact, Assumption \ref{thm:effective mixing} was proved by Sarkar and Winter in \cite[Theorem 1.1]{mixing}. 
	
	Fix a point $w_o \in \opt1(\H^n)$ and denote $M = \stab_G(w_o)$. 
	Denote by $\hat{G}$ and $\hat{M}$ the unitary dual of $G$ and $M$, respectively. A representation $(\pi,\mathcal{H})\in\hat{G}$ is called \emph{tempered} if for any $K$-finite $v\in\mathcal{H}$, the associated matrix coefficient function $g\mapsto\langle\pi(g)v,v\rangle$ belongs to $L^{2+\eps}(G)$ for any $\eps>0$, and \emph{non-tempered} otherwise. The non-tempered part of $\hat{G}$ consists of the trivial representation, and complementary series representations $\mathcal{U}(v,s-n+1)$ parameterized by $v\in\hat{M}$ and $s\in I_v$, where $I_v\subseteq(\frac{n-1}{2},n-1)$ is an interval depending on $v$ (see Hirai \cite{Hirai}). 
	
	\begin{definition}\label{defn: spectral gap}
		The space $L^2(X)$ has a \textbf{spectral gap} if there exist $\frac{n-1}{2}<s_0=s_0(\Gamma)<\delta$ and $n_0=n_0(\Gamma)\in\mathbb{N}$ such that
		\begin{enumerate}
			\item the multiplicity of $\mathcal{U}(v,\delta_{\Gamma}-n+1)$ contained in $L^2(X)$ is at most $dim(v)^{n_0}$ for any $v\in\hat{M}$;
			\item $L^2(X)$ does not weakly contain any $\mathcal{U}(v,s-n+1)$ with $s\in(s_0,\delta)$ and $v\in\hat{M}$.
		\end{enumerate}
	\end{definition}
	
	According to \cite[Theorem 3.27]{Matrix coefficients}, if $\delta_{\Gamma}>\frac{n-1}{2}$ for $n=2,3$, or if $\delta_{\Gamma}>n-2$ for $n\ge4$, then $L^2(X)$ has a spectral gap. If $\delta_\Gamma\leq\frac{n-1}{2}$, then there is no spectral gap, but it was conjectured that whenever $\delta_{\Gamma}>\frac{n-1}{2}$, $L^2(X)$ has a spectral gap (see \cite{Matrix coefficients}). Note that if there are cusps of maximal rank $n-1$, it follows that $\delta_{\Gamma}>\frac{n-1}{2}$.  
	
	For $\Gamma$ geometrically finite such that $L^2(X)$ has a spectral gap and $\delta_\Gamma > \frac{n-1}{2}$, Mohammadi and Oh stated in \cite[Theorem 1.6]{Matrix coefficients} an exponential mixing result similar to Assumption \ref{thm:effective mixing}. In their statement the constant $c$ depends on $\Gamma$ and the support of the functions. The dependence on the support of the functions arises in the last part of the proof (see \cite[\S 6.3]{Matrix coefficients}), and can be omitted by using the following lemma (the BR-measure is defined in \S \ref{subsec: BMS and BR}), hence obtaining a result of the form needed in Assumption \ref{thm:effective mixing}. 
	
	
	
	\begin{lemma}
		If $\delta>(n-1)/2$, then there exists $c=c(\Gamma)>0$ such that any $B\subset X$ of diameter smaller than $1$ satisfies $$\br(B)\le c.$$ 
	\end{lemma}
	
	\begin{proof}
		For any $g\in G$ denote \[
		\Phi_0(g)=|\nu_{g(o)}|, \]
		where $o$ is the projection of $w_o$ onto $\H^n$ and for any $x\in\H^n$, $\nu_x$ is the Patterson-Sullivan density defined in \S \ref{section; PS and Lebesgue}. Since $\Phi_0$ is $\Gamma$-invariant, it can be considered as a smooth function on $X$. 
		Moreover, by assuming $B$ contains $K=\operatorname{Stab}_G(o)$ and using the  Cauchy~Schwartz inequality, we get 
		\begin{align*}
		\br(B)&=\int_{B}\Phi_0(g)dm^{\operatorname{Haar}}(g)\\
		&\leq\sqrt{dm^{\operatorname{Haar}}(B)}\norm{\Phi_0}_2\\
		&\ll\norm{\Phi_0}_2. 
		\end{align*}
		According to \cite[\S 7]{sullivan} and by the assumption $\delta>(n-1)/2$, we have that $\phi_0\in L^2(X)$. 
	\end{proof}

	\section{Notation and Preliminiaries}
	\label{section; notation}
	
	Recall from \S\ref{section; intro} that $G = \SO(n,1)^\circ$ and $\Gamma \subseteq G$ is a geometrically finite Kleinian subgroup of $G$. 
	Denote $$X:=G/\Gamma.$$
	
	$G$ acts transitively on $\H^n$, the hyperbolic $n$-space. Fix a reference point $o\in\H^n$ and let $K=\text{Stab}_G(o)$, then $K\backslash G=\mathbb{H}^n$. 
	Let $\pi:G\rightarrow\H^n$ be the projection \begin{equation}\label{eq: defn of pi} \pi(g)=g(o).\end{equation} We will abuse notation and also write $\pi$ for the induced map from $G/\Gamma$ to $\H^n/\Gamma$. For convenience, we will assume throughout the paper that we have chosen $o$ so that $o\Gamma \in \pi(\mathcal{C}_0)$, where $\mathcal{C}_0$ is defined in \S\ref{section: thick thin decomposition}. This says that $o\Gamma$ is in the convex core of $\H^n/\Gamma$. 
	
	Let $d$ denote the left $G$-invariant metric on $G$ which induces the hyperbolic metric on $K\backslash G=\mathbb{H}^n$. 
	
	Recall from \S \ref{sec:mixing} that $\Lambda(\Gamma)\subseteq\partial(\H^n)$ denotes the limit set of $X$.
	We denote the Hausdorff dimension of $\Lambda(\Gamma)$ by $\delta_\Gamma$. It is equal to the critical exponent of $\Gamma$ (see \cite{Patterson}).
	
	
	
	
	We say that a limit point $\xi\in\Lambda(\Gamma)$ is \emph{radial} if there exists a compact subset of $X$ so that some (and hence every) geodesic ray toward $\xi$ has accumulation points in that set. 
	An element $g\in G$ is called \emph{parabolic} if the set of fixed points of $g$ in $\partial(\mathbb{H}^n)$ is a singleton. We say that a limit point is \emph{parabolic} if it is fixed by a parabolic element of $\Gamma$.  A parabolic limit point $\xi\in\Lambda(\Gamma)$ is called \emph{bounded} if the stabilizer $\Gamma_\xi$ acts cocompactly
	on $\Lambda(\Gamma)-\{\xi\}$. 
	
	We denote by $\Lambda_r (\Gamma)$ and $\Lambda_{bp} (\Gamma)$ the set of all radial limit points and the set of all bounded parabolic limit points, respectively.
	Since $\Gamma$ is geometrically finite (see \cite{bowditch}),  \[\Lambda (\Gamma)=\Lambda_r (\Gamma)\cup\Lambda_{bp} (\Gamma).\]
	

	
	Fix $w_o \in \opt1(\H^n)$ and let $M = \stab_G(w_o)$ so that $\opt1(\H^n)$ may be identified with $M\backslash G$. For $w \in \opt1(\H^n)$, $$w^\pm \in \partial\H^n$$ denotes the forward and backward endpoints of the geodesic $w$ determines. For $g \in G$, we define $$g^\pm := w_o^\pm g.$$ Without loss of generality, we may assume that $w_o^\pm\in\Lambda(\Gamma)$, and hence every $\gamma \in \Gamma$ will satisfy $\gamma^\pm\in\Lambda(\Gamma).$
	
	Let $A=\left\{a_s\::\:s\in\R\right\}$ be a one parameter diagonalizable subgroup such that $M$ and $A$ commute, and such that the right at action on $M\backslash G=\opt1(\H^n)$ corresponds to unit speed geodesic flow. We parametrize $A$ by $A=\{a_s : s\in \R\}$, where
	\begin{equation}\label{eq: a_s defn}
	    a_s = \begin{pmatrix} e^{s} & & \\ & I & \\ & & e^{-s} \end{pmatrix}
	\end{equation}
	and $I$ denotes the $(n-1)\times(n-1)$ identity matrix. 
	
	Let $U$ denote the expanding horospherical subgroup\[
	U=\left\{g\in G\::\:a_{-s}ga_{s}\rightarrow e\text{ as }s\rightarrow+\infty \right\},\]
	let $\tilde U$ be the contracting horospherical subgroup\[
	\tilde U=\left\{g\in G\::\:a_{s}ga_{-s}\rightarrow e\text{ as }s\rightarrow+\infty \right\},\]
	and let $P=MA\tilde U$ be the parabolic subgroup.
	
	The group $U$ is a connected abelian group, isomorphic to $\R^{n-1}$. We may use the parametrization $\t\mapsto u_\t$ so that for any $s\in\R$,\begin{equation}
	a_{s}u_\t a_{-s}=u_{e^s\t}.\label{eqn; reln bt a and u}\end{equation} Similarly, we parametrize $\tilde U$ by $\t \mapsto v_\t \in \tilde{U}$ so that for $s \in \R$, \begin{equation} a_{s}v_\t a_{-s} = v_{e^{-s}\t} \label{eqn; reln bt a and u-}.\end{equation} More explicitly, if $\t \in \R^{n-1}$ is viewed as a row vector, 
	\begin{equation}\label{eq: defn u_t}
	    u_\t = \begin{pmatrix} 1 & \t & \frac{1}{2}\norm{\t}^2 \\ & I & \t^T \\ & & 1 \end{pmatrix}
	\end{equation}
	and $$v_\t =  \begin{pmatrix} 1 & & \\ \t^T & I & \\ \frac{1}{2}|\t|^2 & \t & 1  \end{pmatrix}.$$
	
	For a subset $H$ of $G$ and $\eta>0$, $H_\eta$ denotes the closed $\eta$-neighborhood of $e$ in $H$, i.e. \[
	H_\eta=\left\{h\in H\::\:d(h,e)\le\eta\right\}.\]
	For any $r>0$ let \[
	B_U (r)=\left\{u_\t\::\:\norm{\t}\le r\right \}\quad\mbox{and}\quad B_{\tilde{U}} (r)=\left\{v_\t\::\:\norm{\t}\le r\right \},\] 
	where $\norm{\t}$ is the sup-norm of $\t\in\R^{n-1}$. 
	
	\begin{lemma}\label{lem:rho}
		For $0<\eta<1/4$ and $p \in P_\eta$, there exists $\rho_p : B_U(1) \to B_U(1+O(\eta))$ that is a diffeomorphism onto its image and a constant $D=D(\eta)<3\eta$ such that $$u_\t p\inv \in P_D u_{\rho_p(\t)}.$$ Explicitly, if $p = a_s v_\r$, then $\rho_p(\t) = \dfrac{\t + \frac{1}{2}\|\t\|^2\r}{e^s(1-(\t\cdot\r)+\frac{1}{4}\|\r\|^2\|\t\|^2)}$.
	\end{lemma}
	\begin{proof}
		For $s \in \R$ and $\r \in \R^{n-1}$, let $p = a_sv_\r$. Then $p\inv = \begin{pmatrix} e^{-s} & & \\ -e^{-s} \r^T & I & \\ \frac{1}{2}e^{-s}\|\r\|^2 & -\r & e^{s} \end{pmatrix},$ 
		so $$u_\t p\inv = \begin{pmatrix} e^{-s}(1-(\t\cdot \r)+\frac{1}{4}\|\r\|^2\|\t\|^2) 
		& \t - \frac{1}{2} \|\t\|^2 \r & \frac{1}{2}e^{s}\|\t\|^2 \\ -e^{-s}\r^T + \frac{1}{2}e^{-s}\|\r\|^2\t^T & I - \t^T \r & e^{s}\t^T \\ \frac{1}{2}e^{-s}\|\r\|^2 & -\r & e^{s}\end{pmatrix}.$$
		
		Now, if $p' = a_{s'}v_{\r'},$ we obtain that $$p'u_{\t'} = \begin{pmatrix} e^{s'} & e^{s'}\t' & \frac{1}{2}e^{s'}\|\t'\|^2 \\ \r'^T & \r'^T\t' + I & \frac{1}{2}\|\t'\|^2\r'^T + \t'^T \\ \frac{1}{2}e^{-s'}\|\r'\|^2 & \frac{1}{2}e^{-s'}\|\r'\|^2\t' + e^{-s'}\r' & e^{-s'}(\frac{1}{4}\|\r'\|^2\|\t'\|^2 +(\r'\cdot \t')+1)\end{pmatrix}.$$ 
		We wish to solve for $\t'$.
		
		Setting entries equal yields $$\t + \frac{1}{2}\|\t\|^2\r = e^{s'}\t'$$ and \begin{equation} \label{eqn; exp s'}e^{s'}=e^{-s}\left(1-(\t\cdot\r)+\frac{1}{4}\|\r\|^2\|\t\|^2\right).\end{equation} Combining these implies that $$\t' = \frac{\t + \frac{1}{2}\|\t\|^2\r}{e^s(1-(\t\cdot\r)+\frac{1}{4}\|\r\|^2\norm{\t}^2)}.$$ We define $\rho_p(\t)$ to be this quantity. One can directly check that it satisfies the claim. 

	\end{proof}

	


	\subsection{Patterson-Sullivan and Lebesgue Measures}\label{section; PS and Lebesgue}
	A family of finite measures $\{\mu_x\::\:x\in\H^n\}$ on $\partial(\H^n)$ is called a \emph{$\Gamma$-invariant conformal density of dimension $\delta_\mu>0$} if for every $x,y\in\H^n$, $\xi\in\partial(\H^n)$ and $\gamma\in\Gamma$,\begin{equation}\label{eq: conformal density def}
	\gamma_*\mu_x=\mu_{x\gamma}\:\text{ and }\:\frac{d\mu_y}{d\mu_x}(\xi)=e^{-\delta_\mu \beta_\xi(y,x)},\end{equation}
	where $\gamma_*\mu_x(F)=\mu_x(F\gamma)$ for any Borel subset $F$ of $\partial(\H^n)$. 
	
	We let $\{\nu_x\}_{x \in \H^n}$ denote the Patterson-Sullivan density on $\partial \H^n$, that is, the unique (up to scalar multiplication) conformal density of dimension $\delta_\Gamma$.
	
	For each $x\in\H^n$, we denote by $m_x$ the unique probability measure on $\partial(\H^n)$ which is invariant under the compact subgroup $\text{Stab}_G (x)$. Then $\left\{m_x\::\:x\in \H^n\right\}$ forms a $G$-invariant conformal density of dimension $n-1$, called the Lebesgue density.
	Fix $o \in \H^n$. 
	
	For $x, y \in \H^n$ and $\xi \in \partial(\H^n)$, the \emph{Busemann function} is given by \[ \beta_\xi(x,y):=\lim\limits_{t\to \infty} d(x,\xi_t) - d(y,\xi_t)\] where $\xi_t$ is a geodesic ray towards $\xi$.
	
	For $g\in G$, we can define measures on $Ug$ using the conformal densities defined previously. The Patterson-Sullivan measure (abbreviated as the PS-measure):
	\begin{equation}\label{eqn; defn of ps} 
	d\ps_{Ug}(u_\t g) := e^{\delta_\Gamma \beta_{(u_\t g)^+}(o, u_\t g(o))}d\nu_o((u_\t g)^+),
	\end{equation} 
	and the Lebesgue measure \[
	\leb_{Ug}(u_\t g):=e^{(n-1) \beta_{(u_\t g)^+}(o, u_\t g(o))}dm_o((u_\t g)^+).\] We similarly define the opposite PS measure on $\tilde Ug$:  \begin{equation}\label{eqn; defn of ps-}
	d\vps_{\tilde Ug}(v_\t g) := e^{\delta_\Gamma \beta_{(v_\t g)^-}(o,v_\t g(o))}d\nu_o((v_\t g)^-).\end{equation}
	
	The conformal properties of $m_x$ and $\nu_x$ imply that these definitions are independent of the choice of $o\in\H^n$.
	
	We often view $\ps_{Ug}$ as a measure on $U$ via $$d\ps_{g}(\t):= d\ps_{Ug}(u_\t g),$$ and similarly for $\vps_{\tilde U g}$ on $\tilde U$. For $g \in G, s \in \R$, and $E\subseteq U$ a Borel subset (or $E \subseteq \tilde U$ for $\vps$), these measures satisfy: 
	\begin{align}
	& \leb_g(E) = e^{(n-1)s}\leb_{a_{-s}g}(a_{-s}E a_{s}), \label{eqn; leb scaling}\\
	& \ps_g(E) = e^{\delta_\Gamma s}\ps_{a_{-s}g}(a_{-s} E a_{s}), \label{eqn; ps scaling}\\
	& \vps_g(E) = e^{\delta_\Gamma s}\vps_{a_{s}g}(a_{s} E a_{-s}). \label{eqn; ps- scaling}
	\end{align}
	In particular, 
	\[\ps_g(B_U(e^s)) = e^{\delta_\Gamma s} \ps_{a_{-s}g}(B_U(1)) \text{ and } \vps_g(B_{U^-}(e^{s})) = e^{\delta_\Gamma s } \vps_{a_s g} (B_U(1)).\]
	
	The measure
	\[
	d\leb_{Ug}(u_\t g)=d\leb_U(u_\t)=d\t\]is independent of the orbit $Ug$
	and is simply the Lebesgue measure on $U\equiv\R^{n-1}$ up to a scalar multiple.

	
	
	We will need the following fundamental results, which are stated for $\ps$ and $U$, but also hold if we replace them with $\vps$ and $\tilde U$. 
	
	\begin{lemma}\label{lem; continuity g to psg} The map $g \mapsto \ps_g$ is continuous, where the topology on the space of regular Borel measures on $U$ is given by $\mu_n \to \mu \iff \mu_n(f)\to\mu(f)$ for all $f\in C_c(U)$.
	\end{lemma}
	\begin{proof}
		This is clear from the definition of the PS measure, since it is defined using the Busemann function and stereographic projection.
	\end{proof}
	
	\begin{corollary}\label{cor; inf sup on compact sets for ps}
		For any compact set $\Omega \subseteq G$ and any $r>0$, \[0<\inf\limits_{g \in \Omega, g^+ \in \Lambda(\Gamma)}\ps_g(B_U(r)g) \le \sup\limits_{g\in \Omega, g^+\in \Lambda(\Gamma)} \ps_g(B_U(r)g) < \infty.\]
	\end{corollary}
	
	To define the PS measure on $Ux$ for $x \in X,$ note that
	\begin{equation} \label{eq: radial U orbit injects}\text{if }x^- \in \Lambda_r(\Gamma),\text{ then } u \mapsto ux \text{ is injective,}\end{equation} and we can define the PS measure on $Ux \subseteq X$, denoted $\ps_x$, simply by pushforward of $\ps_g$, where $x = g\Gamma$. In general, defining $\ps_x$ requires more care, see e.g. \cite[\S 2.3]{joinings} for more details. As before, we can view $\ps_x$ as a measure on $U$ via \[
	d\ps_x(\t)=d\ps_x(u_{\t}x).\]

	\subsection{Thick-thin Decomposition and the Height Function}\label{section: thick thin decomposition}
	There exists a finite set of $\Gamma$-representatives $\xi_1,\dots,\xi_q\in\Lambda_{bp}(\Gamma)$. For $i=1,\dots,q$, fix $g_i\in G$ such that $g_i^- =\xi_i$,
	and for any $R>0$, set
	\begin{equation}\label{eq:siegel sets}
	\calh_i(R):=\bigcup_{s>R} Ka_{-s} U g_i,\quad\mbox{and}\quad\mathcal{X}_i(R):=\calh_i(R)\Gamma
	\end{equation}
	(recall, $K=\text{Stab}_G (o)$). 
	Each $\calh_i(R)$ is a horoball of depth $R$.
	
	The \emph{rank} of $\mathcal{H}_i(R)$ is the rank of the finitely generated abelian subgroup $\Gamma_{\xi_i}=\operatorname{Stab}_\Gamma(\xi_i)$. We say that the cusp has maximal rank if $\rank\Gamma_\xi=n-1$. 
	It is known that each rank is strictly smaller than $2\delta_\Gamma$.
	
	
	We denote $$\supp\bms:=\left\{g\Gamma\in X\::\:g^{\pm}\in\Lambda(\Gamma)\right\}.$$ (For now, this is simply notation. The measure $\bms$ will be defined in the next section, and this set is its support. It projects onto the convex core of $\H^n/\Gamma.$) Note that the condition $g^\pm \in \Lambda(\Gamma)$ is independent of the choice of representative of $x=g\Gamma$ in the above definition, because $\Lambda(\Gamma)$ is $\Gamma$-invariant. Thus, the notation $x^\pm \in \Lambda(\Gamma)$ is well-defined, even though $x^\pm$ itself is not.
	
	According to \cite{bowditch}, there exists $R_0\geq 1$ such that $\mathcal{X}_1(R_0),\dots,\mathcal{X}_q(R_0)$  are disjoint, and for some compact set $\mathcal{C}_0\subset G/\Gamma$,  \[
	\supp\bms\subseteq\mathcal{C}_0\sqcup\mathcal{X}_1(R_0)\sqcup\cdots\sqcup\mathcal{X}_q(R_0).\]
	
	For $1\le i\le q$ and $R\ge R_0$, denote\[
	{\mathcal{X}}(R):=\mathcal{X}_1(R)\sqcup\cdots\sqcup\mathcal{X}_q(R),\quad {\mathcal{C}}(R):=\supp\bms-{\mathcal{X}}(R).\]
	
	We will need a version of Sullivan's shadow lemma, obtained by Schapira-Maucourant (see Proposition 5.1 and Remark 5.2 in \cite{MauSchap}).

	\begin{proposition}\label{prop:shadow lemma}
		There exists a constant $\lambda=\lambda(\Gamma)\ge1$ such that for all $x \in \supp\bms$ and all $T>0$, we have \begin{align}\lambda\inv T^{\delta_\Gamma}e^{(k_1(x,T)-\delta_\Gamma)d(\pi(\mathcal{C}_0),\pi(a_{-\log T}x))}
		&\le \ps_x(B_U(T)) \\
		&\le \lambda T^{\delta_\Gamma}e^{(k_1(x,T)-\delta_\Gamma)d(\pi(\mathcal{C}_0),\pi(a_{-\log T}x))}\end{align} and
		\begin{align}
		\lambda\inv T^{\delta_\Gamma}e^{(k_2(x,T)-\delta_\Gamma)d(\pi(\mathcal{C}_0),\pi(a_{\log T}x))} &\le \vps_x(B_{\tilde U}(T)) \\
		&\le \lambda T^{\delta_\Gamma}e^{(k_2(x,T)-\delta_\Gamma)d(\pi(\mathcal{C}_0),\pi(a_{\log T}x))}, \nonumber
		\end{align}
		where $k_1(x,T)$ is the rank of $\mathcal{X}_i(R_0)$ if $a_{-\log T}x\in\mathcal{X}_i(R_0)$ for some $1\le i\le\ell$ and equals 0 if $a_{-\log T}x\in\mathcal{C}_0$, and $k_2(x,T)$ is defined analogously for $a_{\log T}x$. Recall the definition of $\pi$ from \eqref{eq: defn of pi} as the projection from $G$ to $\H^n$.\end{proposition}
	
	
	\begin{definition}\label{defn: height}
	    For $x \in G/\Gamma$, we define the \emph{height} of $x$ by 
	    \begin{equation}\label{eq:defn of height} \height(x)=d(\pi(\mathcal{C}_0),\pi(x)), \end{equation} 
	    where $\pi:G/\Gamma \to \H^n/\Gamma$ is the projection map as in \eqref{eq: defn of pi}, recalling that $\H^n/\Gamma \cong K\backslash G/\Gamma$.
	\end{definition}

	\begin{lemma}\label{lem: height and distance to C0}
		For any $x\in\supp\bms$ and $R\ge R_0$, we have that $$x\in\mathcal{C}(R) \iff \height(x)\le R-R_0.$$
	\end{lemma}
	
	\begin{proof}
	The claim follows from the disjointness of $\mathcal{X}_i(R_0)$, $1\le i\le q$ from $\mathcal{C}_0$, and the fact that $\mathcal{X}_i(R)\subseteq \mathcal{X}_i(R_0)$: 
	
	If $x\in\mathcal{C}(R)$, then either $x\in\mathcal{C}_0$, in which case $\height(x)=0$ and we are done, or $x\in\mathcal{X}_i(R_0)$. Assume the latter, then the Busemann function between $x$ and the boundary of  $\mathcal{X}_i(R_0)$ (which intersects $\mathcal{C}_0)$ is at most $R-R_0$. Thus, we may deduce the claim in this case. 
	
	Next, assume $x\in\mathcal{X}_i(R)$ for some $i$. The Busemann function between two points in different horoballs is at least $R-R_0$. Since a point from $\mathcal{X}_i(R)$ cannot go into $\mathcal{C}_0$ without passing through $\mathcal{X}_i(R_0),$ this is a lower bound for the distance between the base points, i.e. the height.
	\end{proof}
	
	\begin{corollary}\label{cor: reln height and s0}
	 Let $x \in G/\Gamma$ be $(\eps,s_0)$-Diophantine. Then $$\height(x) < (2-\eps)s_0.$$
	\end{corollary}
	\begin{proof}
	    By Definition \ref{defn: diophantine}, $$d(\mathcal{C}_0, a_{-s_0}x)<(1-\eps)s_0.$$ Hence, we have that \begin{align*}
	    \height(x)&\le d(\mathcal{C}_0,x)\\
	    &< d(\mathcal{C}_0, a_{-s_0}x)+d(a_{-s_0}x,x) \\
	    &< (1-\eps)s_0+s_0.
	    \end{align*} 
	    \end{proof}
	
	The \emph{injectivity radius} at $x \in X$ is defined to be the supremum over all
	$\eps>0$ such that the map $$h\mapsto hx \text{ is injective on }G_\eps.$$ We denote the injectivity radius at $x$ by $$\inj(x).$$ The injectivity radius of a set $\Omega$ is defined to be $$\inf_{x \in \Omega} \inj(x).$$
	
	By the proof of \cite[Proposition 6.7]{isolations}, there exists a constant $\sigma = \sigma(\Gamma) > 0$ such that for all $x \in \supp\bms$, \begin{equation}\label{eq: inj radius and height}
	\sigma\inv\inj(x) \le e^{-\height(x)} \le {\sigma} \inj(x).
	\end{equation}


    
	
	The following fact is well-known, but we include a proof for completion.
	
	\begin{lemma}\label{lem: horodist diophantine points}
	    There exists $T_0=T_0(\Gamma)>0$ which satisfies the following. Let $x \in G/\Gamma$ with $x^- \in \Lambda(\Gamma)$, and let $R>0$ be such that $d(\mathcal{C}_0,x)<R.$ Then there exists $\t \in B_U(2(R+T))$ such that $$(u_\t x)^\pm\in\Lambda(\Gamma).$$ 
	    In particular, for every $0<\eps<1$, $s_0\ge1$, and  $(\eps,s_0)$-Diophantine point $x$, there exists $|\t|\ll_\Gamma s_0$ such that $$(u_\t x)^\pm\in\Lambda(\Gamma).$$
  	\end{lemma}
  	\begin{proof}
  	    Let $g,h' \in G$ be such that $x=g\Gamma$,  $h'^-=g^-$, $h'\Gamma\in K\mathcal{C}_0$, and $$d(g,h')\le\height(x)<R.$$ 
  	    Since $K\mathcal{C}_0$ is a compact set, by \cite[Lemma 3.3]{joinings}, there exists a constant $T_0$, which only depends on $\mathcal{C}_0$ (i.e., on $\Gamma$) such that for some $\t\in B_U(T_0)$, $$(u_\t h')^\pm\in\Lambda(\Gamma). $$
  	    
  	    
  	      Fix $h:=u_\t h'$ and observe that \begin{equation}   
  	      d(g,h) < R+T_0.  	      
  	      \end{equation} We must flow $h\Gamma$ with an element of $A$ so that it lies on $Ux.$ 
  	      
  	      Because $h^- = g^-,$ if $s = \beta_{g^-}(h,g)$, then $$a_s h \in Ug.$$ 
  	      Since $\beta_{g^-}(h,g) \le d(h,g)$, we arrive at
  	      \begin{align*}
  	          d(g,a_s h)&\le d(g,h)+d(h,a_s h)\\
  	          &\le 2d(g,h)\\
  	          &\le 2\left(R+T\right).
  	      \end{align*}
  	      
  	      For $(\eps,s_0)$-Diophantine $x$, observe that \begin{align*}
            d(\mathcal{C}_0,x)&\le d(\mathcal{C}_0,a_{-s_0}x) + d(a_{-s_0},x) \\
            &< (1-\eps)s_0+s_0 \\
            &< 2s_0,
        \end{align*} so we see that $R=2s_0$ works for all such points.
  	\end{proof}

	\subsection{Bowen-Margulis-Sullivan and Burger-Roblin Measures} \label{subsec: BMS and BR}
	Recall $\pi : G \to \H^n$ from \eqref{eq: defn of pi}. In this section, we will abuse notation and write $\pi$ for the restriction of $\pi$ to $\opt1(\H^n)\cong M\backslash G.$
	Recalling the fixed reference point $o \in \H^n$ as before, the map $$w \mapsto(w^+,w^-, s:= \beta_{w^-}(o,\pi(w)))$$ is a homeomorphism between $\opt1(\H^n)$ and $$(\partial(\H^n)\times\partial(\H^n) - \{(\xi,\xi):\xi \in \partial(\H^n)\})\times \R.$$
	
	This homeomorphism allows us to define the Bowen-Margulis-Sullivan (BMS) and Burger-Roblin (BR) measures on $\opt1(\H^n)$, denoted by $\tbms$ and $\tbr$ respectively: \[d\tbms(w):=e^{\delta_\Gamma \beta_{w^+}(o,\pi(w))}e^{\delta_\Gamma \beta_{w^-}(o,\pi(w))}d\nu_o(w^+)d\nu_o(w^-)ds,\]
	\[d\tbr(w):= e^{(n-1)\beta_{w^+}(o,\pi(w))}e^{\delta_\Gamma \beta_{w^-}(o,\pi(w))} dm_o(w^+)d\nu_o(w^-)ds.\]
	
	The conformal properties of $\left\{\nu_x\right\}$ and $\left\{m_x\right\}$ imply that these definitions are independent of the choice of $o\in\H^n$. Using the identification of $\opt1(\H^n)$ with $M\backslash G$, we lift the above measures to $G$ so that they are all invariant under $M$ from the left. By abuse of notation, we use the same notation ($\tbms$ and $\tbr$). These measures are left $\Gamma$-invariant, and hence induce locally finite Borel measures on $X$, which are the Bowen-Margulis-Sullivan measure $\bms$ and the Burger-Roblin measure $\br$, respectively.
	
	Note that 
	\begin{equation*}
	\supp\bms:=\left\{x\in X\::\:x^{\pm}\in\Lambda(\Gamma)\right\}
	\end{equation*}
	and 
	\begin{equation*}
	\supp\br=\left\{x\in X\::\:x^{-}\in\Lambda(\Gamma)\right\}.
	\end{equation*}
	
	Recall $P = MA\tilde{U}$, which is exactly the stabilizer of $w_o^+$ in $G$. We can define another measure $\nu$ on $Pg$ for $g \in G$, which will give us a product structure for $\tbms$ and $\tbr$ that will be useful in our approach. For any $g\in G$ define \begin{equation}\label{eq; defn of nu}
	d\nu(pg) := e^{\delta_\Gamma \beta_{(pg)^-}(o, pg(o))} d\nu_o(w_o^-pg)dmds,
	\end{equation} on $Pg$, where $s = \beta_{(pg)^-}(o,pg(o))$, $p = mav \in MA\tilde{U}$ and $dm$ is the probability Haar measure on $M$. 
	
	Then for any $\psi\in C_c(G)$ and $g\in G$, we have 
	\begin{equation}\label{eq:bms product structure}
	\tbms(\psi)=\int_{Pg}\int_U \psi(u_\t pg)d\ps_{pg}(\t)d\nu(pg),
	\end{equation}
	and 
	\begin{equation}\label{eq:br product structure}
	\tbr(\psi)=\int_{Pg}\int_U \psi(u_\t pg)d\t d\nu(pg). 
	\end{equation}

	\begin{lemma}\label{lem:bound on nu}  
		There exists a constant $\lambda=\lambda(\Gamma) > 1$ such that for all $g \in \supp\tbms$ and all $0<\eps<\inj(g)$, we have \begin{align*}&\lambda\inv \eps^{\delta_\Gamma +\frac{1}{2}(n-1)(n-2)+1}e^{(k_2(x,\eps)-\delta_\Gamma)d(\pi(\mathcal{C}_0), \pi(a_{\log \eps}x))} \\
		&\le \nu(P_\eps g) \\
		&\le \lambda \eps^{\delta_\Gamma +\frac{1}{2}(n-1)(n-2)+1}e^{(k_2(x,\eps)-\delta_\Gamma)d(\pi(\mathcal{C}_0), \pi(a_{\log \eps}x))},\end{align*} where $x = g\Gamma$, and $k_2(x,\eps)$ is as defined in Proposition \ref{prop:shadow lemma}.\end{lemma}
	\begin{proof}   
		Let $x = g\Gamma$. By Proposition \ref{prop:shadow lemma}, there exists $\tilde{\lambda} > 1$ such that for all such $\eps$, \begin{equation} \label{eq; ineq bd for ps-} \tilde\lambda\inv \eps^{\delta_\Gamma}e^{(k_2(x,\eps)-\delta_\Gamma)d(\pi(\mathcal{C}_0), \pi(a_{\log \eps}x))} \le \vps_g(B_{\tilde U}(\eps)) \le \tilde\lambda \eps^{\delta_\Gamma}e^{(k_2(x,\eps)-\delta_\Gamma)d(\pi(\mathcal{C}_0), \pi(a_{\log \eps}x))} \end{equation}
		
		From (\ref{eq; defn of nu}), if $m$ denotes the probability Haar measure on $M$ we then have \begin{align*} \nu(P_\eps g) &\le  \int_{A_\eps}\int_{M_\eps} \vps_g(B_{\tilde U}(\eps))dmds\\   &\le C \tilde{\lambda} \eps^{\delta_\Gamma+\frac{1}{2}(n-1)(n-2)+1}e^{(k_2(x,\eps)-\delta_\Gamma)d(\pi(\mathcal{C}_0), \pi(a_{\log \eps}x))},\end{align*}
		where $C$ is determined by the scaling of the probability Haar measures on $A$ and $M$. The lower bound follows similarly. Then, $\lambda=\max\{C \tilde{\lambda},\tilde{\lambda}\}$ satisfies the conclusion of the lemma. 
	\end{proof}

	\subsection{Admissible Boxes and Smooth Partitions of Unity}\label{section; admissible boxes}
	
	Recall that for $\eta>0$ we denoted by $G_\eta$ the closed $\eta$-neighborhood of $e$ in $G$. 
	
	$\eps>0$ such that the map $$g\mapsto gx \text{ is injective on }G_\eps \text{ for all }x\in \Omega.$$ 
	
	
	

	For $x\in X$ and $\eta_1>0$, $\eta_2\geq0$ less than $\inj(x)$, we call $$B = B_U(\eta_1)P_{\eta_2}x$$ an \emph{admissible box} (with respect to the PS measure) if $B$ is the injective image of $B_U(\eta_1)P_{\eta_2}$ in $X$ under the map $h\mapsto h x$ and $$\ps_{px}(B_U(\eta_1)px)\neq0$$ for all $p\in P_{\eta_2}$. For $g\in G$, we say that $B = B_U(\eta_1)P_{\eta_2}g$ is an admissible box if $B = B_U(\eta_1)P_{\eta_2}x$ is one.
	
	Note that if $B_U(\eta_1)P_{\eta_2}g$ is an admissible box, then there exists $\eps>0$ such that $B_U(\eta_1+\eps)P_{\eta_2 + \eps}g$ is also an admissible box. Moreover, every point has an admissible box around it by \cite[Lemma 2.17]{OhShah}.
	
	The error terms in our main theorems are in terms of Sobolev norms, which we define here. For $\ell\in\N$, $1\leq p\leq\infty$, and $\psi\in C^\infty(X)\cap L^p(X)$ we consider the following Sobolev norm\[
	S_{p,\ell}(\psi)=\sum\norm{U\psi}_p
	\]
	where the sum is taken over all monomials $U$ in a fixed basis of $\mathfrak{g}=\mbox{Lie}(G)$ of order at most $\ell$, and $\norm{\cdot}_p$ denotes the $L^p(X)$-norm. Since we will be using $S_{2,\ell}$ most often, we set $$S_\ell=S_{2,\ell}.$$ 
	
	Our proofs will require constructing smooth indicator functions and partitions of unity with controlled Sobolev norms. We prove such lemmas below.
	
	\begin{lemma}\label{lem; existence of smooth indicator}
		Let $H$ be a horospherical subgroup of $G$ (that is, $U$ or $\tilde{U}$). For every $\xi_1,\xi_2>0$ and $g \in G$, there exists a non-negative smooth function $\chi_{\xi_1,\xi_2}$ defined on $H_{\xi_1+\xi_2}g$ such that $0 \le \chi_{\xi_1,\xi_2}\le 1 $, $S_\ell(\chi_{\xi_1,\xi_2}) \ll_{n,\Gamma} \xi_1^{n-1}\xi_2^{-\ell-(n-1)/2},$ and $$\chi_{\xi_1,\xi_2}(h) = \begin{cases} 0 & \text{ if } h \not\in H_{\xi_1+\xi_2 }g \\ 1 &\text{ if } h \in H_{\xi_1-\xi_2}g\end{cases}.$$
	\end{lemma}
	
	\begin{proof}
		According to \cite[Lemma 2.4.7(b)]{KleinbockMargulis} there exists $c_1=c_1(n)>0$ such that for every $\xi>0$, there exists a non-negative smooth function $\sigma_{\xi}$ defined on $H_\xi$ such that
		\begin{equation}\label{eq:sigma def}
		\int_{H}\sigma_{\xi}(h)dm^\text{Haar}(h)=1,\quad S_\ell(\sigma_{\xi})<c_1\xi^{-\ell-(n-1)/2}.
		\end{equation}
		
		For $g\in\Omega$, let $\chi_{\xi_1,\xi_2}=\1_{H_{\xi_1} g}*\sigma_{\xi_2}$. Then for any $h\in H$, we have $0\leq\chi_{\xi_1,\xi_2}(h)\leq1$ and 
		\[
		\chi_{\xi_1,\xi_2}(h)=\begin{cases}
		0 & \text{if }h\notin H_{\xi_1+\xi_2}g\\
		1& \text{if }h\in H_{\xi_1-\xi_2}g
		\end{cases}
		\]
		
		Since for some $c_2=c_2(\Gamma)>0$ \[
		S_{1,0}(\1_{H_{\xi_1}g_0})=m^{\text{Haar}}(H_{\xi_1})<c_2{\xi_1}^{n-1},\]
		by the properties of the Sobolev norm and \eqref{eq:sigma def} we arrive at
		\[
		S_\ell(\chi_{\xi_1,\xi_2})\leq S_{1,0}(\1_{H_{\xi_1}g_0})S_\ell(\sigma_{\xi_2})<c_1 c_2{\xi_1}^{n-1}{\xi_2}^{-\ell-(n-1)/2}.
		\]
	\end{proof}
	
	\begin{lemma}\label{lem:partition of unity}
		
		Let $H$ be a horospherical subgroup of $G$, $r>0$, $\ell\in\N$, and let $E\subset H$ be bounded. Then, there exists a partition of unity $\sigma_1,\dots,\sigma_k$ of $E$ in $H_r E$, i.e. \[
		\sum_{i=1}^k \sigma_i(x)=\begin{cases}
		0 & \text{if }x\notin H_r E\\
		1 & \text{if }x\in E,
		\end{cases}\]
		such that for some $u_1,\dots,u_k\in E$ and all $1\leq i\leq k$ \[
		\sigma_i\in C_c^\infty(H_r u_i),\quad S_\ell(\sigma_i)\ll_n r^{-\ell+n-1}. \]
		Moreover, if there exists $R>r$ such that $E=H_R$, then $k\ll_n \left(\frac{R}{r}\right)^{n-1}$. 
	\end{lemma}
	
	\begin{proof}
		Let $\{u_1,\dots,u_k\}$ be a maximal $\frac{r}{4}$-separated set in $E$. Then 
		\begin{equation}\label{eq:r/2 cover}
		E\subseteq\bigcup_{i=1}^k H_{r/2}u_i. 
		\end{equation}
		
		Let $1\leq i\leq k$. According to \cite[Theorem 1.4.2]{hormander} there exists $\chi_i\in C_c^\infty\left(H_{r} u_i\right)$ such that $0\leq\chi_i\leq1$, $\chi_i(u)=1$ for any  $u\in H_{r/2}u_i$, and for $1\leq m\leq \ell$ 
		\begin{equation}\label{eq:chi_i deriv bound}
		|\chi_i^{(m)}|\ll r^{-m}
		\end{equation}
		(where the implied constant depends only on $n$).
		Let $\sigma_i$ be defined by\[
		\sigma_i=\chi_i(1-\chi_{i-1})\cdots(1-\chi_1). \]
		Then, each $\sigma_i\in C_c^\infty(H_{r} u_i)$ and \[
		1-\sum_{i=1}^k\sigma_i=\prod_{i=1}^k(1-\chi_i)=0 \mbox{ on }\bigcup_{i=1}^k H_{r}u_i\]
		implies that $\sum_{i=1}^k\sigma_i=1$ on $\bigcup_{i=1}^k H_{r/2}u_i$. 
		
		By the rules for differentiating a product and \eqref{eq:chi_i deriv bound} for $1\leq m\leq \ell$ we have \[
		|\sigma_i^{(m)}|\le C r^{-m}, \]
		where $C$ is the multiplicity of the cover in  \eqref{eq:r/2 cover}. By Besicovitch covering theorem, $C$ is bounded by a constant which depends only on $n$. 
		Using the definition of the Sobolev norm we arrive at\[
		S_\ell(\sigma_i)\ll_{n} r^{-\ell+n-1}. \]
		
		Now, assume there exists $R>r$ such that $E=H_R$. 
		Since the geometry of $H$ is of an Euclidean space of dimension $\dim H$, we then have \[
		k\ll_n\left(\frac{R}{r}\right)^{n-1}. \]
	\end{proof}
	

	\begin{lemma}\label{lem:choosing ell}
		Let $H$ be either $U$ or $G$.
		There exists $\ell'=\ell'(H)>0$ such that for any integer $\ell>\ell'$, $\eta>0$, $H\in\{U,G\}$, and $f\in C_c^\infty(H)$, there exist functions $f_{\eta,\pm}\in C_c^\infty(H)$ which are supported on an $2\eta$ neighborhood of $\supp f$, and for any $h\in H$ satisfy 
	\begin{enumerate}
		\item $f_{\eta,-}(h)\le \min_{w\in H_{\eta}}f(wh) \le \max_{w\in H_{\eta}}f(wh) \le f_{\eta,+}(h)$
		\item $\left|f_{\eta,\pm}(h)-f(h)\right|\ll_{\supp f} \eta S_{\ell}(f)$
		\item $S_{\ell}(f_{\eta,\pm})\ll_{H,\supp f}\eta^{-2\ell} S_{\ell}(f).$
	\end{enumerate}
	\end{lemma}

	\begin{proof}
		First, according to \cite{sobolev}, there exists $\ell'\in\N$ such that any $\ell>\ell'$ satisfies  $S_{\infty,1}(\psi)\ll_{\supp \psi} S_{\ell}(\psi)$ for any $\psi\in C_c^\infty(H)$.
		
		Let $f'_{\eta,\pm}$ be defined by \[
		f'_{\eta,+}(h):=\sup_{w\in H_{\eta}}f(wh)\text{ and }	f'_{\eta,-}(h):=\inf_{w\in H_{\eta}}f(wh)\] 
		for any $h\in H$.
		
		As before, we use \cite[Lemma 2.4.7(b)]{KleinbockMargulis} to deduce that there exist $c_1=c_1(H)>0$, $n_1=n_1(H)$ and a non-negative smooth function $\sigma_{\eta}$ supported on $H_\eta$  such that \[
		\int_{H}\sigma_{\eta}(h)dm^{\operatorname{Haar}}(h)=1,\quad S_\ell(\sigma_{\eta})<c_1\eta^{-\ell-n_1}. \]
		
		Define $f_{\eta,\pm}$ by\[
		f_{\eta,\pm}:=f'_{2\eta,\pm}*\sigma_{\eta}. \]
		Then, $f_{\eta,\pm}$ are smooth functions which are supported on an $2\eta$ neighborhood of $\supp f$.  
		Moreover, for any for any $h\in H$  
		\begin{align}
		f'_{\eta,+}(h) & =  \int_{H_\eta}f'_{\eta,+}(h)\sigma_\eta(u^{-1})dm^{\operatorname{Haar}}(u)\nonumber\\
		& \le \int_{H_\eta}f'_{2\eta,+}(uh)\sigma_\eta(u^{-1})dm^{\operatorname{Haar}}(u) &\text{ by definition of } f'_{2\eta,+}\label{eq:fprime2eta line}\\
		& = f_{\eta,+}(h)\nonumber\\
		& \le \int_{H_\eta}f'_{3\eta,+}(h)\sigma_\eta(u^{-1})dm^{\operatorname{Haar}}(u) &\text{by \eqref{eq:fprime2eta line} and definition of }f'_{3\eta,+}\nonumber\\
		& = f'_{3\eta,+}(h).\nonumber
		\end{align}
		In a similar way, one can show\[  
		f'_{3\eta,-}\le f_{\eta,-} \le f'_{\eta,-}, \]
		proving the first inequality.
		
		By the mean value theorem, for any $h\in H$, $w\in H_{3\eta}$ \[ \left|f(wh)-f(h)\right|\ll \eta S_{\infty,1}(f)\ll_{\supp f}S_\ell(f).\]
		Since $f'_{3\eta,-}\le f_{\eta,-} \le f_{\eta,+} \le f'_{3\eta,+},$ there exist some $w_{+},w_{-}\in H_{3\eta}$ such that \[
		\left|f_{\eta,\pm}(h)-f(h)\right|\le \left|f(w_\pm h)-f(h)\right|,\]
		and we have the second inequality. 
		
		Now, we have\[
		S_{\ell}(f_{\eta,\pm})\le S_{\infty,1}(f'_{2\eta,\pm})S_{\ell}(\sigma_{\eta})\ll_{H,\supp f}S_{\ell}(f)\eta^{-\ell-n_1+1}.\]
		By choosing $\ell'>n_1$, we may deduce the last inequality.  
	\end{proof}

	\section{Quantitative Nondivergence}\label{sec:quantitative nondivergence}
	
	In this section, we prove a quantitative nondivergence result that is crucial in the following sections. We use the notation established in \S\ref{section: thick thin decomposition}. The results in this section hold for any $\Gamma$ that is geometrically finite, without need for Assumption \ref{thm:effective mixing}.
	
	Recall from \S\ref{section; intro} that for $0<\eps<1$ and $s_0\ge 1$, we say that $x\in X$ is \emph{$(\eps,s_0)$-Diophantine} if for all $\tau>s_0$, 
	\begin{equation}\label{eq:diophantine condition}
	d(\mathcal{C}_0,a_{-\tau}x)<(1-\eps)\tau,
	\end{equation} where $\mathcal{C}_0$ is the compact set defined in \S\ref{section: thick thin decomposition}. Let $R_0$ and $q$ also be as defined in \S\ref{section: thick thin decomposition}.
	
	This section is dedicated to the proof of the following theorem, which says (in a quantitative way) that most of the $U$ orbit of a Diophantine point is not in the cusp:
	
	\begin{theorem}
		\label{thm:bounded functions}
		There exists $\beta>0$ satisfying the following: for every \linebreak $0<\eps<1$ and $s_0\ge 1$, and for every $(\eps,s_0)$-Diophantine element $x\in X$, every $R\ge R_0$, every $T\gg_{\Gamma,\eps}s_0$, and every $0<s\le T^{\eps}$, we have \[ \ps_{a_{-\log s}x}\left(B_U(T/s)a_{-\log s}x\cap{\mathcal{X}}(R)\right)\ll_{n,\Gamma} \ps_{a_{-\log s}x}(B_U(T/s)a_{-\log s}x)e^{-\beta R}.\]
	\end{theorem}
	
	We now follow the notation of Mohammadi and Oh in \cite[\S 6]{isolations}. Equip $\mathbb{R}^{n+1}$ with the  Euclidean norm. Recall from \S\ref{section: thick thin decomposition} that for $1\le i\le q$, $g_i^-=\xi_i$. Without loss of generality, we may further assume that $g_i$ satisfies \linebreak $\norm{g_i^{-1}e_1}=~1.$ Let $$v_i=g_i^{-1}e_1.$$ 
	
	\begin{lemma}\label{lem:isolations} 
	For any $i=1,\dots,q$, $\Gamma v_i$ is a discrete subset of $\mathbb{R}^{n+1}$. 
	\end{lemma}
	
	\begin{proof} 
	    Since $\xi_i$ is assumed to be a bounded parabolic limit point, by definition $(\Lambda(\Gamma)\setminus\{\xi_i\})/ \Gamma_{\xi_i}=(\Lambda(\Gamma)\setminus\{\xi_i\})/ \Gamma_{v_i}$ is compact, where
	    \[G_{v_i}=g_i^{-1}MUg_i\quad\text{ and }\quad\Gamma_{v_i}=\Gamma\cap G_{v_i}.\]

	    If $\gamma\in\Gamma_{v_i}$, then $\mathcal{H}_i(R_0)\gamma=\mathcal{H}_i(R_0)$. 
	    Therefore, the visual map induces a homeomorphism between $\mathcal{H}_i(R_0)/\Gamma_{v_i}$ and $(\partial\H^n\setminus\{\xi_i\})/\Gamma_{v_i}$. It follows that the quotient of $\{g^+\in\Lambda\::\:g\in\mathcal{H}_i(R_0)\}$ by the action of $\Gamma_{v_i}$ is compact. Using Iwasawa decomposition, it follows that there exists a compact set $U_0\subset U$ such that for any $g=kaug_i\in\mathcal{H}_i(R_0)$ such that $g^+\in\Lambda(\Gamma)$, $k\in K$, $a\in A$, and $u\in U$, there exist $\gamma\in\Gamma_{v_i}$, $k'\in K$, $u'\in U_0$ so that $g\gamma=k'au'g_i$.


	Since $\xi_i$ is assumed to be a parabolic limit point, there exists a parabolic element $\gamma_0\in\Gamma_{\xi_i}$, i.e. $\gamma_0=g_i^{-1}mug_i$.
	
    Assume by contradiction that there exists an infinite sequence $\{\gamma_j\}\in\Gamma$ such that $\{\gamma_j v_i\}$ converges to $0$. 
    Using the Iwasawa decomposition we get that for all $j$ there exist $a_{t_j}\in A$, $k_j\in K$, and uniformly bounded $u_j\in U$ such that $\gamma_j=k_ja_{t_j}u_jg_i$. Since
    \begin{equation*}
        \norm{\gamma_jv_i}=\norm{k_ja_{t_j}u_je_1}=e^{t_j},
    \end{equation*}
    we may deduce that $t_j\rightarrow-\infty$. In particular, $\gamma_j\in H_i(R_0)$ for all large enough $j$. 
    
    We have
    \begin{align*}
        \gamma_j\gamma_0\gamma_j\inv
        &=(k_ja_{t_j}u_jg_i)(g_i^{-1}mug_i)(g_i\inv u_j\inv a_{t_j}\inv k_{j}\inv)\\
        &=k_ja_{t_j}u_jm uu_j\inv  a_{t_j}\inv k_j\inv. 
    \end{align*}
	Since $u_jmu_j\inv=m_ju'_j\in MU$, with $u'_j$ uniformly bounded, and since $M$ centralizes $A$, we have 
	\begin{align*}
        \gamma_j\gamma_0\gamma_j\inv
        &=k_jm_j a_{t_j} u'_j u a_{t_j}\inv k_j\inv. 
    \end{align*}
	Since $u'_j u$ is in a bounded subset of $U$, we get that $a_{t_j} u'_j u a_{t_j}\inv\rightarrow e$ as $t_j\rightarrow~-\infty$. 
	Since $K$ and $M$ are compact, it then follows that the sequence $\gamma_j\gamma_0\gamma_j\inv$ has a convergent subsequence. This contradicts the discreteness of $\Gamma$, since the $\gamma_j$'s were assumed to be distinct.

	
		\end{proof}
	
	For any $g\in G$, we have that $g\Gamma\in\mathcal{X}_i(R)$, if and only if there exists $\gamma \in \Gamma$ such that
	\begin{equation}\label{eq:v_i cusp cond}
	\norm{g\gamma v_i}\le e^{-R}.
	\end{equation}
	Indeed, by the Iwasawa decomposition and \eqref{eq:siegel sets}, if $g\Gamma\in\mathcal{X}_i(R)$, then there exist $\gamma\in\Gamma$, $k\in K$, $s>R$, and $u\in U$, such that \[
	\norm{g\gamma v_i}=\norm{ka_{-s}ug_iv_i}=\norm{a_{-s}e_i}=e^{-s}. \]
	Moreover, it follows from \cite[Lemma 6.4, Lemma 6.5]{isolations} that the $\gamma$ in \eqref{eq:v_i cusp cond} is unique. Note that both lemmas are proved under the additional assumption that $n=3$, but the proofs also hold without it.
	
	On the other hand, by \cite[Lemma 6.5]{isolations} and Lemma \ref{lem:isolations}, there exists a constant $\eta_0 = \eta_0(\Gamma)>0$ such that if $g\Gamma\notin\mathcal{X}_i(R_0)$, then for any $\gamma\in\Gamma$, 
	\begin{equation}\label{eq:v_i cusp 2nd cond}
	\norm{g\gamma v_i}>\eta_0.
	\end{equation}
	
	
	
	
	\begin{lemma}\label{lem:diophantine condition}
		There exists $c=c(\Gamma)>0$ which satisfies the following. Let $\eps, s_0>0$ and let $g \in G$. If $x = g\Gamma$ is $(\eps,s_0)$-Diophantine, then for any $T\gg_{\Gamma,\eps} s_0$,
		\begin{equation}\label{eq:cusp bound}
		\sup_{\norm{\t}\leq T}\inf_{\gamma\in\Gamma}\inf_{i=1,\dots,q}\norm{u_\t g\gamma v_i}>cT^{\eps}.
		\end{equation}
	\end{lemma}
	
	\begin{proof}
		Fix $T>T_0=\max\left\{s_0,\eta_0^{\frac{1}{\eps-1}}\right\}$. 
		We will first show that
		\begin{equation}\label{eq:v_i first bound}
		\inf_{\gamma\in\Gamma}\inf_{i=1,\dots,q}\norm{a_{-\log T}g\gamma v_i}>cT^{\eps-1},
		\end{equation}
		for some constant $1>c=c(\Gamma)>0$. 
		
		There are two cases to consider. If $a_{-\log T}x\notin\mathcal{X}_i(R_0)$, then \eqref{eq:v_i first bound} follows from \eqref{eq:v_i cusp 2nd cond} and the choice of $T$.
		
		Otherwise, $a_{-\log T}x\in\mathcal{X}_i(R)$ for some maximal $R>R_0$. According to Lemma \ref{lem: height and distance to C0}, we have $$d(x,\mathcal{C}_0)\ge R-R_0.$$
		Then, because $x$ is $(\eps,s_0)$ Diophantine and $T > s_0,$ by \eqref{eq:diophantine condition}, we may deduce  that $$R-R_0<(1-\eps)\log T.$$ Hence, $a_{-\log T}x \not\in \mathcal{X}_i((1-\eps)\log T+R_0),$ so \eqref{eq:v_i cusp cond} implies \eqref{eq:v_i first bound}.
		
		Now, fix $\gamma\in\Gamma$ and $1\le i \le q$, and let \[
		\begin{pmatrix}
		x_1\\
		\vdots\\
		x_{n+1}		\end{pmatrix}=a_{-\log T}g\gamma v_i.\]
		According to \eqref{eq:v_i first bound}, there exists $1\leq k\leq n$ such that $|x_k|>cT^{\eps-1}$. If $|x_1|>cT^{\eps-1}$, then it follows from the action of $a_{-\log T}$ on $\mathbb{R}^{n+1}$ that \[
		\norm{g\gamma v_i}\geq|cT x_1|>cT^{\eps}. \]
		
		Otherwise, there exists $2\leq k\leq n$ such that $|x_k|>cT^{\eps-1}$. Then, for any $\t\in\mathbb{R}^{n-1}$, the first coordinate of $u_{\t}a_{-\log T}g\gamma v_i$ is \[
		x_1+\t\cdot \mathbf{x}'+\frac{1}{2}\norm{\t}^2 x_{n+1},\quad\mbox{where}\quad \mathbf{x}'=\begin{pmatrix}
		x_2\\
		\vdots\\
		x_n
		\end{pmatrix}.\]
		In particular, by taking $t_k=\pm T$ (the $k$-th entry in $\t$) one can ensure that $\|a_{\log T}u_{\t}a_{-\log T}g\gamma v_i\|>cT^{\eps}$. 
	\end{proof}
	
	A measure $\mu$ is called \emph{$D$-Federer} if for all $v\in\supp(\mu)$ and $0 <\eta\leq1$,\[	\mu(B(v,{3\eta}))\le D\mu(B(v,\eta)). \]
	It is proved in the appendix, \S\ref{friendliness appendix}, (specifically Corollary \ref{cor:ps measure is doubling}) that there exists $D=D(\Gamma)>0$ such that:
	\begin{equation}\label{eq:ps is federer}
	    \text{if }x\in X\text{ satisfies }x^-\in\Lambda(\Gamma)\text{, then }\ps_x\text{ is }D\text{-Federer}.
	\end{equation}
	
	For $f:\mathbb{R}^d\rightarrow\mathbb{R}$ and $B\subset\mathbb{R}$, let\[
	\norm{f}_B:=\sup_{x\in B}|f(x)|.\]
	
	
    Recall that $U \cong \R^{n-1}$.
	
	\begin{lemma}\label{lem: KLW 7.5}
	Let $y \in \supp\bms$ and let $f: B_U(\eta)\to\R^{n-1}$ be such that there exists $b\ne 0$ so that for every coordinate function $f_i: B_U(\eta) \to \R$, there exist $a_i\in \R$, such that \[f_i(\t) = a_i + b\t_i.\] Then for $0<\eta\le 1$ and $0<\eps<1,$ we have 
	\begin{equation}\label{eq:(c,alpha) good}
	    \ps_y\left(\left\{\t\in B_U(\eta):\|f(\t)\|<\eps\right\}\right)\ll_{\Gamma}\left(\frac{\eps}{\norm{f}_{B_U(\eta)}}\right)^\sigma\ps_y(B_U(\eta)),
	\end{equation} where $\|f(x)\|$ denotes the max norm. 
	\end{lemma}
	
	\begin{proof}
	First, note that if $\|f\|_{B_U(\eta)} < 2\eps$, then the result holds by assuming that the implied coefficient in \eqref{eq:(c,alpha) good} is bigger than $2^\sigma$: in this case, the right hand side is greater or equal to
\begin{align*}2^{\sigma}\left(\frac{\eps}{\norm{f}_{B_U(\eta)}}\right)^\sigma\ps_y(B_U(\eta))	&\ge2^{\sigma}\left(\frac{\eps}{2\eps}\right)^\sigma\ps_y(B_U(\eta))\\	& \ge\ps_y\left(\left\{\t\in B_U(\eta):\|f(\t)\|<\eps\right\}\right),	\end{align*} as desired.
Thus, we now assume that \begin{equation}\label{eq: norm f ge 2 eps}\norm{f}_{B_U(\eta)} \ge 2\eps.\end{equation}
	
	If $\|f(\t)\|\ge \eps$ for all $\t \in B_U(\eta)$ such that $(u_\t y)^+\notin\Lambda(\Gamma)$, then there is nothing to prove. So assume that $\|f(\t)\|<\eps$ and $(u_\t y)^+\in\Lambda(\Gamma)$. Since each $f_i$ is linear, for all $\t'\in B_U(\eta)$ with $\|f(\t')\|<\eps$ we get that for all $1\le i\le n-1$, \begin{align*}
	   |f_i(\t')|= |a_i+b\t'| <\eps
	\end{align*} \[|b(t'_i-t_i)|=|f_i(\t')-f_i(\t)|<2\eps.\] 
	Therefore, \[\|f(x)\|<\eps \implies x \in B_U(2\eps/b)z.\]
	
	Thus, by \eqref{eq:ps is federer}, we have that there exists $\sigma = \sigma(\Gamma)>0$ so that \begin{align*}
	    \ps_y(\{x \in B_U(\eta)y : \|f(x)\|<\eps\})&\le\ps_z( B_U(2\eps/b))\\
	    & \ll_\Gamma\left(\frac{2\eps}{b\eta}\right)^\sigma\ps_z( B_U(\eta))\\
	    & \ll_\Gamma\left(\frac{2\eps}{b\eta}\right)^\sigma\ps_y( B_U(3\eta))\\
	    & \ll_\Gamma\left(\frac{6\eps}{b\eta}\right)^\sigma\ps_y( B_U(\eta))
	\end{align*}
	
	Assuming $\|f(\t)\|<\eps$ for some $\t\in B_U(\eta)$ (otherwise, as before, there is nothing to prove), for any $\t''\in B_U(\eta)$ and $1\le i\le n-1$ we have \[
	|f_i(\t'')|\le |f_i(\t'')-f_i(\t)|+|f_i(\t)|<2b\eta+\eps.\]
	Thus, $\norm{f}_{B_U(\eta)y}-\eps\le 2b\eta$, so by \eqref{eq: norm f ge 2 eps}, $$\frac{1}{2}\norm{f}_{B_U(\eta)} \le 2b \eta,$$ which completes the proof.
	\end{proof}

	A function $f$ which satisfies \eqref{eq:(c,alpha) good} with the implied constant $C$ for any $\eps>0$ and any ball $B\subset U\subset\mathbb{R}^m$, is called \emph{$(C, \sigma)$-good on $U$ with respect to $\mu$}. Observe that \begin{equation}
	    \label{eq: bounded below by C alpha good} \text{if }g\text{ is }(C,\sigma)\text{-good and if } |g(x)|\le |f(x)|\text{ for }\mu\text{-a.e. }x\text{, then }f\text{ is }(C,\sigma)\text{-good.}
	\end{equation} 
	
	In the proof of the following theorem we use similar ideas to the ones which appear in the proof of \cite[Lemma 5.2]{KLW}. Note that the proof in this case is simplified by the third assumption, reflecting our rank one setting.

    \begin{proposition}\label{prop:(C,alpha)-good}
		Given positive constants $C,\beta, D$, and $0<\eta<1$, there exists $C'=C'(C,\beta, D)>0$ with the following property. Suppose $\mu$ is a $D$-Federer measure on $\R^m$,  $f:\R^m\rightarrow\operatorname{SL}_k(\R)$ is a continuous map, $0\le\varrho\le\eta$, $z\in\supp\mu$, $\Lambda\subset\R^k$, $B=B(z,r_0)\subset\R^m$, and $\tilde B=B(z,3 r_0)$ satisfy: 
		\begin{enumerate}
			\item For any $v\in\Lambda$, the function $\t\mapsto \norm{f(\t)v}$ is $(C,\beta)$-good on $\tilde B$ with respect to $\mu$. 
			\item For any $v\in\Lambda$, there exists $\t\in B$ such that  $\norm{f(\t)v}\geq\varrho$. 
			\item For any $\t\in B$, there is at most one $v\in\Lambda$ which satisfies $\norm{f(\t)v}<\eta$. 
		\end{enumerate}
		Then, for any $0<\eps<\varrho$,\[
		\mu\left(\left\{\t\in B\::\:\exists v\in\Lambda\text{ such that }\norm{f(\t)v}<\eps\right\}\right)\le C'\left(\frac{\eps}{\varrho}\right)^\beta\mu(B). \]
	\end{proposition}
	
	\begin{proof}
		For any $\t\in B$, denote \[
		f_\Lambda(\t)=\min\left\{\norm{f(\t)v}\::\:v\in\Lambda\right\}. \]
		Let \[
		E=\left\{\t\in B\::\: f_\Lambda(\t)<\varrho\right\}\cap \supp\mu, \]
		and for each $v \in \Lambda$, define \[E_v=\{\t \in B\:: \|f(\t)v\|<\varrho\} \cap \supp\mu.\] Observe that by assumption (3), the $E_v$'s are a disjoint cover of $E$. For each $t \in E_v$, define  \[r_{\t,v}=\sup\{r\::\:\|f(\mathbf{s})v\|<\varrho \text{ for all } \s\in B(\t,r_{\t,v})\}. \] 
		
		By assumption (2), we know that for every $\t\in E$, the set $B(\t,r_{\t,v})$ does not contain $B$. Thus, since $\t\in B$, we deduce that $r_{\t,v}<2r_0$. For any fixed $r_{\t,v}<r'_{\t,v}<2r_0$, we have that 
		\begin{equation}\label{eq: I is in B tilde}
		B(\t,r'_{\t,v})\subset B(z,3r_0)=\tilde B,
		\end{equation} 
		and by the definition of $r_{\t,v}$, there exists $\s \in B(\t,r'_{\t,v})$ such that \[\norm{f(\s)v}\geq\varrho.\]
		
		Note that $\left\{B(\t,r_{\t,v})\::\t\in E, v\in\Lambda\right\}$ is a cover of $E$. According to the Besicovitch covering theorem, there exists a countable subset $I\subset E\times \Lambda$ such that $\left\{B(\t,r_{\t,v})\::(\t,v)\in I\right\}$ is a cover of $E$ with a covering number bounded by a constant which only depends on $m$. Thus,
		\begin{equation}\label{eq:Besicovitch bound on balls}
		\sum_{(\t,v)\in I}\mu\left(B(\t,r_{\t,v})\right)
		\ll_{m} 
		\mu\left(\bigcup_{(\t,v)\in I}B(\t,r_{\t,v})\right).
		\end{equation}
		
		By assumption (3) and the continuity of $f$, for any $(\t,v)\in I$ and $\bold{s}\in E\cap B(\t,r_{\t,v})$, $$f_\Lambda(\s)=\norm{f(\s)v}.$$ Thus,
		\begin{align*}
		\mu\left(\left\{\bold{s}\in B(\t,r_{\t,v})\::\:{f_{\Lambda}(\bold{s})}<\eps\right\}\right)&=\mu\left(\left\{\bold{s}\in B(\t,r_{\t,v})\::\:\norm{f(\bold{s})v}<\eps\right\}\right)\\
		&\le\mu\left(\left\{\bold{s}\in B(\t,r'_{\t,v})\::\:\norm{f(\bold{s})v}<\eps\right\}\right).\end{align*}
		Thus, assumption (1) and the assumption that $\mu$ is $D$-Federer together imply that
		\begin{align}
		\mu\left(\left\{\bold{s}\in B(\t,r_{\t,v})\::\:{f_{\Lambda}(\bold{s})}<\eps\right\}\right)&\le\mu\left(\left\{\bold{s}\in B(\t,r'_{\t,v})\::\:\norm{f(\bold{s})v}<\eps\right\}\right)\nonumber\\
		& \le C\left(\frac{\eps}{\varrho}\right)^\beta\mu(B(\t,r'_{\t,v}))\nonumber\\
		& \le CD \left(\frac{\eps}{\varrho}\right)^\beta\mu(B(\t,r_{\t,v}))\label{eq: h(s) good bound}.
		\end{align}
		
		Since $E$ covers the set of points for which $f_{\Lambda}$ is less than $\eps$, we may now conclude 
		\begin{align*}
		&\mu\left(\left\{\t\in B\::\: f_{\Lambda}(\t)<\eps\right\}\right)\\ 
		&\le \sum_{(\t,v)\in I}\mu\left(\left\{\bold{s}\in B(\t,r_{\t,v})\::\:f_{\Lambda}(\t)<\eps\right\}\right) & \\
		& \le CD \sum_{(\t,v)\in I}\left(\frac{\eps}{\varrho}\right)^\beta\mu(B(\t,r_{\t,v})) & \text{by \eqref{eq: h(s) good bound}}\\
		& \ll_{m} CD \left(\frac{\eps}{\varrho}\right)^\beta\mu\left(\bigcup_{(\t,v)\in I}B(\t,r_{\t,v})\right)& \text{by \eqref{eq:Besicovitch bound on balls}}\\
		& \ll_{m} CD \left(\frac{\eps}{\varrho}\right)^\beta\mu\left(\tilde B \right) & \text{by \eqref{eq: I is in B tilde}}\\
		& \ll_{m} CD^2 \left(\frac{\eps}{\varrho}\right)^\beta\mu\left(B \right) &\mu\text{ is }D\text{-Federer.}
		\end{align*}
	\end{proof}


	
	\begin{remark}
	Fix $x\in X$ such that $x^-\in\Lambda(\Gamma)$. Since the PS-measure $\ps_x$ is supported on $Ux\cap\supp\bms$, it follows from Lemma \ref{lem: KLW 7.5} and \eqref{eq:ps is federer} that Proposition \ref{prop:(C,alpha)-good} holds for $\ps_x$ and function $f$ which satisfies the assumption of Lemma \ref{lem: KLW 7.5}. 
	\end{remark}
	
	We are now ready to prove Theorem \ref{thm:bounded functions}.

	\begin{proof}[Proof of Theorem \ref{thm:bounded functions}]
		Let $x_0=a_{-\log s}x$, and fix $g\in G$ such that $x=g\Gamma$. By Lemma \ref{lem:diophantine condition}, for all $T \gg_{\Gamma,\eps} s_0,$ we have \eqref{eq:cusp bound}, that is, that \[\sup_{\norm{\t}\leq T}\inf_{\gamma\in\Gamma}\inf_{i=1,\dots,q}\norm{u_\t g\gamma v_i}>T^{\eps}.\]
		
		
		Let $f:\R^{n-1}\rightarrow\operatorname{SL}_{n+1}(\R)$ be defined by $$f(\t)=u_\t a_{-\log s}g.$$ 
		We first show that parts (1), (2), and (3) of Proposition \ref{prop:(C,alpha)-good} for $\mu=\ps_{x_0}$, $f$, $\varrho=1$, $z=x_0$, $r=T/s$, $\eta=e^{-R_0}$, and $$\Lambda=\Gamma \left\{v_1,\dots,v_q\right\}.$$ 
		Note that $0\notin\Lambda$. 
		
		
		It follows from the action of $u_\t$ on $\mathbb{R}^{n+1}$ that for any $v\in\R^{n+1}$ there exists $v'=(v'_1,\dots,v'_{n+1})^T\in\R^{n+1}$ such that \begin{equation}\label{eq: f(t)v eqn in nondiv}		f(\t)v=\left(v'_1+\t\cdot v''+\frac{1}{2}\norm{\t}^2v'_{n+1},v'_2-t_1v'_{n+1},\dots,v'_{n}-t_{n-1}v'_{n+1},v'_{n+1}\right)^T,\end{equation}
        where $v''=(v'_2,\dots,v'_{n})^T$. Thus, if $v'_{n+1}\ne 0$, then	$\t \mapsto f(\t)v$ is bounded from below by a function which satisfies the assumption of Lemma \ref{lem: KLW 7.5}. 
		Therefore, by \eqref{eq: bounded below by C alpha good}, 
		for any $v\in\Lambda$ the function $\t\mapsto \norm{f(\t)v}$ is $(C,\beta)$-good on $\tilde B$ with respect to $\ps_{x_0}$, for some $C=C(\Gamma)\ge 1$, $\beta=\beta(\Gamma)>0$, which proves (1) of Proposition \ref{prop:(C,alpha)-good}. Note that these constants are uniform across all $v$ so that $v_{n+1}'\ne 0.$
		
		On the other hand, if $v\neq 0$ and $$v_{n+1}' = 0,$$ then $\t \mapsto f(\t)v$ is bounded below by some positive constant, and since positive constant functions are $(C,\beta)$-good for any $C\ge 1,\beta>0,$ we conclude that so is this function by \eqref{eq: bounded below by C alpha good}. 
		
		
		By \eqref{eqn; reln bt a and u}, we have\[
		u_{\t}a_{-\log s}=a_{-\log s} u_{s\t}.\]
		Since multiplication by $a_{-\log s}$ only changes the matrix entries by scaling, using \eqref{eq:cusp bound}, for $i=1,\dots,q$ we get \[
		\sup_{\norm{\t}\leq T/s}\norm{a_{-\log s}u_{s\t} g\gamma v_i}>s^{-1}\sup_{\norm{\t}\leq T}\norm{u_\t g\gamma v_i}>s^{-1}T^{\eps}.\]
		Thus, for any $s\le T^{\eps}$, $\t<T/s$, and $v\in\Lambda$,\[
		\norm{f(\t)v}\ge 1, \]
		which establishes (2) of Proposition \ref{prop:(C,alpha)-good}.
		
		Since $\eta=e^{-R_0}$ and  $\mathcal{H}_i(R_0)$'s are pairwise disjoint, part (3) of Proposition \ref{prop:(C,alpha)-good} follows from the uniqueness of $\gamma$ in \eqref{eq:v_i cusp cond} and \eqref{eq:v_i cusp 2nd cond}. 
		
		According to \ref{eq:ps is federer}, the measure $\ps_{x_0}$ is $D$-Federer for any $D>0$. 
		Thus, we may now use \eqref{eq:v_i cusp cond} and Proposition \ref{prop:(C,alpha)-good} to deduce
		\begin{align*}
		&\ps_{x_0}\left(B_U(T/s)\cap\calh_i(R)\right)\\ &=\ps_{x_0}\left(\left\{\t\in B_U(T/s)\::\:\exists \gamma\in\Gamma,1\le i\le q\text{ such that }\norm{f(\t)\gamma v_i}<e^{-R}\right\}\right)\\
		&\ll e^{-R \beta} \ps_{x_0}(B_U(T/s)x_0),
		\end{align*}
		where the implied constant depends on $n$ and $\Gamma$. 
	\end{proof}

	\section{Friendliness Properties of the PS-measure}\label{subsection; friendly measures}
	
	In this section we prove several key properties of the PS-measure, including that slightly enlarging a ball does not increase the measure too much and that scaling the size of the ball has a bounded multiplicative increase on the measure. Note that the results in this section hold for any $\Gamma$ that is geometrically finite; we do not require Assumption \ref{thm:effective mixing}. In the setting that all cusps have maximal rank, stronger statements hold. See the appendix, specifically \S\ref{section: appendix}, for more details.
	
	The main results in this section are the following, which both establish control over the measure of a slightly enlarged ball.	 
	Many technical details of the proofs are hidden in Proposition \ref{prop: friendliness intermed step with planes}, which is proved in the appendix, \S\ref{friendliness appendix}.
	
		\begin{theorem}\label{thm: nicer friendly}
   There exists a constant $\alpha' = \alpha'(\Gamma)>0,$ such that for every $x \in G/\Gamma$ that is $(\eps,s_0)$-Diophantine, for every $0<s\le T^{\frac{\eps}{1-\eps}}$, every $0<\xi\ll_\Gamma 1,$ and every $T\gg_{\Gamma,\eps} s_0$,
	\[	\frac{\ps_{a_{-\log s}x}(B_U(\xi+T) )}{\ps_{a_{-\log s}x}(B_U(T))}-1\ll_\Gamma \xi^{\alpha'}.\]
	\end{theorem}

	
	\begin{theorem}\label{prop: nu new friendly} There exist $\alpha'=\alpha'(\Gamma)>0$, $\theta'=\theta'(\Gamma)\ge\alpha'$, $\omega'=\omega'(\Gamma)\ge 2\delta_\Gamma$, such that for any $g \in G$ with $g^- \in \Lambda(\Gamma)$ and $0<\xi<\eta \ll_\Gamma e^{-\height(g\Gamma)},$
	we have that \[\frac{\nu(P_{\xi+\eta}g)}{\nu(P_\eta g)}-1\ll_{\Gamma} e^{\omega'\height(g\Gamma)}\frac{\xi^{\alpha'}}{\eta^{\theta'}}.\]
	\end{theorem}
	
	Theorem \ref{prop: nu new friendly} will be obtained as a corollary of the following:
	
	\begin{proposition} 
	    There exist constants $\alpha = \alpha(\Gamma)>0,$ $\theta=\theta(\Gamma)\ge \alpha$, and $\omega=\omega(\Gamma)\ge 2\delta_\Gamma$ such for $x \in G/\Gamma$ which satisfies $x^+\in \Lambda(\Gamma)$, and $0<\xi< \eta \ll_\Gamma e^{-\height(x)}$ we have
	   \[\frac{\ps_x(B_U(\xi+\eta))}{\ps_x(B_U(\eta))}-1 \ll_\Gamma e^{\omega\height(x)}\frac{\xi^\alpha}{\eta^\theta}.\]\label{prop: new friendliness}
	\end{proposition}
	
	We first show how to obtain Theorem \ref{prop: nu new friendly} from Proposition \ref{prop: new friendliness}.

	\begin{proof}[Proof of Theorem \ref{prop: nu new friendly} assuming Proposition \ref{prop: new friendliness}]
	    Using the product structure of $\nu$, we can write $$\nu(P_\eta g) = \int_{A_\eta}\int_{M_\eta} \vps_g(B_{\tilde{U}} (\eta))dmds.$$
	    Then, by an analogous statement to Proposition \ref{prop: new friendliness} for $\vps$, there exists a constant $c_0 = c_0(\Gamma)>0$ such that 
	    \begin{align*}
	       \nu(P_{\eta+\xi}g) &= \int_{A_{\xi+\eta}}\int_{M_{\xi+\eta}}\vps_g(B_{\tilde U}(\xi + \eta)) dm ds \\
	       &\le\int_{A_{\xi+\eta}}\int_{M_{\xi+\eta}}\vps_g(B_{\tilde U}(\eta))\left[1+c_0\frac{\xi^\alpha}{\eta^\theta} e^{\omega'\height(g\Gamma)}\right] dm ds\\
	       &=\left[1+c_0\frac{\xi^\alpha}{\eta^\theta} e^{\omega'\height(g\Gamma)}\right]\left[ \frac{(\xi+\eta)^{\frac{1}{2}(n-1)(n-2)+1}}{\eta^{\frac{1}{2}(n-1)(n-2)+1}}\nu(P_\eta g)\right].\\
	       &\le\left[1+c_0\frac{\xi^\alpha}{\eta^\theta} e^{\omega'\height(g\Gamma)}\right]\left[1+c_1\frac{\xi}{\eta}\right]\nu(P_\eta g),
	    \end{align*} where $c_1>0$ is an absolute constant (which depends only on $n$) arising from the binomial theorem.
	    Therefore,
	    \begin{equation}\label{eq: bound on the ratio}
	        \frac{\nu(P_{\xi+\eta}g)}{\nu(P_\eta g)}-1\ll_{\Gamma} e^{\omega'\height(g\Gamma)}\frac{\xi^\alpha}{\eta^\theta}\cdot\frac{\xi}{\eta}+\frac{\xi}{\eta}+e^{\omega'\height(g\Gamma)}\frac{\xi^\alpha}{\eta^\theta}.
	    \end{equation}
	    
	    Since $\xi<\eta$, the first term on the left hand side of \eqref{eq: bound on the ratio} is dominated by the last term, and so \[\frac{\nu(P_{\xi+\eta}g)}{\nu(P_\eta g)}-1\ll_{\Gamma}\frac{\xi}{\eta}+e^{\omega'\height(g\Gamma)}\frac{\xi}{\eta^\theta}.\] Since $e^{\omega\height(g\Gamma)}\ge 1$, if we define $$\alpha'=\min\{1,\alpha\},\quad \theta'=\max\{1,\theta\},$$ then both terms are dominated by $$e^{\omega\height(g\Gamma)}\frac{\xi^{\alpha'}}{\eta^{\theta'}},$$ which completes the proof.
	\end{proof}

The following result, showing that the PS measure is not concentrated near hyperplanes, is proved in the appendix to improve the readability of this section. See Proposition \ref{prop: PS points friendliness intermed step with planes} for the proof. This result builds upon the work of Das, Fishman, Simmons, and Urba\'{n}ski in \cite{friendly}, where it is shown that the PS density $\nu_o$ is \emph{friendly} when $\Gamma$ is geometrically finite.

For a hyperplane $L \subset U \cong \R^{n-1}$ and $\xi>0$, define $$\mathcal{N}_U(L,\xi):=\{u_\t y: y \in L, \t \in B_U(\xi)\}.$$

\begin{proposition}\label{prop: friendliness intermed step with planes} Let $\Gamma$ be geometrically finite and Zariski dense. There exist constants $\alpha = \alpha(\Gamma)>0, \omega=\omega(\Gamma) \ge 0, $ and $\theta = \theta(\Gamma)>\alpha$ satisfying the following: for any $x \in G/\Gamma$ with $x^+\in \Lambda(\Gamma)$, and for every $\xi>0$ and  $0<\eta\ll_\Gamma e^{-\height(x)}$, we have that for every hyperplane $L$, \[\ps_x(\mathcal{N}_U(L,\xi)\cap B_U(\eta)) \ll_\Gamma e^{\omega \height(x)} \frac{\xi^\alpha}{\eta^\theta}\ps_x(B_U(\eta)).\]
\end{proposition}

	We are now ready to prove Proposition \ref{prop: new friendliness}.
	
	\begin{proof}[Proof of Proposition \ref{prop: new friendliness}]
	
	It follows from the geometry of $B_U(\xi+\eta)x-B_U(\eta)x$ that there exist hyperplanes $L_1,\dots,L_m$, where $m$ only depends on $n$, such that\[ 
	B_U(\xi+\eta)x-B_U(\eta)x\subseteq\bigcup_{i=1}^{m} \mathcal{N}_U(L_i,2\xi).\]
	For any $0<\xi<\eta\ll_\Gamma e^{-\height(x)},$ we have that
	\begin{align*}
		\frac{\ps_x(B_U(\xi+\eta))}{\ps_x(B_U(\eta))} -1 & =\frac{\ps_x(B_U(\xi+\eta)-B_U(\eta))}{\ps_x(B_U(\eta))}\\
		& \le \sum_{i=1}^m\frac{\ps_x(\mathcal{N}(L_i,\xi)\cap B(x,\xi+\eta))}{\ps_x(B_U(\eta))}& \text{by Proposition \ref{prop: friendliness intermed step with planes}}\\
		& \ll_{\Gamma}m e^{\omega\height(x)}\frac{\xi^\alpha}{\eta^\theta}\cdot\frac{\ps_x( B_U(2\eta))}{\ps_x(B_U(\eta))}
	\end{align*}
	By \eqref{eq:ps is federer}, $\ps_x$ is $D$-Federer (see Corollary \ref{cor: effective federer leafwise} for more detail), in particular \[\ps_x(B_U(2\eta)) \ll_\Gamma \ps_x(B_U(\eta)).\]  
	Thus, we obtain 
	\[
	\frac{\ps_x(B_U(\xi+\eta))}{\ps_x(B_U(\eta))} -1 \ll_{\Gamma}e^{{\omega}\height(x)}\frac{ \xi^\alpha}{\eta^{\theta}},\] and relabeling the constants completes the proof.
	\end{proof}

    In \eqref{eq:ps is federer}, we saw that $\ps_x$ is Federer when $x \in \supp\bms.$ Below, we show that $\ps_x$ satisfies a similar condition for sufficiently large balls when $x$ is Diophantine, but not necessarily a BMS point.

	\begin{corollary}\label{cor: effective doubling for diophantine}
	There exists a constant $\sigma=\sigma(\Gamma)\ge \delta_\Gamma$ such that for every $c \ge 1$ and every $x \in G/\Gamma$ that is $(\eps,s_0)$-Diophantine, if $T \gg_{\Gamma,\eps} s_0 $, then \[\ps_x(B_U(cT))\ll_\Gamma c^\sigma \ps_x(B_U(T)).\]
	\end{corollary}
	\begin{proof}
	 By Lemma \ref{lem: horodist diophantine points}, for some $T_0\gg_{\Gamma,\eps} s_0$ there exists $$y \in B_U(T_0)x\cap\supp\bms.$$ Then for $T \ge T_0$, we have $$B_U(T-T_0)y\subseteq B_U(T)x\subseteq B_U(T+T_0)y.$$ 
	 Since $c \ge 1,$ we therefore have that for $T \ge 2T_0,$
	 \begin{align*}
	    \ps_x(B_U(cT))&\le \ps_y(B_U(cT+T_0))\\
	    &\le \ps_y(B_U((c+1)T))\\
	    &\ll_{\Gamma} (c+1)^\sigma\ps_y(B_U(T/2)) &\text{by \eqref{eq:ps is federer}}\\
	    &\ll_\Gamma (c+1)^\sigma \ps_y(B_U(T-T_0))\\
	    &\ll_\Gamma (c+1)^\sigma\ps_x(B_U(T))\\
	    &\ll_\Gamma (2c)^\sigma \ps_x(B_U(T)) &\text{since }c \ge 1 \\
	    &\ll_\Gamma c^\sigma \ps_x(B_U(T)).
	\end{align*} 
	\end{proof}
	
	\begin{remark}
	\label{remark: horodist}
	Observe that if $x$ is $(\eps,s_0)$-Diophantine and $T \gg_{\Gamma,\eps} s_0$, then $T$ is sufficiently large to use Corollary \ref{cor: effective doubling for diophantine} on $a_{-s}x$ for $s>0.$ To see this, observe that, in the notations of the proof of Corollary \ref{cor: effective doubling for diophantine}, $T_0$ is such that for any $(\eps,s_0)$-Diophantine point $x$, there exists $y \in \supp\bms$ and $\t\le T_0$ so that $x=u_\t y$. Then $$a_{-s}x = a_{-s}u_\t y=u_{e^{-s}\t}a_{-s}y.$$ Thus, the distance to the nearest BMS point in the $U$ orbit shrinks, and so $T$ is still sufficiently large.
	\end{remark}
	

	\begin{proposition}	\label{prop: friendliness large balls capped height} 
	Let $H_R = \{y \in G/\Gamma : \height(y)\le R\}$. 
	There exist constants $\alpha = \alpha(\Gamma)>0,$ and $\omega=\omega(\Gamma)\ge 0$ such that for every $x \in G/\Gamma$ that is $(\eps,s_0)$-Diophantine and for every $0<\xi<1/2,$ and $T\gg_{\Gamma,\eps} s_0$,
	\[	\frac{\ps_x((B_U(\xi+T) \cap H_R)-(B_U(T)\cap H_R) )}{\ps_x(B_U(T))}\ll_{\Gamma} e^{\omega R}{\xi^\alpha}.\]
	\end{proposition}
	
	\begin{proof}
    Let $T_0\gg_{\Gamma,\eps} s_0$ satisfy the conclusion of  Lemma \ref{lem: horodist diophantine points}. 
    For $T \ge T_0$, let\[E_T:=\mathcal{N}(L,\xi)\cap B_U(T) \cap H_R\cap \supp\bms,\] and observe that $\ps_x(E_T)=\ps_x(\mathcal{N}(L,\xi)\cap B_U(T)\cap H_R).$
    
    Let $c_1=c_1(\Gamma)>0$ be the implied constant in Proposition \ref{prop: friendliness intermed step with planes}. Fix $r=c_1e^{-R}$ and let $\{u_1,\dots,u_k\}$ be a maximal $\frac{r}{2}$-separated set in $E_{T-\frac{r}{4}}$. Then, $$E_T\subseteq\bigcup_{i=1}^k B_U(r)u_i.$$ 
    Note also that by \eqref{eq:ps is federer}, we have that there exists a constant $c_2=c_2(\Gamma)>0$ such that for all $u_i$, \begin{equation} \label{eq:using doubling in large ball result} 
        \ps_{u_i}(B_U(r))=\ps_{u_i}(B_U(8(r/8)) \le c_2\ps_{u_i}(B_U(r/8)).
    \end{equation}
    Therefore, 
	\begin{align*}
	    &\ps_x(\mathcal{N}(L,\xi)\cap B_U(T) \cap H_R)\\
	    & \le \sum_{i=1}^{k}\ps_{u_i}(\mathcal{N}(L,\xi)\cap B_U(r))\\
	    &\ll_{\Gamma}e^{\omega R}\frac{\xi^\alpha}{r^\theta}\sum_{i=1}^{k}\ps_{u_i}(B_U(r))&\text{by Proposition \ref{prop: friendliness intermed step with planes}}\\
	    &\ll_{\Gamma}e^{(\omega+\theta) R}\xi^\alpha\sum_{i=1}^{k}\ps_{u_i}(B_U(r/8))&\text{by \eqref{eq:using doubling in large ball result}}\\
	    &\ll_{\Gamma}e^{(\omega+\theta) R}\xi^\alpha\ps_{x}(B_U(T+1)) &\text{as the $1/8$ balls are disjoint}.
	\end{align*} By Corollary \ref{cor: effective doubling for diophantine}, there exists $\sigma =\sigma(\Gamma) \ge \delta_\Gamma$ so that \[\ps_x(B_U(T+1)) \subseteq \ps_x(B_U(2T)) \ll_\Gamma 2^\sigma \ps_x(B_U(T)).\] 
	
	Let $$\omega' = \omega+\theta.$$
	
	
	It follows from the geometry of $B_U(\xi+T)x-B_U(T)x$ that there exist $L_1,\dots,L_m$, where $m$ only depends on $n$, such that\[ 
	B_U(\xi+T)x-B_U(T)x\subseteq\bigcup_{i=1}^{m} \mathcal{N}_U(L_i,2\xi).\] Thus, we also have \[\left(B_U(\xi + T)x - B_U(T)x\right) \cap H_R\subseteq\bigcup_{i=1}^{m} \mathcal{N}_U(L_i,2\xi).\]
	
	We arrive at
	\begin{align*}
		\frac{\ps_x((B_U(\xi+T) \cap H_R)-(B_U(T)\cap H_R) )}{\ps_x(B_U(T))} 
		& \le \sum_{i=1}^m\frac{\ps_x(\mathcal{N}(L_i,2\xi)\cap B_U(\xi+T))}{\ps_x(B_U(T))}\\
		& \ll_{\Gamma}m e^{\omega'R}{\xi^\alpha} \frac{\ps_x( B_U(\xi+T))}{\ps_x(B_U(T))}
	\end{align*}
	By Corollary \ref{cor: effective doubling for diophantine} again, we conclude that \[
		\frac{\ps_x((B_U(\xi+T) \cap H_R)-(B_U(T)\cap H_R) )}{\ps_x(B_U(T))}
		 \ll_{\Gamma}e^{\omega'R}\xi^\alpha,\] which completes the proof.
	
	
	\end{proof}

	
	We are now ready to prove Theorem \ref{thm: nicer friendly}.

	\begin{proof}[Proof of Theorem \ref{thm: nicer friendly}]
	Observe that by Lemma \ref{lem: height and distance to C0},
	\begin{equation}
	    \label{eq: first step in nicer friendly}
	    \ps_{a_{-\log s}x}(B_U(T+\xi))=\ps_{a_{-\log s}x}(B_U(T+\xi)\cap H_{R-R_0}) + \ps_{a_{-\log s}x}(B_U(T+\xi)\cap \mathcal{X}(R))\\
	\end{equation}
	By Theorem \ref{thm:bounded functions}, for $T \gg_{\Gamma,\eps} s_0$, $0<s\le T^{\frac{\eps}{1-\eps}}$, and any $R \ge R_0,$ \begin{align*}
	    \ps_{a_{-\log s}x}(B_U(T+\xi)\cap \mathcal{X}(R))&= \ps_{a_{-\log s}x}(B_U((s(T+\xi)/s) \cap \mathcal{X}(R)) \\
	    &\ll_\Gamma \ps_{a_{-\log s}x}(B_U(T+\xi))e^{-\beta R}\\
	    &\ll_\Gamma \ps_{a_{-\log s}x}(B_U(T))e^{-\beta R} &\text{ by Corollary }\ref{cor: effective doubling for diophantine} 
	\end{align*} Observe that use of Corollary \ref{cor: effective doubling for diophantine} is justified if $T \gg_{\Gamma,\eps} s_0$ by Remark \ref{remark: horodist}. Similarly, by Proposition \ref{prop: friendliness large balls capped height} and the same reasoning as in Remark \ref{remark: horodist}, for $T \gg_{\Gamma,\eps} s_0$, we have
	\begin{align*}
	    \ps_{a_{-\log s}x}(B_U(T+\xi)\cap H_{R-R_0})&\ll_\Gamma e^{\omega R}\xi^\alpha \ps_{a_{-\log s}x}(B_U(T))+\ps_{a_{-s}x}(B_U(T)\cap H_{R-R_0}).
	\end{align*}
	
	Putting this together with \eqref{eq: first step in nicer friendly}, we conclude
	\begin{align*}
	    &\ps_{a_{-\log s}x}(B_U(T+\xi))\\
	    &=\ps_{a_{-\log s}x}(B_U(T+\xi)\cap H_{R-R_0}) + \ps_{a_{-\log s}x}(B_U(T+\xi)\cap \mathcal{X}(R))\\
	    &\ll_\Gamma \left[e^{\omega R}\xi^\alpha \ps_{a_{-\log s}x}(B_U(T))+\ps_{a_{-\log s}x}(B_U(T)\cap H_{R-R_0})\right] + e^{-\beta R}\ps_{a_{-\log s}x}(B_U(T))\\
	    &\ll_\Gamma \left(e^{\omega R}\xi^\alpha +e^{-\beta R}+1\right)\ps_{a_{-\log s}x}(B_U(T))
	\end{align*}
	Taking $R=-\frac{\alpha}{\omega+\beta}\log\xi$ implies the result, provided that $\xi$ is sufficiently small so that this is larger than $R_0.$ Note that since $\alpha,\omega,\beta, R_0$ are all constants depending only on $\Gamma$, this is equivalent to requiring $\xi \ll_\Gamma 1.$ 
	\end{proof}

	\section{Proof of Theorem \ref{thm: PS translates thm}}\label{section: proof of PS translates theorem}
	
	In this section, we keep the notation of \S \ref{section: thick thin decomposition}. In particular, $d$ denotes the hyperbolic distance, $\height$ is the height of a point in the convex core into the cusps, and $\mathcal{C}_0$ is the fixed compact set in $G/\Gamma$ which is defined in \S \ref{section: thick thin decomposition}. 
	
	We will first prove the following proposition, which is a form of Theorem \ref{thm: PS translates thm} for $G$. Theorem \ref{thm: PS translates thm} will follow by a partition of unity argument.
	
	\begin{proposition}\label{thm; effective ps}
		There exist $\kappa=\kappa(\Gamma)$ and $\ell=\ell(\Gamma)$ which satisfy the following:
		let $0<r<1$, 
		$\psi\in C_c^\infty(G)$ supported on an admissible box, and $f\in C_c^\infty(B_U (r))$. 
		Then, there exists $c=c(\Gamma,\supp\psi)>0$ such that for any $g\in\supp\tbms$,  and $s\gg_\Gamma \height(g\Gamma),$ we have 
		\begin{align*}\label{eq:effective ps}
		&\left|\sum\limits_{\gamma \in \Gamma}\int_{U}\psi(a_su_\t g\gamma)f(\t)d\ps_g(\t)- \ps_g(f)\tbms(\psi)\right|\\
		&<cS_{\ell}(\psi)S_{\ell}(f)e^{-\kappa s}\ps_g(B_U(1)). 
		\end{align*}
	\end{proposition}
	
	\begin{proof}
		Without loss of generality assume that $f$ and $\psi$ are non-negative functions.
		
		\noindent\textbf{Step 1: Setup and approximations.}
		
		Let $\kappa',\ell'$ satisfy the conclusion of Assumption \ref{thm:effective mixing}, and let $\ell>\ell'$ satisfy the conclusion of Lemma \ref{lem:choosing ell}. Observe that $\ell$ can be increased if necessary while maintaining this property.
		
		Because $\psi$ is supported on an admissible box, there exists $0<\eta_0<1/2$ (depending on $\supp\psi$) such that $G_{3\eta_0}\supp\psi$ is still an admissible box. For $0<\eta<\eta_0$, let $\psi_{\eta,\pm}$ satisfy the conclusion of Lemma \ref{lem:choosing ell} for $G,3\eta$, and $\psi$. In particular, for all small $\eta>0$ 
		\begin{equation}\label{eq:phi+- sob bound}
		S_{\ell'}(\psi_{\eta,\pm})\ll_{\supp\psi}\eta^{-2\ell}S_{\ell}(\psi).
		\end{equation}		
		 
		Since $\psi$ is uniformly continuous and the BMS-measure is finite, we may deduce from Lemma \ref{lem:choosing ell}(2) that
		\begin{equation}\label{eq:phi+- bms bound}
		\left|\tbms(\psi_{\eta,\pm})-\tbms(\psi)\right|\ll_{\supp\psi,\Gamma}\eta S_\ell(\psi).
		\end{equation}
		
		According to Lemma \ref{lem:rho}, for any $p \in P_\eta$, there exists $\rho_p : B_U(1) \to B_U(1+O(\eta))$ that is a diffeomorphism onto its image and a constant $D=D(\eta)<3\eta$ such that \begin{equation}\label{eq: rho p t}u_t p\inv \in P_D u_{\rho_p(t)}.\end{equation}	

		\noindent\textbf{Step 2: Assuming that $f$ is supported on a small ball.}
		
		We start by proving that there exists $\kappa > 0$ such that if $f\in C_c^\infty(B_U(r_1))$, where $r_1\le\inj(g)$, then for $s>0$,
		\begin{align}\label{eq:bound for r1}   & \sum_{\gamma\in\Gamma}\int_U\psi(a_s u_\t g\gamma)f(\t)d\ps_g(\t)-\tbms(\psi)\ps_g(f)\\ &\ll_{\Gamma,\supp\psi}S_{\ell}\left(\psi\right)S_{\ell}\left(f\right)e^{-2\kappa s}\ps_g(B_U(1)).\nonumber
		\end{align} 
		

		For any $s>0$ and $\gamma \in \Gamma$, from \eqref{eq: rho p t} we have that
		\begin{align*}
		&\int_{B_U(r_1)} \psi(a_s u_\t g\gamma)f(\t)d\ps_g(\t) \\
		&= \frac{1}{\nu(P_{\eta} g)}\int_{P_{\eta} g}\int_{B_U(r_1)}\psi(a_{s}u_\t p\inv pg\gamma)f(\t)d\ps_{g}(\t)d\nu(pg)\\
		&\le \frac{1}{\nu(P_{\eta} g)} \int_{P_{\eta} g}\int_{B_U(r_1)} \psi_{\eta,+}(a_{s}u_{\rho_p(\t)}pg\gamma)f(\t)d\ps_{g}(\t)d\nu(pg),
		\end{align*} 
		where the last inequality follows since $a_s P_{3\eta}a_{-s}\subset P_{3\eta}$ for any positive $s$. \newline
		
		\noindent\textbf{Step 2.1: Use the product structure of the BMS measure.}
		
		For any $p\in P_{\eta}$, $(u_\t g)^+= (u_{\rho_p (\t)}pg)^+$, the measures $d\ps_{g}(\t)$ and $d({\rho_p}_\ast\ps_{pg}(\t))=d\ps_{pg}(\rho_p (\t))$ are absolutely continuous with each other, and the Radon-Nikodym derivative at $\t$ is given by
		\begin{equation}\label{eq:rad-nik}
		\frac{d\ps_{g}(\t)}{d\ps_{pg}(\rho_p (\t))}=e^{\delta_\Gamma\beta_{(u_\t g)^+}(u_\t g(o),u_{\rho_p(\t)}pg(o))}.
		\end{equation}

		Let $0<\xi<\eta$. Let $\chi_{\eta,\xi}$ satisfy the conclusion of Lemma \ref{lem; existence of smooth indicator} for $H=P$, $\xi_1=\eta-\xi$, $\xi_2=\xi$, and $g$. Let $\varphi_{\eta,g}$ be the function defined on $B_U (1) P_{\eta}g$ given by \[
		\varphi_{\eta,g}(u_{\rho_p(\t)}pg):=\frac{f(\t)\chi_{\eta,\xi}(pg)}{\nu(P_{\eta} g)e^{\delta_\Gamma\beta_{(u_\t g)^+}(u_\t g(o),u_{\rho_p(\t)}pg(o))}}.\] 
		
		We will need a bound on $S_\ell(\varphi_{\eta,g}).$ To that end, note that 
		\begin{align*}
		\left|\beta_{(u_\t g)^+}(u_\t g(o), u_{\rho_p(\t)} pg(o))\right| &\le d(u_\t g(o), u_{\rho_p(\t)}pg(o))\\
		&= d(g(o), u_{-\t} u_{\rho_p(\t)}pg(o)).
		\end{align*}
		Since $u_{-\t} u_{\rho_p(\t)}p\in G_{5\eta}$, the above is bounded by some absolute constant (depending only on $\Gamma$) for all $\eta<\frac{1}{2}$. 
		
		Thus, because the Busemann function is Lipschitz, we have that for all $p \in P_\eta$,
		\begin{align}\label{eq:buse sob bound}
		S_\ell(\beta_{(u_\t g_0)^+}(u_\t g(o),u_{\rho_p(\t)}pg(o)))\ll_{\Gamma}1.
		\end{align} 
		
		By \cite[Lemma 2.4.7(a)]{KleinbockMargulis}, Lemma \ref{lem; existence of smooth indicator}, \eqref{eq:buse sob bound}, and Lemma \ref{lem:bound on nu}, we have
	\begin{align}
		    S_\ell(\varphi_{\eta,g})&\ll_{\Gamma,\ell} \nu(P_\eta g)\inv S_\ell(f)S_\ell(\chi_{\eta,\xi})\nonumber\\
		    &\ll_{\Gamma,\ell} \eta^{-(\delta_\Gamma + \frac{1}{2}(n-1)(n-2)+1)}e^{(\delta_\Gamma - k_2(x,\eta))d(\pi(\mathcal{C}_0),\pi(a_{\log \eta}g))}S_\ell(\chi_{\eta,\xi})S_\ell(f)\nonumber \\
		    &\ll_{\Gamma,\ell} e^{\delta_\Gamma(|\log \eta|+\height(g\Gamma))} \eta^{-(\delta_\Gamma + \frac{1}{2}(n-1)(n-2)+1)}\eta^{n-1} {\xi}^{-\ell-(n-1)/2} S_\ell(f)\nonumber\\
		    &\ll_{\Gamma,\ell} e^{\delta_\Gamma \height(g\Gamma)}\eta^{-(2\delta_\Gamma + \frac{1}{2}(n-1)(n-2)+1)}\eta^{n-1} {\xi}^{-\ell-(n-1)/2} S_\ell(f)\nonumber\\
		    &\ll_{\Gamma,\ell} e^{\delta_\Gamma \height(g\Gamma)} \eta^{4n-\frac{1}{2}n^2-3-2\delta_\Gamma}\xi^{-\ell-(n-1)/2}S_\ell(f)\label{eq:sobolev bound varphi}
		\end{align} Note that the dependence on $\ell$ arises from the exponential of the Busemann function in the denominator.

		Also, using the product structure of $\tbms$ in \eqref{eq:bms product structure}, we get
		\begin{align*}
		& \frac{1}{\nu(P_{\eta}g)} \int_{P_{\eta} g}\int_{B_U(r_1)} \psi_{\eta,+}(a_{s}u_{\rho_p(\t)}pg\gamma)f(\t)d\ps_{g}(\t)d\nu(pg)\\
		&= \frac{1}{\nu(P_{\eta}g)}\int_{P_{\eta} g}\int_{B_U(r_1)} \psi_{\eta,+}(a_{s}u_{\rho_p(\t)}pg\gamma)f(\t)\frac{d\ps_{g}(\t)}{d\ps_{pg}(\rho_p (\t))}d\ps_{pg}(\rho_p (\t))d\nu(pg)\\
		&\leq\int_G \psi_{\eta,+}(a_{s}h\gamma)\varphi_{\eta,g}(h)d\tbms(h).\\
		\end{align*}  
		
		\noindent\textbf{Step 2.2: Use the exponential mixing assumption.}
		
		By defining $\Psi_{\eta,+}(h\Gamma) = \sum\limits_{\gamma \in \Gamma} \psi_{\eta,+}(h\gamma)$ and $\Phi_{\eta,g}(h\Gamma):=\sum\limits_{\gamma \in \Gamma}\varphi_{\eta,g}(h\gamma),$ we obtain 
		\[
		\sum_{\gamma\in\Gamma}\int_G\psi_{\eta,+}(a_sh\gamma)\varphi_{\eta,g}(h)d\tbms(h)\leq\int_{X} \Psi_{\eta,+}(a_s x)\Phi_{\eta,g}(x)d\bms(x)\]
		for any positive $s$. Note that 
		\begin{equation}\label{eq:Sobolev for Phi}
		S_{\ell'}(\Psi_{\eta,+})=S_{\ell'}(\psi_{\eta,+}) \text{ and }S_{\ell'}(\Phi_{\eta,g})=S_{\ell'}(\varphi_{\eta,g}).
		\end{equation}
		In particular, \eqref{eq:phi+- sob bound} and \eqref{eq:sobolev bound varphi} imply 
		\begin{align}
		&S_{\ell'}(\Psi_{\eta,+})\ll_{\supp\psi}\eta^{-2\ell}S_{\ell}(\psi) \text{ and }\nonumber\\&S_{\ell'}(\Phi_{\eta,g})
		\ll_{\Gamma} e^{\delta_\Gamma \height(g\Gamma)} \eta^{4n-\frac{1}{2}n^2-3-2\delta_\Gamma}\xi^{-\ell-(n-1)/2}S_\ell(f).\label{eq:Sobolev for Psi}
		\end{align}
		
		By Assumption \ref{thm:effective mixing},
		\begin{align*}
		&\int\Psi_{\eta,+}\left(a_{s}x\right)\Phi_{\eta,g}\left(x\right)d\bms\left(x\right)-\bms\left(\Psi_{\eta,+}\right)\bms\left(\Phi_{\eta,g}\right)\\
		&\ll_{\Gamma} S_{\ell'}\left(\Psi_{\eta,+}\right)S_{\ell'}\left(\Phi_{\eta,g}\right)e^{-\kappa' s}.
		\end{align*}
		Then, by \eqref{eq:Sobolev for Psi}, there exists $c_1=c_1(\Gamma,\supp\psi)$ such that
		\begin{align*}
		&\sum_{\gamma\in\Gamma}\int_{B_U(r_1)} \psi(a_su_\t g\gamma)f(\t)d\ps_g(\t)\\
		&<\bms\left(\Psi_{\eta,+}\right)\bms\left(\Phi_{\eta,g}\right)+c_1e^{\delta_\Gamma \height(g\Gamma)} \eta^{4n-\frac{1}{2}n^2-3-2\delta_\Gamma}\xi^{-\ell-(n-1)/2}S_{\ell}\left(\psi\right)S_{\ell}\left(f\right)e^{-\kappa' s}.\\
		\end{align*}
		
		\noindent\textbf{Step 2.3: Rewrite in terms of $\psi$ and $f$.}
		
		Using Lemma \ref{lem; existence of smooth indicator} and \eqref{eq:rad-nik}, one can calculate
		\begin{align*}
		\bms(\Phi_{\eta,g})&=\int_{G}\varphi_{\eta,g}(h)d\tbms(h)\\
		&=\frac{1}{\nu(P_{\eta} g)}\int_{Pg}\int_U\frac{f(\t)\chi_{\eta,\xi}(p)}{e^{\delta_\Gamma\beta_{(u_\t g)^+}(u_\t g(o),u_{\rho_p(\t)}pg(o))}}d\ps_{pg}(\rho_p (\t))d\nu(pg)\\
		&=\frac{1}{\nu(P_{\eta} g)}\int_{Pg}\int_U f(\t)\chi_{\eta,\xi}(p)d\ps_{g}(\t)d\nu(pg)\\
		&\leq\frac{\nu(P_{\eta+\xi} g)}{\nu(P_{\eta} g)}\int_{B_U(r_1)}f(\t)d\ps_{g}(\t).
		\end{align*} 
		
		Thus, by Theorem \ref{prop: nu new friendly}, there exist $\alpha,\theta,\omega,c_0>0$ depending only on $\Gamma$ such that for any $0<\xi<\eta\ll_\Gamma e^{-\height(g\Gamma)}$
		\begin{align*}
		\bms(\Phi_{\eta,g})&\leq \left(1+c_2e^{\omega\height(g\Gamma)} \frac{\xi^{\alpha}}{\eta^{\theta}}\right)\int_{B_U(r_1)}f(\t)d\ps_{g}(\t)\\
		&=\left(1+c_2e^{\omega\height(g\Gamma)} \frac{\xi^{\alpha}}{\eta^{\theta}}\right)\int_{B_U(r_1)}f(\t)d\ps_{g}(\t)\\
		&=\left(1+c_2e^{\omega\height(g\Gamma)} \frac{\xi^{\alpha}}{\eta^{\theta}}\right)\ps_g(f).
		\end{align*}  
		
		Using \eqref{eq:phi+- bms bound}, we get that there exists $c_3=c_3(\Gamma,\supp\psi)$ such that
		\begin{align*}
		\bms(\Psi_{\eta,+})&\leq\int_G\psi_{\eta,+}(g)d\tbms(g)\\
		&<\tbms(\psi)+c_3\eta S_\ell(\psi).
		\end{align*}

		To summarize, we have  
		\begin{align*}
		&\sum_{\gamma\in\Gamma}\int_{B_U(r_1)}\psi(a_s u_\t g\gamma)f(\t)d\ps_g(\t)\\
		&\le \frac{1}{\nu(P_{\eta} g)}\sum_{\gamma\in\Gamma}\int_{P_{\eta} g}\int_{B_U(r_1)} \psi_{\eta,+}(a_{s}u_{\rho_p(\t)}pg\gamma)f(\t)d\ps_{g}(\t)d\nu(pg)\\
		&\leq\sum_{\gamma\in\Gamma}\int_G \psi_{\eta,+}(a_{s}h\gamma)\varphi_{\eta,g}(h)d\tbms(h)\\
		&\leq\int_{X} \Psi_{\eta,+}(a_s x)\Phi_{\eta,g}(x)d\bms(x)\\
		&<\bms\left(\Psi_{\eta,+}\right)\bms\left(\Phi_{\eta,g}\right)+c_1{\eta}^{4n-\frac{1}{2}n^2-3-\delta_\Gamma-2\ell}{\xi}^{-\ell -(n-1)/2}S_{\ell}\left(\psi\right)S_{\ell}\left(f\right)e^{-\kappa' s}\\
		&<\left(\tbms(\psi)+c_3\eta S_\ell(\psi)\right)\left(\left(1+c_2e^{\omega\height(g\Gamma)} \frac{\xi^{\alpha}}{\eta^{\theta}}\right)\ps_g(f)\right)\\
		& \quad+c_1e^{\delta_\Gamma \height(g\Gamma)} \eta^{4n-\frac{1}{2}n^2-3-2\delta_\Gamma}\xi^{-\ell-(n-1)/2}S_{\ell}\left(\psi\right)S_{\ell}\left(f\right)e^{-\kappa' s}. 
		\end{align*}
		It follows from the proof of Lemma \ref{lem:choosing ell} that $\tbms(\psi)\ll_{\supp\psi}S_\ell(\psi)$ and  $\ps_g(f)\ll S_\ell(f)\ps_g(B_U(1))$. Then, using Proposition \ref{prop:shadow lemma} we arrive at
		\begin{align*}
		&\sum_{\gamma\in\Gamma}\int_{B_U(r_1)}\psi(a_s u_\t g\gamma)f(\t)d\ps_g(\t)- \ps_g(f)\tbms(\psi)\\
		&\ll_{\Gamma}\left(e^{\omega\height(g\Gamma)} \frac{\xi^{\alpha}}{\eta^{\theta}}
		+e^{\delta_\Gamma \height(g\Gamma)} \eta^{4n-\frac{1}{2}n^2-3-2\delta_\Gamma}\xi^{-\ell-(n-1)/2}e^{-\kappa' s}\right)\\
		&\quad\quad\cdot S_{\ell}\left(\psi\right)S_{\ell}\left(f\right)\ps_g(B_U(1))
		\end{align*}
		
		Define \[\kappa = \frac{3\alpha\theta\kappa'}{2\theta(2\ell+n-1)+9\alpha(2\delta_\Gamma+3+n^2/2-4n)},\] and note that by making $\ell$ larger if necessary, we guarantee $\kappa>0$.
		Recall from \eqref{eq: inj radius and height} that $$e^{-\height(g\Gamma)}\ll_\Gamma\inj(g).$$ For $s \ge\max\{\theta,\omega\}\height(g\Gamma) /\kappa$, choose
		\begin{equation}\label{eq:choosing eta,xi}
		\eta=e^{-\kappa s/\theta},\:\xi=e^{-{4\kappa s}/{\alpha} }. 
		\end{equation} 
		Note that $\eta < \inj(g\Gamma)$ by choice of $s$, $\omega\height(g\Gamma)\le\kappa s$, and $\xi<\eta$ since by Proposition \ref{prop: nu new friendly}, $\alpha<\theta$. By Proposition \ref{prop: nu new friendly} we have $\omega>\delta_\Gamma$, therefore $\delta_\Gamma\height(g\Gamma)\le\kappa s$. Note also that  $\max\{\theta,\omega\}\height(g\Gamma) /\kappa\ll_\Gamma \height(g\Gamma)$. 
		
		With these choices, we obtain  
		\begin{equation}e^{\omega\height(g\Gamma)}\left(\frac{\xi}{\eta^{\theta'}}\right)^{\alpha'}
		+e^{\delta_\Gamma \height(g\Gamma)} \eta^{4n-\frac{1}{2}n^2-3-2\delta_\Gamma}\xi^{-\ell-(n-1)/2}e^{-\kappa' s} \le 2e^{-2\kappa s}. \label{eq:eta xi kappa' kappa'' ineq}
		\end{equation}
		In a similar way, using $\psi_{\eta,-}$, one can show a lower bound, proving \eqref{eq:bound for r1}. \newline

		\noindent\textbf{Step 3: Covering argument for general $f$.}
		
		We now deduce the claim by decomposing $f$ into a sum of functions, each defined on a ball of radius $r_1$ in $U$. 
		
		Let $u_1,\dots,u_k$ and $\sigma_1,\dots,\sigma_k\in C_c^\infty(B_U(r))$ satisfy the conclusion of Lemma \ref{lem:partition of unity} for $E=B_U(r)$ and $r_1$. 
		For $1\leq i\leq k$, let \[
		f_i:=f\sigma_i.\] 
		Then, $f\le\sum_{i-1}^k f_i$, and by Lemma \ref{lem:partition of unity} and \cite[Lemma 2.4.7(a)]{KleinbockMargulis}
		\begin{equation}\label{eq:sobolev bound f_i}
		S_\ell(f_i)\ll_{\Gamma} S_\ell(f)S_\ell(\sigma_i)\ll_{\Gamma} r_1^{-\ell+n-1} S_\ell(f).
		\end{equation}
		
		Since each $f_i$ is supported on $B_U(r_1)u_i$ for some $u_i\in B_U(1)$, by \eqref{eq:bound for r1} we have
		\begin{align*}
		& \sum_{\gamma\in\Gamma}\int_{B_U(r_1)}\psi(a_s u_\t g\gamma)f_i(\t)d\ps_g(\t)-\tbms(\psi)\ps_g(f_i)\\
		& \ll_{\Gamma,\supp\psi}\ps_g(B_U(1)) S_{\ell}\left(\psi\right)S_{\ell}\left(f_i\right)e^{-2\kappa s}. 
		\end{align*} 
		
		Summing the above expressions for $i=1,\dots,k$, we get
		\begin{align*}
		&\sum_{\gamma\in\Gamma}\int_{B_U(r)}\psi(a_s u_\t g\gamma)f(\t)d\ps_g(\t)-\tbms(\psi)\ps_g(f)\\
		& \ll k r_1^{-\ell+n-1}S_{\ell}\left(\psi\right)S_{\ell}\left(f\right)e^{-2\kappa s}\ps_g(B_U(1))\\
		& \ll\left(\frac{r}{r_1}\right)^{n-1} r_1^{-\ell+n-1}S_{\ell}\left(\psi\right)S_{\ell}\left(f\right)e^{-2\kappa s}\ps_g(B_U(1))\\
		& \ll r_1^{-\ell}S_{\ell}\left(\psi\right)S_{\ell}\left(f\right)e^{-2\kappa s}\ps_g(B_U(1))\\
		& \ll S_{\ell}\left(\psi\right)S_{\ell}\left(f\right)e^{-\kappa s}\ps_g(B_U(1)),
		\end{align*} 
		where the first inequality is by Lemma \ref{lem:partition of unity}, the second inequality follows from $r_1=\inj(g)>e^{-\kappa s/\ell}$, the third is by \eqref{eq:choosing eta,xi} and because $r<1$, and the implied constants depend on  $\Gamma$ and $\supp\psi$. 
		
		As before, using similar arguments, one can show a lower bound, proving the claim. 
		
	\end{proof}

	We will now use a partition of unity argument to prove Theorem \ref{thm: PS translates thm}. For the reader's convenience, we restate it below. 
	
	\begin{theorem}\label{restated PS thm}
		There exist $\kappa = \kappa(\Gamma)$ and $\ell = \ell(\Gamma)$ which satisfy the following: 
		for any $\psi \in C_c^\infty(X)$, there exists $c = c(\Gamma, \supp\psi)>0$ such that for any $f \in C_c^\infty(B_U(r))$, $0<r<1$, $x \in \supp\bms$, and $s\gg_\Gamma \height(x)$, we have 
		\begin{align*}
		\left|\int_U \psi(a_s u_\t x)f(\t)d\ps_x(\t) - \ps_x(f)\bms(\psi)\right| < cS_\ell(\psi)S_\ell(f)e^{-\kappa s}.
		\end{align*}
	\end{theorem}

	\begin{proof}
		According to \cite[Lemma 2.17]{OhShah}, there exists an admissible box $B_y$ around $y$, for any $y\in X$. Then, $\left\{B_y\::\:y\in\supp\psi\right\}$ is an open cover of the compact set $\supp\psi$. Hence, there exists a minimal sub-cover $B_{y_1},\dots,B_{y_k}$. Using a similar construction to one in Lemma \ref{lem:partition of unity}, there exist $\sigma_1,\ldots, \sigma_k$, a partition of unity for $\supp\psi$, such that for $i=1,\dots,k$ we have $\sigma_i \in C_c^\infty(B_{y_i})$ and for $i=1,\dots,k$ and $m=1,\dots,\ell$  
		\begin{equation}\label{eq:sigma_i sob bound}
		|\sigma_i^{(m)}|\ll_{\supp\psi,\Gamma} 1
		\end{equation}
		(the implied constant depends on the chosen sub-cover). 
		
		Define $\psi_i=\psi\sigma_i$. Then 
		\begin{equation}\label{eq:psi as sum PS}
		\psi=\sum_{i=1}^k\psi_i,
		\end{equation}
		and by \eqref{eq:sigma_i sob bound} and the product rule, we have
		\begin{equation}\label{eq:s_ell psi_i PS}
		S_\ell(\psi_i)\ll_{\supp\psi,\Gamma}S_\ell(\psi).
		\end{equation}
		According to Proposition \ref{thm; effective ps} and Proposition \ref{prop:shadow lemma}, there exist $c=c(\Gamma,\supp\psi)>0$, $\lambda=\lambda(\Gamma)>1$ such that for $s\gg_\Gamma \height(x)$,
		\begin{align*}
		&\int_{B_U(r)} \psi(a_s u_\t x)f(\t)d\t\\
		&= \sum\limits_{i=1}^k \int_{B_U(r)} \psi_i(a_s u_\t x)f(\t)d\t \\
		& \le \sum\limits_{i=1}^k \bms(\psi_i) \ps_x(f) + c S_\ell(\psi_i)S_\ell(f)e^{-\kappa s}\ps_x(B_U(1))\\
		& \le \sum\limits_{i=1}^k \bms(\psi_i) \ps_x(f) + c \lambda S_\ell(\psi_i)S_\ell(f)e^{-\kappa s+(n-1-\delta_\Gamma)d(\pi(\mathcal{C}_0),\pi(x))}\\
		&\ll_{\Gamma,\supp\psi}  \bms(\psi) \ps_x(f) + c \lambda S_\ell(\psi)S_\ell(f)e^{-\kappa s+(n-1-\delta_\Gamma)\height(x)}.
		\end{align*}
		where the last line follows by the definition of $\height(x)$ and equations \eqref{eq:psi as sum PS} and \eqref{eq:s_ell psi_i PS}. Moreover, we may assume that $s \ge \frac{2(n-1-\delta_\Gamma)}{\kappa}\height(x)$ without changing the assumption $s\gg_\Gamma \height(x)$. Then $$e^{-\kappa s+(n-1-\delta_\Gamma)\height(x)} \ll_\Gamma e^{-\kappa s/2 },$$ as desired.
	\end{proof}

	We will now use Theorem \ref{thm: PS translates thm} to prove a similar result for the Haar measure. This will be necessary for the proof of Theorem \ref{thm; main}. Note that such a result is proven in \cite{Matrix coefficients} under a spectral gap assumption on $\Gamma$, but we show here how to prove it whenever the frame flow is exponentially mixing. 
	
	\begin{theorem}\label{cor; br translates restated} There exists $\kappa = \kappa(\Gamma)<1$ and $\ell = \ell(\Gamma)$ that satisfy the following: let $0<r<1$, let $f \in C_c^\infty(B_U(r))$, and let $\psi \in C_c^\infty(X)$ be supported on an admissible box. Then there exists $c = c(\Gamma,\supp\psi)>0$ such that for every $x \in \supp\bms$ and $s\gg_{\Gamma,\supp\psi} \height(x)$,
		\begin{align*}
		\left| e^{(n-1-\delta_\Gamma)s} \int_{B_U(r)}\psi(a_su_\t x)f(\t)d\t - \ps_x(f)\br(\psi)\right|< cS_\ell(\psi)S_\ell(f)e^{-\kappa s}.
		\end{align*}
	\end{theorem}
	
	\begin{proof}
		
		\noindent\textbf{Step 1: Setup and approximations.}
		
		Assume $s \gg_\Gamma \height(x)$, and let $\kappa,\ell'$ satisfy the conclusion of Theorem \ref{thm: PS translates thm}, and $\ell>\ell'$ satisfy the conclusion of Lemma \ref{lem:choosing ell}.
		
		Since $\psi$ is assumed to be supported on an admissible box, there exist $r_0,\eta,\eps_0,\eps_1>0$ (depending only on $\supp \psi$) and $z\in X$ such that $$\supp{\psi}=B_U(r_0) P_{\eta}z,$$ and $$G_{\eps_0}\supp\psi\subset B_U(r_0+\eps_1) P_{\eta+\eps_1}z,$$ where $B_U(r+\eps_1) P_{\eta+\eps_1}z$ is also an admissible box. Denote $\eta'=\eta+\eps_1$ and $r'_0=r_0+\eps_1$. 
		
		Without loss of generality, assume that $f$ is a non-negative function. Continuously extend $\psi$ to $P_{\eta'}$ by defining $\psi=0$ on $P_{\eta'}\setminus P_{\eta}$. 
		
		For $0<\eps<\eps_0$ let $\psi_{\eps,\pm}$ and $f_{\eps,\pm}$ for Lemma \ref{lem:choosing ell} for $G,\eps,\psi$ and $U,\eps,f$, respectively. 
		By Lemma \ref{lem:choosing ell}, 
		\begin{equation}
		\label{eqn; ucty of psi, f}
		S_{\ell'}(\psi_{\eps,\pm})\ll_{\Gamma,\supp(\psi)} \eps^{-2\ell}S_{\ell}(\psi)\quad\text{and}\quad S_{\ell'}(f_{\eps,\pm})\ll_{\Gamma}\eps^{-2\ell}S_{\ell}(f). 
		\end{equation} Moreover, by Lemma \ref{lem:choosing ell}(2), \begin{equation}\label{eq: contintuity feps pm from sobolev lemma} \|f_{\eps,\pm}-f\|_\infty \le \eps S_\ell(f).\end{equation}


		For $p \in P_{\eta'}$, define 
		\begin{equation}\label{eq:varphi def}
		\varphi(p):= \ps_{pz}(B_U(r'_0)pz).
		\end{equation}
		
		\noindent\textbf{Step 1.1: Construct a smooth approximation to $1/\varphi$.}
		
		Since the Busemann function is smooth and $\varphi$ is bounded below by a positive quantity on $P_{\eta'}$ by Corollary \ref{cor; inf sup on compact sets for ps}, the mean value theorem implies that for any $0<\eps<\eps_0$ and all $p, p' \in P_{\eps},$ there exists a constant $d = d(\Gamma,\supp\psi)$ such that  \begin{equation}\label{eqn; continuity varphi} \left|\frac{1}{\varphi(p)} - \frac{1}{\varphi(p')}\right| \le \frac{d\eps}{\varphi(p)}.\end{equation}
		
		By Lemma \ref{lem; existence of smooth indicator}, for any $\xi>0$, there exists a non-negative smooth function $\chi_\xi$ with \begin{equation}\label{eqn; bounds on chi xi} 1_{P_\eps-\xi}\le \chi_\xi \le \1_{P_{\eps}}\end{equation} and $S_{\ell'}(\chi_\xi) \ll_{\Gamma,n} (\eps-\xi/2)^{n-1}(\xi/2)^{-\ell'-(n-1)/2}.$ Define \begin{equation}\label{eqn; defn of sigma p}\sigma(p):= \frac{1}{\varphi}\ast \frac{\chi_\xi}{m(P_{\eps-\xi})}\end{equation} where $m$ denotes the probability Haar measure on $P$. Then, assuming $\eps_0<1/2$ and $\xi\le\eps^2$, by (\ref{eqn; continuity varphi}), (\ref{eqn; bounds on chi xi}), and (\ref{eqn; defn of sigma p}), we have that 
		\begin{align}
		\frac{1-d\eps}{\varphi(p)} & \le\frac{1}{m(P_{\eps-\xi})}\int_{pP_{\eps - \xi}}\frac{1}{\varphi(p')}dp' \label{eq:lower bound on sigma p}\\
		& \le \sigma(p) \nonumber\\
		&\le \frac{1}{m(P_{\eps-\xi})}\int_{pP_{\eps}}\frac{1+d\eps}{\varphi(p')}dp' \nonumber \\
		& \le \left(\frac{\eps}{\eps-\xi}\right)^n\frac{1+d\eps}{\varphi(p)}\nonumber\\
		& \le\frac{1+d'\eps}{\varphi(p)}, \label{eqn; bounds on sigma p}
		\end{align} 
		for some absolute constant $d'>0$.
		
		For $upz \in B_U(r'_0) P_{\eta'} z$ and $0<\eps<\eps_0$, let \[
		\Psi_{\eps,\pm}(upz)=\sigma(p)\int_{Upz}\psi_{c_1\eps,\pm}(u_{\t}pz)d\t.\]
		Then, by \eqref{eqn; continuity varphi},
		\begin{align}
		\sup_{w\in G_\eps}\Psi_{\eps,\pm}(wupz)
		&=\sup_{w\in P_\eps}\sigma(wp)\int_{Uwpz}\psi_{c_1\eps,+}(u_{\t}wpz)d\t\nonumber\\
		&\le(1+d'\eps)\Psi_{2\eps,\pm}.\label{eq:bound on sup Psi}
		\end{align}
		
		\noindent\textbf{Step 2: Bounding with PS measure.}
		
		Let \[P(f,\psi, x ; s) = \{p \in P_{\eta} : a_s\supp(f)x \cap B_U(r_0)pz \ne \emptyset\}.\] 	
		By \cite[Lemma 6.2]{Matrix coefficients}, there exists an absolute constant $c_1>0$ such that 
		\begin{align}
		&e^{(n-1)s}\int_{B_U(r)}\psi(a_s u_\t x)f(\t)d\t\label{eq:after 6.2}\\
		&\le(1+c_1\eps)\sum_{p\in P(f,\psi,x;s)}f_{c_1e^{-s}\eta}(a_{-s}pz)\int_{Upz}\psi_{c_1\eps,+}(u_{\t}pz)d\t.\nonumber
		\end{align}

		It now follows from \cite[Lemma 6.5]{Matrix coefficients}, \eqref{eq:lower bound on sigma p}, and \eqref{eq:bound on sup Psi} that there exists an absolute constant $c_2>0$ such that 
		\begin{align*}
		&e^{-\delta_\Gamma s}\sum_{p\in P(f,\psi,x;s)}f_{c_1e^{-s}\eta}(a_{-s}pz)\int_{Upz}\psi_{c_1\eps,+}(u_{\t}pz)d\t\\
		&\le\frac{(1+c_2\eps)(1+d'\eps)}{1-d\eps}\int_U\Psi_{2c_2\eps,+}(a_{s}u_{\t}x)f_{(c_1+c_2)e^{-s}\eps_0,+}(\t)d\ps_{x}(\t).
		\end{align*} Note that \eqref{eq:lower bound on sigma p} is needed because our definition of $\Psi_{\eps, +}$ is not identical to $\Psi$ as defined in \cite[Lemma 6.5]{Matrix coefficients}. The latter is bounded above by $\frac{1}{1-d\eps} \Psi_{\eps,+}$ by \eqref{eq:lower bound on sigma p}.
		
		Combining the above with \eqref{eq:after 6.2}, we get that there exist constants $c_3,c_4=c_4(\Gamma,\supp\psi)>0$ such that
		\begin{align*}
		&e^{(n-1-\delta_\Gamma)s}\int_{B_U(r)}\psi(a_s u_\t x)f(\t)d\t\\
		&\leq(1+c_4\eps)\int_U\Psi_{c_3\eps,+}(a_{s}u_{\t}x)f_{c_3e^{-s}\eps_0,+}(\t)d\ps_{x}(\t). 
		\end{align*}
		It follows from Theorem \ref{thm: PS translates thm} that for some constant $c_5=c_5(\Gamma,\supp\psi)>0$  
		\begin{align}
		&e^{(n-1-\delta_\Gamma)s}\int_{B_U(r)}\psi(a_s u_\t x)f(\t)d\t\nonumber\\
		&\leq(1+c_4\eps)\left(\ps_x(f_{c_3e^{-s}\eps_0,+})\bms(\Psi_{c_3\eps,+})+c_5S_{\ell'}(\Psi_{c_3\eps,+})S_{\ell'}(f_{c_3e^{-s}\eps_0,+})e^{-\kappa s}\right).\label{eq:after using 4.2}
		\end{align} 
		
		\noindent\textbf{Step 3: Bounding the error terms.}
		
		We now show how to bound the various error terms to obtain the desired conclusion.

		To compute $\bms(\Psi)$, we  use (\ref{eq:bms product structure}), \eqref{eq: contintuity feps pm from sobolev lemma}, and (\ref{eqn; bounds on sigma p}) to deduce that for some $c_6=c_6(\Gamma,\supp\psi)$, if $\xi = \eps^2$,
		\begin{align} 
		&\bms(\Psi_{c_3\eps,+})\nonumber\\
		&\le (1+d'\eps)\int_{P_{\eta'} z} \int_{B_U(r'_0)} \frac{1}{\ps_{pz}(B_U(r'_0)pz)} \int_{B_U(r'_0)pz}\psi_{c_1\eps,\pm}(u_{\t}pz)d\t d\ps_{pz}(\t)d\nu(pz)\nonumber\\
		&\le (1+d'\eps)\int_{P_{\eta'} z} \int_{B_U(r'_0)pz}\psi_{c_1\eps,\pm}(u_{\t}pz)d\t d\nu(pz)\nonumber\\
		&\le (1+d'\eps)\left(\br(\psi)+c_6\eps S_\ell(\psi)\right).\label{eq: bms into br}
		\end{align} 
		
		By Proposition \ref{prop:shadow lemma}, if $s$ is sufficiently large so that $r+c_3e^{-s}\eps_0\le 1$ (note that this requirement on $s$ depends only on $\Gamma$ and $\supp\psi$), we have that
		\begin{equation} \label{eq:after shadow}
		\ps_x(B_U(r+c_3e^{-s}\eps_0))\leq\ps_x(B_U(1))\ll_\Gamma e^{(n-1-\delta_\Gamma)d(\pi(\mathcal{C}_0),\pi(x))}.
		\end{equation}
		
		Hence, by \eqref{eq: contintuity feps pm from sobolev lemma} and \eqref{eq:after shadow}, we have
		\begin{align}
		\ps_x(f_{c_3e^{-s}\eps_0,+})-\ps_x(f)&\ll_{\Gamma}e^{- s}\eps_0 S_\ell(f)\ps_x(B_U(r+c_3e^{-s}\eps_0)) \nonumber \\
		&\ll_\Gamma e^{-s}\eps_0 S_\ell(f) e^{(n-1-\delta_\Gamma)d(\pi(\mathcal{C}_0),\pi(x))}. \label{eq:ps bound for f_ceps}
		\end{align}

		According to \cite[Lemma 2.4.7(a)]{KleinbockMargulis} and \eqref{eqn; ucty of psi, f}, if $\xi = \eps^2$ and \begin{equation}\eps = e^{-\frac{\kappa s}{2(n+4\ell)}},\label{eq:choosing eps in translates}\end{equation} then	
		\begin{align}
		S_{\ell'}(\Psi_{c_3\eps,+}) &\ll_{\Gamma} S_{\ell'}(\psi_{c_3\eps,+})S_{\ell'}(\sigma)\nonumber\\ 
		&\ll_{\Gamma} (m(P_{\eps-\xi}))\inv (\eps-\xi/2)^{n-1}\xi^{-{\ell'}-(n-1)/2}\eps^{-2\ell}S_{\ell}(\psi)\nonumber\\
		&\ll \eps^{-1-2\ell}\xi^{-{\ell'}-(n-1)/2}S_{\ell}(\psi) \nonumber\\
		&\le e^{\kappa s/2} S_\ell(\psi).
		\label{eqn; bound on S ell Psi} 
		\end{align}
		
		Using \eqref{eq:after using 4.2}, \eqref{eq: bms into br}, \eqref{eq:ps bound for f_ceps}, and \eqref{eqn; bound on S ell Psi}, we obtain
		\begin{align}
		&e^{(n-1-\delta_\Gamma)s}\int_{B_U(r)}\psi(a_su_\t x)f(\t)d\t - \ps_x(f)\br(\psi)\nonumber\\
		&\le (1+c_4\eps) \Bigg[d'\eps \ps_x(f)\br(\psi) + (1+d'\eps)\Bigg\{ c_6 \eps\ps_x(f)S_\ell(\psi)\nonumber\\
		&\hspace{1cm}+(e^{- s}\eps_0 \br(\psi)S_\ell(f) + c_6 e^{- s}\eps_0 \eps S_\ell(f) S_\ell(\psi))e^{(n-1-\delta_\Gamma)d(\pi(\mathcal{C}_0),\pi(x))}\Bigg\}\nonumber\\
		&\hspace{1cm}+ c_8 S_\ell(\psi)S_\ell(f)e^{-\kappa s/2} \label{eq:bound after subtracting ps f br psi}\Bigg]
		\end{align}
		
		These remaining error terms can be controlled as follows. Using \eqref{eq:after shadow}, we can deduce
		\begin{equation}
		\ps_x(f)\le\norm{f}_\infty\ps_x(B_U(r))\ll_{\Gamma}S_{\ell}(f)e^{(n-1-\delta_\Gamma)(\pi(\mathcal{C}_0),\pi(x))}. \label{eq:bound ps f}
		\end{equation} 
		We also have that \begin{equation}\label{eq:bound br psi by s ell}\br(\psi) \ll_{\Gamma, \supp\psi} S_\ell(\psi).\end{equation}
		
		Combining \eqref{eq:bound after subtracting ps f br psi}, \eqref{eq:bound ps f}, and \eqref{eq:bound br psi by s ell} implies\begin{align}
		&e^{(n-1-\delta_\Gamma)s}\int_{B_U(r)}\psi(a_su_\t x)f(\t)d\t - \ps_x(f)\br(\psi)\nonumber\\
		&\ll_{\Gamma, \supp \psi} S_\ell(\psi)S_\ell(f) \Bigg[ d'\eps + (1+d'\eps)\left(c_6\eps(1+e^{-s}\eps_0) + e^{- s}\eps_0\right)e^{(n-1-\delta_\Gamma)d(\pi(\mathcal{C}_0),\pi(x))} + c_8e^{-\kappa s/2} \Bigg]\label{eq:penultimate bound on translates}
		\end{align}
		
		Finally, by the choice of $\eps$ in \eqref{eq:choosing eps in translates} and because we may assume without loss of generality that $\kappa <1$, we obtain from \eqref{eq:penultimate bound on translates} that there exists $\kappa'<1$ such that \begin{align*}
		&e^{(n-1-\delta_\Gamma)s}\int_{B_U(r)}\psi(a_su_\t x)f(\t)d\t - \ps_x(f)\br(\psi)\\
		&\ll_{\Gamma,\supp\psi} S_\ell(\psi)S_\ell(f)e^{-\kappa' s + (n-1-\delta_\Gamma)d(\pi(\mathcal{C}_0),\pi(x))}.
		\end{align*}
		Recall that $d(\pi(\mathcal{C}_0),\pi(x))=\height(x)$. Thus, if we assume that $s \ge \frac{2(n-1-\delta_\Gamma)}{\kappa'}\height(x)$ (which means $s\gg_\Gamma \height(x)$), then $$e^{-\kappa' s+(n-1-\delta_\Gamma)\height(x)} \ll_\Gamma e^{-\kappa'/2 s},$$ which completes the proof.
	\end{proof}

	\section{Proof of Theorem \ref{thm; main PS}}\label{section: proof of PS thm}
	
	In this section, we prove Theorem \ref{thm; main PS}, which is restated below for the reader's convenience. The proof relies on the quantitative nondivergence result in Theorem \ref{thm:bounded functions} and Theorem \ref{thm: PS translates thm}.
	
	\begin{theorem}\label{restated main PS}
		For any $0<\eps<1$ and $s_0\ge 1$, there exist constants $\ell = \ell(\Gamma) \in \N$ and $\kappa = \kappa(\Gamma, \eps)>0$ satisfying: for every $\psi \in C_c^\infty(G/\Gamma)$, there exists $c = c(\Gamma,\supp\psi)$ such that every $x\in G/\Gamma$ that is $(\eps,s_0)$-Diophantine, and for every $T$ with $T^{1-\eps/2} \gg_{\Gamma} s_0$, \begin{align*}&\left|\frac{1}{\ps_x(B_U(T))} \int_{B_U(T)}\psi(u_\t x)d\ps_x(\t) - \bms(\psi)\right| \le c S_\ell(\psi) r^{-\kappa},\end{align*} where $S_\ell(\psi)$ is the $\ell$-Sobolev norm.
	\end{theorem}
	\begin{proof}
		Let $\beta>0$ satisfy the conclusion of Theorem \ref{thm:bounded functions} for $\eps$ and $s_0$. 
		Let $\kappa'>0$, $\ell\in\mathbb{N}$ satisfy the conclusion of Theorem \ref{thm: PS translates thm}. 

		Since $x$ is $(\eps,s_0)$-Diophantine, by Theorem \ref{thm:bounded functions}, for $T_0\gg_{\Gamma} s_0$ and $R \ge R_0$,  \begin{equation} \label{eqn; ps eqdistr nondiv} \ps_{x_0}(B_U(T_0)x_0 \cap {\mathcal{X}}(R)) \ll \ps_{x_0}(B_U(T_0)x_0)e^{-\beta R}, \end{equation}
		where \begin{equation} \label{eq:s0 def PS} s_\eps := \frac{\eps}{2}\log T, \quad T_0 := Te^{-s_\eps} = T^{1-\eps/2},\quad x_0 := a_{-s_\eps}x. \end{equation}
		
		By \eqref{eqn; reln bt a and u} and \eqref{eqn; ps scaling}, we have \[ \frac{1}{\ps_x(B_U(T))} \int_{B_U(T)} \psi(u_\t x)d\ps_x(\t) = \frac{1}{\ps_{x_0}(B_U(T_0))}\int_{B_U(T_0)}\psi(a_{s_\eps}u_\t x_0)d\ps_{x_0}(\t).\]
		
		Fix $R>R_0$, and define $$Q_0 = B_U(T_0)x_0 \cap \mathcal{C}(R).$$ By the definition of ${\mathcal C}(R)$, $$Q_0\subseteq \supp\bms.$$
		
		Let $\rho>0$ be smaller than half of the injectivity radius of $Q_0$. 
		
		First, by Lemma \ref{lem:partition of unity}, there exist $\{y : y \in I_0\}\subseteq Q_0$ and $f_y \in C_c^\infty(B_U(2\rho)y)$ satisfying 
		\begin{equation}\label{eq:sobolev fy PS}
		S_\ell(f_y)\ll \rho^{-\ell+n-1}
		\end{equation}
		and $$\sum\limits_{y} f_y = 1 \text{ on }E_1:= \bigcup_{y\in I_0} B_U(\rho)y\supseteq Q_0$$ and $0$ outside of $$E_2=\bigcup_{y\in I_0} B_U(2\rho)y.$$
		Observe that
		\begin{equation}
		Q_0\subseteq E_1\subseteq E_2 \subseteq B_U(T_0+2\rho)x_0.  
		\end{equation}
		Thus,
		\begin{align*}\int_{u_\t x_0 \in E_1} \psi(a_{s_\eps}u_\t x_0)d\ps_{x_0}(\t)\le \sum\limits_{y\in I_0} \int_{u_\t x_0 \in B_U(2\rho)y} \psi(a_{s_\eps}u_\t x_0)f_y(u_\t x_0) d\ps_{x_0}(\t) 
		\end{align*}
		
		
		Because $Q_0 \subseteq \supp\bms$, we may use Proposition \ref{prop:shadow lemma} to deduce that there exists $\lambda=\lambda(\Gamma)\ge1$ such that for any $y\in I_{0}$,  we have
		\begin{align*}
		\ps_y(B_U(\rho))&\ge\lambda^{-1}\rho^{\delta_\Gamma}e^{(k(y,\rho)-\delta_\Gamma)d(\pi(\mathcal{C}_0), \pi(a_{-\log \rho}y))}\\
		&\ge \lambda\inv \rho^{\delta_\Gamma}e^{-\delta_\Gamma d(\pi(\mathcal{C}_0), \pi(a_{-\log \rho}y))} \\
		&\ge \lambda\inv \rho^{\delta_\Gamma}e^{-\delta_\Gamma (-\log\rho)}e^{-\delta_\Gamma \height(y)} &\text{since } \rho<1\\
		&\gg_\Gamma \lambda\inv \rho^{2\delta_\Gamma}e^{-\delta_\Gamma \height(y)}\\
		&\ge \lambda\inv \rho^{2\delta_\Gamma}e^{-\delta_\Gamma R},
		\end{align*} where the last line follows by Lemma \ref{lem: height and distance to C0}.
		
		Since $e^{s_\eps} = T^{\eps/2}$, it follows from \eqref{eq:sobolev fy PS} and the above, that if we choose $\rho$ and $R$ such that 
		\begin{equation}
		\label{eqn; choosing rho in long orbits PS} 
		e^{\delta_\Gamma R}\rho^{n-1-\ell-2\delta_\Gamma}\ll_\Gamma
		T^{\eps\kappa' /4},
		\end{equation} then, by the choice of $f_y$, we have 
		\begin{equation}\label{eq:bounding sob by ps2}
		S_\ell(f_y)\ll\ps_y(B_U(\rho))e^{\kappa' s_\eps/2} \ll\ps_y(f_y)e^{\kappa' s_\eps/2} 
		\end{equation} where the implied constant is absolute.
		
		If we further assume that \begin{equation}\label{eq: bound on T from height} T \gg_\Gamma e^{2R/\eps}\end{equation} (with the implied constant coming from Theorem \ref{thm: PS translates thm}), then $s_\eps \gg_\Gamma R$, and by \eqref{eq:bounding sob by ps2}, Theorem \ref{thm: PS translates thm}, and Lemma \ref{lem: height and distance to C0}, there exist $c_1,c_2>0$ which depend only on $\Gamma$ and $\supp\psi$ such that 
		\begin{align*}
		&\sum_{y\in I_0} \int_{u_\t x_0 \in B_U(2\rho)y} \psi(a_{s_\eps}u_\t x_0)f_y(u_\t x_0) d\ps_{x_0}(\t)\\
		&\le \sum_{y\in I_0} \left(\bms(\psi)\ps_y(f_y)+c_1 S_\ell(\psi)S_\ell(f_y)e^{-\kappa' s_\eps}\right)\\
		&\le \sum_{y\in I_0} \ps_y(f_y)\left(\bms(\psi)+c_2 S_\ell(\psi)e^{-\kappa' s_\eps/2}\right)
		\end{align*} 
		By Lemma \ref{lem: height and distance to C0} and Theorem \ref{thm: nicer friendly}, there exists $c_3=c_3(\Gamma)>0$ such that if $T_0\gg s_0$, then there exist $\alpha=\alpha(\Gamma)>0$, $c_3=c_3(\Gamma)>0$ such that  
		\begin{align*}
		    \sum\limits_{y\in I_0} \ps_y(f_y) &\le \ps_{x_0}(B_U(T_0+2\rho)\cap \mathcal{C}(R+1))\\
		    &\ll_\Gamma (1+c_3(2\rho)^\alpha) \ps_{x_0}B_U(T_0)
		\end{align*} 
		Thus, we arrive at
		\begin{align}&\sum_{y\in I_0} \int_{u_\t x_0 \in B_U(2\rho)y} \psi(a_{s_\eps}u_\t x_0)f_y(u_\t x_0) d\ps_{x_0}(\t)\label{eq:using lemma 3.12 in new place}\\
		&\le\ps_{x_0}(B_U(T_0))\left(1+c_3\left(2\rho\right)^\alpha\right)\left(\bms(\psi) + c_2 S_\ell(\psi)e^{-\kappa' s_\eps/2}\right)\nonumber\end{align}

		
		Fix 
		\begin{equation}\label{eq:rho def PS}
		\kappa:=\kappa'\eps/4, \quad R>\frac{\kappa}{\beta}\log T,\quad 
		\rho<T^{-\frac{\kappa}{\alpha}}
		\end{equation}
		such that $\rho$ also satisfies the assumption of Theorem \ref{thm: nicer friendly}, and $R$ which satisfies \eqref{eqn; choosing rho in long orbits PS} and \eqref{eq: bound on T from height}.
		Thus, \eqref{eq:using lemma 3.12 in new place} and \eqref{eq:rho def PS} imply
		\begin{align}
		    &\frac{1}{\ps_{x_0}(B_U(T_0))} \int_{u_\t x_0 \in E_1}\psi(a_{s_\eps}u_\t x_0)d\ps_{x_0}(\t) - \bms(\psi) \nonumber\\
		    &\ll_{\Gamma,\supp\psi} S_\ell(\psi)T^{-\kappa},\label{eq:E1 calc PS}
		\end{align} where we have used that by \cite{sobolev}, $\|\psi\|_\infty \ll_{\supp \psi} S_\ell(\psi)$, so $\bms(\psi)\ll_{\supp\psi} S_\ell(\psi).$
		
		By (\ref{eqn; ps eqdistr nondiv}), \begin{align*}\int_{B_U(T_0)x_0\setminus E_1} \psi(a_{s_\eps}u_\t x_0)d\ps_{x_0}(\t) &\le \|\psi\|_\infty \ps_{x_0}(B_U(T_0)\setminus E_1)\\
		&\ll_{\supp\psi} S_\ell(\psi)\ps_{x_0}(B_U(T_0)x_0) e^{-\beta R}\\
		&\ll_{\supp\psi} S_\ell(\psi)\ps_{x_0}(B_U(T_0)x_0) T^{-\kappa},\end{align*} where we have again used that by \cite{sobolev}, $\|\psi\|_\infty \ll_{\supp \psi} S_\ell(\psi)$. Combining the above with (\ref{eq:E1 calc PS}) implies that \begin{align*}&\frac{1}{\ps_{x_0}(B_U(T_0))} \int_{B_U(T_0)} \psi(a_{s_\eps}u_\t x_0)d\ps_{x_0}(\t) - \bms(\psi) \\
		&\ll_{\Gamma,\supp\psi} S_\ell(\psi)T^{-\kappa}\end{align*}
		
		
		The lower bound is obtained similarly, as in the proof of Theorem \ref{thm; main}.
	\end{proof}

	\section{Proof of Theorem \ref{thm; main}}\label{section; long orbits}
	
	In this chapter we will prove Theorem \ref{thm; main} using Theorem \ref{cor; br translates restated}. We will use a partition of unity argument for a cover of the intersection of $B_U(r)x$ with a fixed compact set by small balls centered at PS-points.
	
	We will need the following lemma.
	\begin{lemma}\label{lem:claim A}
		There exists an absolute constant $c>0$ satisfying the following: for $x\in X$, $y\in Ux$, $\psi\in C_c^\infty(X)$ supported on an admissible box of diameter smaller than $1$, $0<\rho<\inj(y)$,  $f\in C_c^{\infty}(B_U(\rho)y)$ such that $0\le f\le1$, and $s>0$ which satisfies $ce^{-s}\eps<\rho$, we have\[
		e^{(n-1-\delta_\Gamma)s}\int_{Ux}\psi(a_s u_\t y)f(u_\t y)d\t\ll_{\Gamma,\supp\psi} S_\ell(\psi)\ps_y(B_U(2\rho)y), \]
		where $\ell\in\N$ satisfies the conclusion of Lemma \ref{lem:choosing ell}.
	\end{lemma} 
	
	\begin{proof}Assume that for $0<\eps_0,\eps_1<1$, $\psi$ is supported on the admissible box $B_U(\eps_0)P_{\eps_1}z$ for  $z\in X$. Without loss of generality, we may assume that $\psi$ is non-negative. Fix $y\in Ux$. 
		
		
		
		For small $\eta>0$, $h\in G_{\eta}\supp(\psi)$, and $p\in P$, let 
		\[\psi_{\eta,+}(h):=\sup_{w\in G_\eta}\psi(wh),\quad\Psi_{\eta,+}(ph):=\int_{Uph}\psi_{\eta,+}(u_\t ph)d\t\] and for $upz\in B_U(\eps_0)P_{\eps_1}z$ let \[\tilde{\Psi}_{\eta,+}(upz):=\frac{1}{\ps_{pz}(B_U(\eps_0)pz)}\Psi_{\eta,+}(pz). \]
		
		By choice of $\ell$, for any $\eta>0$ and $h\in G_\eta\supp(\psi)$, \[	\left|\psi_{\eta,+}(h)-\psi(h)\right|\ll \eta S_{\ell}(\psi), \]
		and \[
		\left|\psi(z)\right|\le S_{\infty,0}(\psi)\ll S_\ell(\psi),\]
		where the implied constants depend on $\supp\psi$. Since the diameter of $\supp\psi$ is smaller than $1$, we may assume that the implied constant in the above is absolute. Then, for any $u\in U$ such that $a_s uy=u'pz\in B_U(\eps_0)P_{\eps_1}z$ and $0<\eta<1$, we have
		\begin{align}
		\left|\tilde\Psi_{\eta,+}(a_s u y)\right|
		&=\left|\frac{1}{\ps_{pz}(B_U(\eps_0)pz)}\int_{U pz}\psi_{\eta,+}(u_\t pz)d\t\right|\nonumber\\
		&=\frac{\leb_{pz}(B_U(\eps_0)pz)}{\ps_{pz}(B_U(\eps_0)pz)} S_\ell(\psi)\nonumber\\
		& \ll S_\ell(\psi),\label{eq:tilde psi bound}
		\end{align}
		where the implied constant depends only on $\supp\psi$. 
		
		For small $\eta>0$ and $uy\in B_U(\eta+\eps_0) y$, let \[
		f_{\eta,+}(uy):=\sup_{w\in B_U(\eta)}f(wu y)\]
		
		Using Lemma 6.2 and Lemma 6.5 in \cite{Matrix coefficients}, we get that for some absolute constant $c>0$,
		\begin{align*}    
		e^{(n-1-\delta_\Gamma)s}\int_{B_U(\rho)y }\psi(a_{s}u_\t y)f(u_\t y)d\t
		&\ll\int_{U}\tilde{\Psi}_{c\eps,+}(a_su_\t y)f_{ce^{-s}\eps,+}(u_\t y)d\ps_y(\t)\\
		& \le\int_{B_U(\rho+ce^{-s}\eps)y}\tilde{\Psi}_{c\eps,+}(a_su_\t y)d\ps_y(\t),
		\end{align*}
		where the implied constant is absolute. 
		Then, by \eqref{eq:tilde psi bound} we get
		\begin{align*}    
		e^{(n-1-\delta_\Gamma)s}\int_{B_U(\rho)y }\psi(a_{s}u_\t y)f(u_\t y)d\t
		&\ll_{\supp\psi} \ps_y(B_U(\rho+ce^{-s}\eps)) S_\ell(\psi)\\
		&\le \ps_y(B_U(2\rho))S_\ell(\psi).
		\end{align*}
		
	\end{proof}
	
	We are now ready to prove Theorem \ref{thm; main}. For the reader's convenience, we restate that theorem below:
	\begin{theorem}\label{main restated}
		For any $0<\eps<1$ and $s_0\ge 1$, there exist $\ell=\ell(\Gamma)\in\N$ and $\kappa=\kappa(\Gamma,\eps)>0$ satisfying: for every $\psi \in C_c^\infty(G/\Gamma)$, there exists $c = c(\Gamma,\supp\psi)$ such that for every $x\in G/\Gamma$ that is $(\eps,s_0)$-Diophantine, and for all $T$ such that $T^{1-\eps/2} \gg_{\Gamma,\supp\psi} s_0$,\[\left|\frac{1}{\ps_x(B_U(T))} \int_{B_U(T)}\psi(u_\t x)d\t - \br(\psi)\right| \le c S_\ell(\psi)r^{-\kappa},\] where $S_\ell(\psi)$ is the $\ell$-Sobolev norm .
	\end{theorem}

	\begin{proof}
		We keep the notation of \S \ref{sec:quantitative nondivergence}. By an argument similar to the proof of Theorem \ref{thm: PS translates thm}, we may assume that $\psi$ is supported on an admissible box. Because $\psi$ is compactly supported, we may also assume $\psi \ge 0$.
		
		Let $\beta>0$ satisfy the conclusion of Theorem \ref{thm:bounded functions} for $\eps$ and $s_0$. 
		Let $\kappa'>0$, $\ell\in\mathbb{N}$ satisfy the conclusion of Theorem \ref{cor; br translates restated}.
		
		Since $x$ is $(\eps,s_0)$-Diophantine, by Theorem \ref{thm:bounded functions}, for $T_0\gg_{\Gamma} s_0$ and $R\ge R_0$, we have \begin{equation} \label{eqn; using schapira} \ps_{x_0}(B_U(T_0)x_0\cap {\mathcal{X}}(R)) \ll \ps_{x_0}(B_U(T_0)x_0)e^{-\beta R},\end{equation} where 
		\begin{equation}\label{eq:s0 def}
		s_\eps:=\frac{\eps}{2}\log T,\quad x_0 := a_{-s_\eps}x,\quad\mbox{and}\quad T_0 = T^{1-\frac{\eps}{2}}.
		\end{equation}
		Observe that by \eqref{eqn; reln bt a and u}, \eqref{eqn; leb scaling}, and \eqref{eqn; ps scaling}, \[\frac{1}{\ps_x(B_U(T))} \int_{B_U(T)}\psi(u_\t x)d\t=\frac{e^{(n-1-\delta)s_\eps}}{\ps_{x_0}(B_U(T_0))} \int_{B_U(T_0)}\psi(a_{s_\eps}u_\t x_0)d\t.\] 
		
		Fix $R>R_0$ and define $$Q_0:= B_U(T_0)x_0 \cap {\mathcal{C}}(R).$$ Since for any $R\geq R_0$ the set ${\mathcal{C}}(R)$ is in the convex core of $\H^n /\Gamma$,  
		\begin{equation}\label{eq:Q0 in supp of BMS}Q_0\subseteq\supp\bms.\end{equation}
		Let $\rho>0$ be smaller than half of the injectivity radius of $Q_0$.
		
		First, by Lemma \ref{lem:partition of unity}, there exist $\{y : y \in I_0\}\subseteq Q_0$ and $f_y \in C_c^\infty(B_U(2\rho)y)$ satisfying 
		\begin{equation}\label{eq:sobolev fy}
		S_\ell(f_y)\ll \rho^{-\ell+n-1}
		\end{equation}
		and $$\sum\limits_{y} f_y = 1 \text{ on }E_1:= \bigcup_{y\in I_0} B_U(\rho)y\supseteq Q_0$$ and $0$ outside of $$E_2=\bigcup_{y\in I_0} B_U(2\rho)y.$$
		
		By replacing references to Theorem \ref{thm: PS translates thm} with references to Theorem \ref{cor; br translates restated}, the exact same argument as in the proof of Theorem \ref{thm; main PS} will establish that for $T \gg_\Gamma e^{2R/\eps}$ 
		and
		\begin{equation}\label{eq:rho def}
	\kappa=\frac{\beta\eps}{2},\quad R=\frac{\kappa\log T}{\beta},\quad \rho \le T^{-\kappa/\alpha}
		\end{equation}
		we get that if we assume without loss of generality that $\kappa'<2\beta$ and also that $T\gg_\Gamma 1$,
		\begin{align}\label{eq:E1 calc}
		&\frac{e^{(n-1-\delta_\Gamma) s_\eps}}{\ps_{x_0}(B_U(T_0)x_0)} \int_{u_\t x_0 \in E_1} \psi(a_{s_\eps}u_\t x_0)d\t-\br(\psi) \ll_{\Gamma,\supp\psi} S_\ell(\psi)T^{-\kappa}.
		\end{align}
		
		We now want to bound the integral over $B_U(T_0)x_0\setminus E_1$. Using Lemma \ref{lem:partition of unity} again, we may deduce that there exist $\{y : y \in I_1\}\subseteq B_U(T_0)x_0\setminus E_1$ and $f_y \in C_c^\infty(B_U(\rho/4)y)$ satisfying $\sum\limits_{y\in I_1} f_y = 1$  on $ \bigcup\limits_{y\in I_1} B_U(\rho/8)y$ and $0$ outside of $$\bigcup\limits_{y\in I_1} B_U(\rho/4)y.$$ In particular, by the definition of $E_1$, we have \[
		E_3:=\bigcup\limits_{y\in I_1} B_U(\rho/2)y\subseteq (B_U(T_0)x_0\setminus Q_0)\cup (B_U(T_0+\rho/2)x_0\setminus B_U(T_0))x_0. \]
		Using Lemma \ref{lem:claim A}, 
		we arrive at 
		\begin{align*}
		&e^{(n-1-\delta_\Gamma)s_\eps}\int_{B_U(T_0) \setminus E_1} \psi(a_{s_\eps}u_\t x_0)d\t\\
		&\le e^{(n-1-\delta_\Gamma)s_\eps}\sum_{y\in I_1} \int_{B_U(\rho/2)y} \psi(a_{s_\eps}u_\t y)f_y(u_\t  y)d\t\\
		&\ll\sum_{y\in I_1}  S_{\ell}(\psi)\ps_y(B_U(\rho/2)y)\\
		&\le S_{\ell}(\psi)(\ps_{x_0}(B_U(T_0)x_0\setminus Q_0)+\ps_{x_0} ((B_U(T_0+\rho/2)\setminus B_U(T_0))x_0)).
		\end{align*}
		Thus, by Theorem \ref{thm: nicer friendly} there exists $\alpha=\alpha(\Gamma)>0$ such that using equations \eqref{eqn; using schapira}, \eqref{eq:rho def}, we arrive at
		\begin{align*}
		&e^{(n-1-\delta_\Gamma)s_\eps}\int_{B_U(T_0) \setminus E_1} \psi(a_{s_\eps}u_\t x_0)d\t \\
		&\ll_{\Gamma,\supp\psi} S_{\ell}(\psi)\ps_{x_0}(B_U(T_0)x_0)\left(e^{-\beta R}+\rho^\alpha \right)\\
		& \ll_{\Gamma,\supp\psi} S_{\ell}(\psi)\ps_{x_0}(B_U(T_0)x_0)T^{-\kappa}.
		\end{align*}
		Using \eqref{eq:E1 calc}, we may now deduce \[
		\frac{e^{(n-1-\delta_\Gamma)s_\eps}}{\ps_{x_0}(B_U(T_0)x_0)}\int_{B_U(T_0)} \psi(a_{s_\eps}u_\t x_0)d\t-\br(\psi)\ll_{\Gamma,\supp\psi} S_{\ell}(\psi)T^{-\kappa}.\]
		
		
		On the other hand, define $$Q_1:= B_U(T_0-2\rho)x_0 \cap {\mathcal{C}}(R).$$
		As before, according to Lemma \ref{lem:partition of unity}, there exist $\{y : y \in I_1\}\subseteq Q_1$ and $f_y \in C_c^\infty(B_U(2\rho)y)$ satisfying $$S_\ell(f_y)\ll \rho^{-\ell+n-1}$$ and $$\sum_{y\in I_1} f_y = 1 \text{ on }E_4:= \bigcup_{y\in I_1} B_U(\rho)y$$ and $0$ outside of \begin{equation}
		\bigcup_{y\in I_1} B_U(2\rho)y\subseteq B_U(T_0)x_0.  
		\end{equation}
		Hence,
		\begin{align*}\int_{u_\t x_0 \in B_U(T_0)x_0} \psi(a_{s_\eps}u_\t x_0)d\t\ge \sum_{y\in I_1} \int_{u_\t x_0 \in B_U(2\rho)y} \psi(a_{s_\eps}u_\t x_0)f_y(u_\t x_0) d\t.
		\end{align*}
		By Theorem \ref{cor; br translates restated} we have \[
		\int_{u_\t x_0 \in B_U(T_0)x_0} \psi(a_{s_\eps}u_\t x_0)d\t
		\ge \ps_{x_0}(B_U(T_0-2\rho))(\br(\psi) - c_2 S_\ell(\psi)e^{-\kappa' s_\eps/2}),\]
		and by Theorem \ref{thm: nicer friendly} we arrive at  \[
		\ge\ps_{x_0}(B_U(T_0))\left(1-c_3(2\rho)^\alpha\right)\left(\br(\psi) - c_2 S_\ell(\psi)e^{-\kappa' s_\eps/2}\right).\]
		Hence, \eqref{eq:rho def} implies 
		\begin{equation*}
		\frac{e^{(n-1-\delta_\Gamma) s_\eps}}{\ps_{x_0}(B_U(T_0)x_0)} \int_{u_\t x_0 \in E_1} \psi(a_{s_\eps}u_\t x_0)d\t 
		-\br(\psi)\gg_{\Gamma,\supp\psi} S_\ell(\psi)T^{-\kappa}. 
		\end{equation*}
	\end{proof}
	
	\begin{remark}
	The dependence of $T$ on $\supp\psi$ in the previous proof arises from Theorem \ref{cor; br translates restated}, through the quantity $s_\eps$. Upon closer inspection, one can verify that this means $T$ depends on $\supp\psi$ through the maximum height of elements in $\supp\psi$. In particular, we may choose a larger compact set containing $\supp\psi$ and have $T$ depend on that compact set, rather than $\supp\psi$ specifically.
	\end{remark}

	\section{Appendix: Friendliness of the PS measure}\label{friendliness appendix}
		For simplicity, in this section we work in the Poincar\'e ball models of hyperbolic geometry $\D^n$, instead of $\H^n$. Recall that $\D^n$ and $\H^n$ are isometric via the Cayley transform. 
	
	Denote by $d_E$ the Euclidean metric on $\R^m$. 
	For a subset $S\subseteq\R^m$ and $\xi>0$, let \[
	\mathcal{N}(S,\xi)=\{x\in \R^m\::\: d_E(x,S)\le\xi\}.\]
	For $v\in\R^m$ and $r>0$, let \[
	B(v,r)=\left\{u\in\R^m\::\:d_E(u,v)\le r\right\}\]
	be the Euclidean ball of radius $r$ around $v$. 
	
	\begin{definition}\label{defn:friendly}
	    Let $\mu$ be a measure defined on $\R^m$. 
		\begin{enumerate}
			\item
			$\mu$ is called \textbf{Federer} (respectively, \textbf{doubling}) if for any $c>1$, there exists $k_1>0$ such that for all $v\in\supp(\mu)$ and $0 <\eta\leq1$ (respectively, $\eta>0$),\[
			\mu(B(v,{c\eta}))\le k_1\mu(B(v,\eta)). \]
			\item $\mu$ is called \textbf{decaying and nonplanar} 
			if there exist $\alpha,c_2>0$ such that for all $v\in\supp\mu$, $\xi>0$, $0<\eta\le1$, and every affine hyperplane $L \subseteq \R^{n}$,  $$\mu(\mathcal{N}(L,\xi\norm{d_L}_{\mu,B(v,\eta)})\cap B(v,\eta)) \le c_2{\xi}^\alpha \mu (B(v,\eta)),$$ 
			where \[ \norm{d_L}_{\mu,B(v,\eta)}:=\sup\left\{d(\textbf{y},L)\::\:\textbf{y}\in B(v,\eta)\cap\supp\mu\right\}. \]
			\item $\mu$ is called \textbf{friendly} if it is Federer, decaying, and nonplanar.
		\end{enumerate}
	\end{definition}

	In the case that all cusps have maximal rank (which vacuously includes the case of convex cocompact $\Gamma$), a stronger statement holds, see \S\ref{section: appendix}.
	

	\begin{theorem}\cite[Theorem 1.9]{friendly}\label{thm:friendly}
		Assume $\Gamma$ is geometrically finite and Zariski dense. Then the PS-densities $\left\{\nu_x\right\}_{x\in \D^n}$ are friendly. Moreover, in this case, the constants in Definition \ref{defn:friendly} only depend on $\Gamma$. 
	\end{theorem}
	Note that, as in \cite[Definition 1.1(1.3)]{friendly}, using closed thickenings, one obtains Definition \ref{defn:friendly}(2) by combining the separate definitions of decaying and of nonplanar from \cite{friendly}.
	The above result for the case $\Gamma$ is convex cocompact was proved in \cite[Theorem 2]{friendly 2}.
	
	In this section, we will prove the results in \S\ref{subsection; friendly measures}. In particular, because of the shadow lemma, Proposition \ref{prop:shadow lemma}, we will see that the leafwise PS measures $\{\ps_x\}$ satisfy a stronger condition than that of friendliness. In general, we will begin by proving a statement for $\nu_o$, then for $\ps_x$ when $x^+\in\Lambda(\Gamma),$ and then finally a nicer statement for $x \in \supp\bms.$
	
	The next lemma and subsequent corollaries are necessary to move between these measures.
	
	
	As in \S \ref{section; PS and Lebesgue}, we fix $o\in\D^n$. 
	For any $x\in \D^n$ define the \textbf{Gromov distance} at $x$ of $\xi,\eta\in\partial\D^n$ by \[
	d_x(\xi,\eta)=\exp\left(-\frac{1}{2}\beta_{\xi}(x,y)-\frac{1}{2}\beta_\eta(x,y)\right),\]
	where $y$ is on the ray joining $\xi$ and $\eta$. For any $x\in\D^n$, $\xi\in\partial\D^n$, and $r>0$ let \[
	B_x(\xi,r):=\left\{\eta\in\partial\D^n\::\:d_x(\xi,\eta)\le r\right\}. \]
	
	For $v\in \opt1(\mathbb{D}^n)$, denote by $\operatorname{Pr}_{v^-}:Uv\rightarrow\partial\D^n\setminus\{v^-\}$ the projection $w\mapsto w^+$. 
	
	The next lemma follows from \S 1.6 in \cite{Kaimanovich}, and \cite[Lemma 2.5, Theorem 3.4]{Schapira}.

	\begin{lemma} \label{lem: Schapira}
		There exist constants $\alpha_0>0$, $c>1$ such that for all $g\in G$ and $0<\eta\le \alpha_0$, we have \[B_{\pi(g)}\left(g^+,c^{-1}\eta \right)\subseteq \operatorname{Pr}_{g^-}\left(B_U\left(\eta\right)g\right)\subseteq B_{\pi(g)}\left(g^+,c \eta\right).\]
	\end{lemma}
	
	According to \cite[Lemma 3.5.1]{gromov} for any $\xi,\eta\in\partial\D^n$ 
	\begin{equation}\label{eq:gromove to euclid}
	d_o(\xi,\eta)=\frac{1}{2}d_E(\xi,\eta). 
	\end{equation}
	Using the triangle inequality on the hyperbolic distance and the definition of the Busemann function, one can show that for any $x\in\D^n$ and $\xi,\eta\in\partial\D^n$
	\begin{equation} \label{eq:dx to do}
	e^{-d(o,x)}\le \frac{d_x(\xi,\eta)}{d_o(\xi,\eta)}\le e^{d(o,x)}. \end{equation}
	
		The following is a direct corollary of \eqref{eq:gromove to euclid}, \eqref{eq:dx to do}, and Lemma \ref{lem: Schapira}.
	\begin{corollary}\label{cor:proj to euclid old}
		There exist constants $\alpha_0>0$, $c>1$ such that for all $g\in G$ and $0<\eta\le \alpha_0$, we have \[B\left(g^+,c^{-1}e^{-d(o,\pi(g))} \eta\right)\subseteq \operatorname{Pr}_{g^-}\left(B_U\left(\eta\right)g\right)\subseteq B\left(g^+,ce^{d(o,\pi(g))}\eta\right).\]
	\end{corollary}
	
	The next corollary will be necessary to obtain a nonplanarity result for $\ps_x.$ It follows from Corollary \ref{cor:proj to euclid old} by covering the hyperplane with small balls using the fact that $\eta\le 1$ to uniformly bound the $d(o,\pi(g'))$'s with $d(o,\pi(g))$, where $g'$ is the center of one of the balls in this cover.
	
	\begin{corollary}\label{cor: proj to euclid for planes}
	    Let $\alpha_0$ be as in Corollary \ref{cor:proj to euclid old}. There exists a constant $c>1$ so that for every $g \in G$, every $0<\xi<\eta\le\alpha_0,$ and every hyperplane $L$ in $\R^{n-1}$, there exists a hyperplane $L'$ in $\partial(\H^n)$ so that\begin{align*}&\mathcal{N}\left(L', c\inv e^{-d(o,\pi(g))}\xi\right)\cap B\left(g^+, c\inv e^{-d(o,\pi(g))}\eta\right)\\
	    &\subseteq\operatorname{Pr}_{g^-}\left(\mathcal{N}(L,\xi)\cap B_U(\eta)g\right) \\
	    &\subseteq\mathcal{N}\left(L', ce^{d(o,\pi(g))}\xi\right)\cap B\left(g^+, c e^{d(o,\pi(g))}\eta\right)\end{align*} 
	\end{corollary}
\begin{proof}
	Let $\{g_i\}_{i \in I}$ be chosen so that $$\mathcal{N}(L,\xi)=\bigcup\limits_{i \in I} B_U(\xi)g_i.$$ By Corollary \ref{cor:proj to euclid old}, \[B\left(g_i^+,c^{-1}e^{-d(o,\pi(g_i))} \xi\right)\subseteq \operatorname{Pr}_{g_i^-}\left(B_U\left(\xi\right)g_i\right)\subseteq B\left(g_i^+,ce^{d(o,\pi(g_i))}\xi\right).\]
	
	Let $$I_2 = \{i : B_U(\xi)g_i\cap B_U(\eta)g \ne \emptyset\}.$$ Observe that for $i \in I_2,$ $g_i \in B_U(2)g.$ Then $$d(o,\pi(g_i)) \le d(o,\pi(g))+d(\pi(g),\pi(g_i)),$$ and $g_i = u_\t g$ for $|\t|\le 2$. Thus, \begin{align*}
	    d(\pi(g),\pi(g_i)) &= d(\pi(g),\pi(u_\t g)) \\
	    &=d(g(o), u_\t g(o))\\
	\end{align*} This implies that there exists a constant $\hat{c}>1$ that is uniform for all $g \in G$ so that for all $i \in I_2$, \begin{equation*}
	    \label{eq: bound on distance in proj to euclid for planes}
	   \hat{c}\inv e^{-d(o,\pi(g))} \ll_\Gamma e^{-d(o,\pi(g_i))} \ll_\Gamma e^{d(o,\pi(g_i))} \ll_\Gamma \hat{c}e^{d(o,\pi(g))}.
	\end{equation*} Thus, for all $i \in I_2$, we have \begin{equation*}
	   B\left(g^+,\hat{c}\inv c^{-1}e^{-d(o,\pi(g))} \xi\right)\subseteq \operatorname{Pr}_{g_i^-}\left(B_U\left(\xi\right)g_i\right)\subseteq B\left(g^+,\hat{c}ce^{d(o,\pi(g))}\xi\right).
	\end{equation*} Hence, for every $i \in I_2,$ \begin{align*}
	     \label{eq: inclusion for i in I2 in planes projection} &B\left(g^+,\hat{c}\inv c^{-1}e^{-d(o,\pi(g))} \xi\right) \cap B\left(g^+,\hat{c}\inv c\inv e^{-d(o,\pi(g))}\eta\right)\\
	     &\subseteq \operatorname{Pr}_{g^-}\left(B_U(\xi)g_i\cap B_U(\eta)g\right) \\
	     &\subseteq B\left(g^+,\hat{c} ce^{d(o,\pi(g))} \xi\right) \cap B\left(g^+,\hat{c}c e^{d(o,\pi(g))}\eta\right).
	\end{align*} The result then follows by taking the union over all $i \in I_2$.
	\end{proof}	

\subsection{The PS measure is Federer}
	
	In this section, we prove more specific Federer statements for $\nu_o$ and $\ps_x$. 
	
	\begin{lemma}
	    \label{lem: effective doubling of density}
	    There exists a constant $\sigma \ge \delta_\Gamma$ depending only on $\Gamma$ such that for any $\lambda \in \Lambda(\Gamma), \eta>0$ and $c \ge 1,$ we have that \[\nu_o(B(\lambda,c\eta)) \ll_\Gamma c^\sigma \nu_o(B(\lambda,\eta)).\]
	\end{lemma}
	\begin{proof}
	We will prove this for the balls $B_o(\lambda,c\eta),$ and $B_o(\lambda,\eta)$ using the Gromov distance. It then immediately follows for the Euclidean balls $B(\lambda,c\eta)$ and $B(\lambda,\eta)$ by the Federer condition and \eqref{eq:gromove to euclid}. 
	
	Let $\{\lambda_t\}_{t\ge 0}$ be a geodesic ray joining $o$ to $\lambda$. By the shadow lemma for $\nu_o$, \cite[Theorem 2]{StratmannVelani} (see also \cite[Theorem 3.2]{Schapira}), we have that for any $\eta >0,$
	\begin{align}\label{eq:shadow lemma}
	    \eta^{\delta_\Gamma}e^{(k(\lambda_{-\log \eta})-\delta_\Gamma)d(\pi(\mathcal{C}_0),\lambda_{-\log \eta})}&\ll_\Gamma\nu_o\left(B_o\left(\lambda,\eta \right)\right)\nonumber\\
	    &\ll_{\Gamma}\eta^{\delta_\Gamma}e^{(k(\lambda_{-\log \eta})-\delta_\Gamma)d(\pi(\mathcal{C}_0),\lambda_{-\log \eta})}.
	\end{align} 
	Here, $k(\lambda_{-\log\eta})$ denotes the rank of the cusp that $\lambda_{-\log\eta}$ lies in; if it is in $\pi(\mathcal{C}_0)$, it is defined to be zero. (Recall the definition of $\mathcal{C}_0$ from \S\ref{section: thick thin decomposition}.) Note also that we have absorbed a constant depending on $\operatorname{diam} \pi(\mathcal{C}_0)$ (hence only on $\Gamma$) in order to write the distance from $\pi(\mathcal{C}_0)$ rather than from the fixed reference point $o$.
	
	It follows from \eqref{eq:shadow lemma} that it is enough to show that for some $\sigma\ge\delta_\Gamma$,
	\begin{align*}
	   \nu_o\left(B_o\left(\lambda,c\eta \right)\right)&\ll_{\Gamma}(c\eta)^{\delta_\Gamma}e^{(k(\lambda_{-\log c\eta})-\delta_\Gamma)d(\pi(\mathcal{C}_0),\lambda_{-\log c\eta})}\\
	   &\ll_{\Gamma}c^\sigma\eta^{\delta_\Gamma}e^{(k(\lambda_{-\log \eta})-\delta_\Gamma)d(\pi(\mathcal{C}_0),\lambda_{-\log\eta})}\\
	   &\ll_{\Gamma}c^\sigma\nu_o\left(B_o\left(\lambda,\eta \right)\right). 
	\end{align*}
	Equivalently, it is enough to show that
	\begin{align}\label{eq:need for Federer}
	  & (k(\lambda_{-\log c\eta})-\delta_\Gamma)d(\pi(\mathcal{C}_0),\lambda_{-\log c\eta})-(k(\lambda_{-\log \eta})-\delta_\Gamma)d(\pi(\mathcal{C}_0),\lambda_{-\log\eta})\\
	  &\ll_{\Gamma}(\sigma-\delta_\Gamma)\log c.\nonumber
    \end{align}
	
	
	\textbf{Case 1:} Assume $k(\lambda_{-\log c\eta})\le k(\lambda_{-\log \eta})$.
	
	Then \begin{align*}
	    & (k(\lambda_{-\log c\eta})-\delta_\Gamma)d(\pi(\mathcal{C}_0),\lambda_{-\log c\eta})-(k(\lambda_{-\log \eta})-\delta_\Gamma)d(\pi(\mathcal{C}_0),\lambda_{-\log\eta})\\
	    &\le (k(\lambda_{-\log\eta})-\delta_\Gamma)(d(\pi(\mathcal{C}_0),\lambda_{-\log c\eta}) - d(\pi(\mathcal{C}_0),\lambda_{-\log\eta})) \\
	    &\le (k(\lambda_{-\log\eta})-\delta_\Gamma)\log c\\
	    &\le (n-1-\delta_\Gamma)\log c.\\
	\end{align*}
	
	\textbf{Case 2:} $k(\lambda_{-\log c\eta})> k(\lambda_{-\log \eta})$ and $k(\lambda_{-\log \eta})=0$.
	
	Then, $d(\pi(\mathcal{C}_0),\lambda_{-\log \eta})=0$ and 
	\begin{align*}
	    0<d(\pi(\mathcal{C}_0),\lambda_{-\log c\eta })&\le d(\pi(\mathcal{C}_0),\lambda_{-\log \eta })+d(\lambda_{-\log \eta },\lambda_{-\log c\eta })\le \log c. 
	\end{align*}
	Therefore,
	\begin{align*}
	    &(k(\lambda_{-\log c\eta})-\delta_\Gamma)d(\pi(\mathcal{C}_0),\lambda_{-\log c\eta})-(k(\lambda_{-\log \eta})-\delta_\Gamma)d(\pi(\mathcal{C}_0),\lambda_{-\log\eta})\\
	    &\le(k(\lambda_{-\log c\eta})-\delta_\Gamma)d((\pi(\mathcal{C}_0),\lambda_{-\log c\eta})\\
	    &\le(k(\lambda_{-\log c\eta})-\delta_\Gamma)\log c\\
	    &\le (n-1-\delta_\Gamma)\log c.\\
	\end{align*}

	\textbf{Case 3:} Assume $k(\lambda_{-\log c\eta})> k(\lambda_{-\log \eta})$ and $k(\lambda_{-\log \eta})>0$. In particular, $\lambda_{-\log\eta}$ and $\lambda_{-\log c\eta}$ are in two different cusps, and hence there exists $1<r<c$ such that $\lambda_{-\log r\eta}\in \pi(\mathcal{C}_0)$. Then, 
	\begin{align*}
	    &d(\pi(\mathcal{C}_0),\lambda_{-\log\eta})\le d(\lambda_{-\log r\eta},\lambda_{-\log\eta})\le \log r\le\log c\\
	    &d(\pi(\mathcal{C}_0),\lambda_{-\log c\eta})\le d(\lambda_{-\log r\eta},\lambda_{-\log c\eta})\le \log (c/r)\le\log c
	\end{align*}
	Note that since $k(\lambda_{-\log c\eta})\ge 2$, we have $\delta_\Gamma>1$, because $\delta_\Gamma > k/2$, where $k$ is the maximal cusp rank. We arrive at
	\begin{align*}
	    (k(\lambda_{-\log{\eta}})-\delta_\Gamma)d(\pi(\mathcal{C}_0),\lambda_{-\log\eta}) &\ge (1-\delta_\Gamma) d(\pi(\mathcal{C}_0),\lambda_{-\log\eta})\\
	    &\ge (1-\delta_\Gamma)\log c\\
	    (k(\lambda_{-\log{ c\eta}})-\delta_\Gamma)d(\pi(\mathcal{C}_0,\lambda_{-\log c\eta})
	    &\le (n-1-\delta_\Gamma) d(\pi(\mathcal{C}_0),\lambda_{-\log c\eta})\\ &\le (n-1-\delta_\Gamma)\log c.
	\end{align*}
	It follows that
	\begin{align*}
	    &(k(\lambda_{-\log c\eta})-\delta_\Gamma)d(\pi(\mathcal{C}_0),\lambda_{-\log c\eta})-(k(\lambda_{-\log \eta})-\delta_\Gamma)d(\pi(\mathcal{C}_0),\lambda_{-\log\eta})\\
	    &\le(n-1-\delta_\Gamma)\log c-(1-\delta_\Gamma)\log c\\
	    &\le(n-2)\log c. 
	\end{align*}
	
	Thus, choosing $$\sigma = \max\{n-1-\delta_\Gamma, n-2\} + \delta_\Gamma$$ completes the proof.
	\end{proof}

	When $c<1$, we obtain a similar result, with a slightly more involved argument.
	
		\begin{lemma}
	    \label{lem: effective small doubling of density}
	    There exists a constant $\sigma>0$ depending only on $\Gamma$ such that for any $\lambda \in \Lambda(\Gamma), \eta>0$ and $0<c< 1,$ we have that \[\nu_o(B(\lambda,c\eta)) \ll_\Gamma c^\sigma \nu_o(B(\lambda,\eta)).\]
	\end{lemma}
	\begin{proof} The proof is extremely similar to that of Lemma \ref{lem: effective doubling of density}.
	
	
	By the shadow lemma, as in the proof of Lemma \ref{lem: effective doubling of density}, it is enough to show that for some $\sigma>0$,
	\begin{align}\label{eq:need for Federer}
	  & (k(\lambda_{-\log c\eta})-\delta_\Gamma)d(\pi(\mathcal{C}_0),\lambda_{-\log c\eta})-(k(\lambda_{-\log \eta})-\delta_\Gamma)d(\pi(\mathcal{C}_0),\lambda_{-\log\eta})\\
	  &\ll_{\Gamma}(\delta_\Gamma-\sigma) |\log c|.\nonumber
    \end{align}

	\textbf{Case 1:} Assume $k(\lambda_{-\log c\eta})\le k(\lambda_{-\log \eta})$.
	
	Then \begin{align*}
	    & (k(\lambda_{-\log c\eta})-\delta_\Gamma)d(\pi(\mathcal{C}_0),\lambda_{-\log c\eta})-(k(\lambda_{-\log \eta})-\delta_\Gamma)d(\pi(\mathcal{C}_0),\lambda_{-\log \eta})\\
	    &\le (k(\lambda_{-\log c\eta})-\delta_\Gamma)(d(\pi(\mathcal{C}_0),\lambda_{-\log \eta}) - d(\pi(\mathcal{C}_0),\lambda_{-\log c\eta})) \\
	    &\le |k(\lambda_{-\log\eta})-\delta_\Gamma||\log c|
	\end{align*}
    Let $k$ be the maximal cusp rank.  Since $   |k-\delta_\Gamma|<\delta_\Gamma$, we get that\[
    \sigma:=\delta_\Gamma-|k-\delta_\Gamma|>0\]
	satisfies the claim. \\
	
	\textbf{Case 2:} $k(\lambda_{-\log c\eta})> k(\lambda_{-\log \eta})$ and $k(\lambda_{-\log \eta})=0$.
	
	Then, $d(\pi(\mathcal{C}_0),\lambda_{-\log \eta})=0$ and 
	\begin{align*}
	    0<d(\pi(\mathcal{C}_0),\lambda_{-\log c\eta })&\le d(\pi(\mathcal{C}_0),\lambda_{-\log \eta })+d(\lambda_{-\log \eta },\lambda_{-\log c\eta })\le |\log c|. 
	\end{align*}
	Therefore,
	\begin{align*}
	    &(k(\lambda_{-\log c\eta})-\delta_\Gamma)d(\pi(\mathcal{C}_0),\lambda_{-\log c\eta})-(k(\lambda_{-\log \eta})-\delta_\Gamma)d(\pi(\mathcal{C}_0),\lambda_{-\log\eta})\\
	    &\le(k(\lambda_{-\log c\eta})-\delta_\Gamma)d(\pi(\mathcal{C}_0),\lambda_{-\log c\eta})\\
	    &\le|k(\lambda_{-\log c\eta})-\delta_\Gamma||\log c|,
	\end{align*}
	and the claim follows as in Case 1. \\
	
	\textbf{Case 3:} Assume $k(\lambda_{-\log c\eta})> k(\lambda_{-\log \eta})$ and $k(\lambda_{-\log \eta})>0$. In particular, $\lambda_{-\log\eta}$ and $\lambda_{-\log c\eta}$ are in two different cusps, and hence there exists $c<r<1$ such that $\lambda_{-\log r\eta}\in \pi(\mathcal{C}_0)$. Then since $r<1$,
	\begin{align}
	    &d(\pi(\mathcal{C}_0),\lambda_{-\log\eta})\le d(\lambda_{-\log r\eta},\lambda_{-\log\eta})\le |\log r| \label{eq: log eta leq r}\\
	    &d(\pi(\mathcal{C}_0),\lambda_{-\log c\eta})\le d(\lambda_{-\log r\eta},\lambda_{-\log c\eta})\le \log(r/c)\label{eq: log c eta leq log r over c}
	\end{align}
	Note that since $k(\lambda_{-\log c\eta})\ge 2$, we have $\delta_\Gamma>1$. By \eqref{eq: log eta leq r} and \eqref{eq: log c eta leq log r over c}, we arrive at
	\begin{align}
	    (k(\lambda_{-\log{\eta}})-\delta_\Gamma)d(\pi(\mathcal{C}_0),\lambda_{-\log\eta}) &\ge (1-\delta_\Gamma) d(\pi(\mathcal{C}_0),\lambda_{-\log\eta}) \label{eq: bound k log eta by log c}\\
	    &\ge (\delta_\Gamma - 1)\log c\\
	    (k(\lambda_{-\log{ c\eta}})-\delta_\Gamma)d(\pi(\mathcal{C}_0),\lambda_{-\log c\eta}) 
	    &\le \max\{0,\log(r/c) (k(\lambda_{-\log c\eta})-\delta_\Gamma)\}\nonumber
	\end{align}

	We now have two cases. First, assume that $\log(r/c)(k(\lambda_{-\log c\eta})-\delta_\Gamma)\le0$. Then $k(\lambda_{-\log c\eta})-\delta_\Gamma \le 0,$ so by \eqref{eq: bound k log eta by log c}, we have that
	\begin{align*}
	    &(k(\lambda_{-\log c\eta})-\delta_\Gamma)d(\pi(\mathcal{C}_0),\lambda_{-\log c\eta})-(k(\lambda_{-\log \eta})-\delta_\Gamma)d(\pi(\mathcal{C}_0),\lambda_{-\log\eta})\\
	    &\le - (\delta_\Gamma -1 )\log c\\
	    &= (\delta_\Gamma - 1 )|\log c|
	\end{align*}
	
	Now, assume that $\log(r/c) (k(\lambda_{-\log c\eta})-\delta_\Gamma)>0$, i.e. that $k(\lambda_{-\log c\eta})-\delta_\Gamma>0$. Then it follows from \eqref{eq: log c eta leq log r over c} that
	\begin{align}
	    &(k(\lambda_{-\log c\eta})-\delta_\Gamma)d(\pi(\mathcal{C}_0),\lambda_{-\log c\eta})-(k(\lambda_{-\log \eta})-\delta_\Gamma)d(\pi(\mathcal{C}_0),\lambda_{-\log\eta})\nonumber\\
	    &\le (k(\lambda_{-\log c\eta})-\delta_\Gamma)\log(r/c)-(\delta_\Gamma-1)\log r.\label{eq: 2nd case small doubling}
	\end{align} Now, consider two further cases: $k(\lambda_{-\log c\eta})-\delta_\Gamma > \delta_\Gamma - 1$ or $k(\lambda_{-\log c\eta})-\delta_\Gamma \le \delta_\Gamma -1.$ In the first case, \eqref{eq: 2nd case small doubling} is bounded above by \[
	    (\delta_\Gamma -1)\log (r/c) - (\delta_\Gamma - 1)\log r = -(\delta_\Gamma -1)\log c = (\delta_\Gamma-1)|\log c|.\] In the second case, note that \eqref{eq: 2nd case small doubling} is equal to $$(k-2\delta_\Gamma+1)\log r - (k(\lambda_{-\log c\eta}-\delta_\Gamma) \log c,$$ and our assumption implies that the first term is negative. Thus, an upper bound is $$- (k(\lambda_{-\log c\eta}-\delta_\Gamma) \log c = (k(\lambda_{-\log c\eta}-\delta_\Gamma) |\log c| \le (k-\delta_\Gamma)|\log c|,$$ where $k$ is the maximal cusp rank, as before. Note that $k-\delta_\Gamma < \delta_\Gamma$ because $\delta_\Gamma>2k$ always holds.

	Thus, choosing $$\sigma = \min\{\delta_\Gamma-|k-\delta_\Gamma|,1\}$$ completes the proof.
	\end{proof}

		Using Lemma \ref{lem: Schapira}, we obtain the following quantitative Federer-like statement for $\left\{\ps_x\right\}_{x^+\in\Lambda(\Gamma)}$:

	
		\begin{corollary}\label{cor: effective federer leafwise}
	    There exists constants $\sigma_1 = \sigma_1(\Gamma)\ge \delta_\Gamma$, $\sigma_2 = \sigma_2(\Gamma)>0$ which satisfy the following: let $x \in G$ be such that $x^+\in \Lambda(\Gamma)$. Then for $c>0$ and $\eta \ll_\Gamma c\inv e^{-\height(x)}$, we have that \[\ps_x(B_U(c\eta))\ll_\Gamma \max\{c^{\sigma_1}, c^{\sigma_2}\} e^{2(\delta_\Gamma+\sigma_1)\height(x)}\ps_x(B_U(\eta)).\]
	\end{corollary}
    \begin{proof}
    Fix $g\in G$ which satisfies $x=g\Gamma$ and $\height(x)=d(\pi(\mathcal{C}_0),\pi(g))$. By \eqref{eq: inj radius and height}, $\inj(x)$ and $\height(x)$ are related, so that for $\eta \ll_\Gamma c\inv\height(x),$ $$\ps_g(B_U(c\eta))=\ps_x(B_U(c\eta)).$$
	
 For any $0<\eta\le 1$ and $u_\t \in B_U(\eta)$, we have that
	\begin{align*}
	    |\beta_{(u_\t g)^+}(o,u_\t g(o))| 
	   &\le d(u_{\t}\inv (o),g(o)) \\
	   &\le d(u_\t\inv(o),o)+d(o,g(o)) \\
	    &\le 2\operatorname{diam}(B_U(1)\pi(\mathcal{C}_0)) +\height(x).
	\end{align*} 
	

The above gives a bound on the Busemann function for the following when $\eta \le 1$:
\begin{align}
    &e^{-\delta_\Gamma \height(x)}\nu_o(\operatorname{Pr}_{g^-}(B_U(\eta))) \nonumber\\
    &\ll_\Gamma  \ps_g(B_U(\eta)) = \int_{\t \in B_U(\eta)} e^{\delta_{\Gamma}\beta_{(u_\t g)^+}(o,u_\t g(o))} d\nu_o((u_\t g)^+) \label{eq: lower bounded busemann for effective federer}\\
    &\ll_\Gamma e^{\delta_\Gamma \height(x)} \nu_o\left(\operatorname{Pr}_{g^-}(B_U(\eta))\right).\label{eq: upper bounded busemann for effective federer}
\end{align}

    Assume $c\ge 1$. By Lemma \ref{lem: Schapira} and Lemma \ref{lem: effective doubling of density}, we have that
    \begin{align}
        \nu_o(B(g^+,\eta))&= \nu_o(B(g^+,(\tilde{c}e^{\height(x)}\tilde{c}\inv  e^{-\height(x)}\eta)) \nonumber \\
        &\ll_\Gamma (\tilde{c}e^{\height(x)})^{\sigma_1} \nu_o(B(g^+,\tilde{c}e^{-\height(x)}\eta)).\label{eq: step in proving effective federer for ps}
    \end{align}
    
    Let $\tilde{c}>1$ be as in Corollary \ref{cor:proj to euclid old}. Then as long as $$\eta \le \tilde{c}\inv c\inv e^{-\height(x)},$$ we have the following:
    \begin{align*}
        \ps_g(B_U(c\eta)) &\ll_\Gamma e^{\delta_\Gamma \height(x)} \nu_o(\operatorname{Pr}_{g^-}(B_U(c\eta)) &\text{by \eqref{eq: upper bounded busemann for effective federer}}\\ 
        &\ll_\Gamma e^{\delta_\Gamma \height(x)}\nu_o(B(g^+,\tilde{c}e^{\height(x)}c\eta)) &\text{by Corollary \ref{cor:proj to euclid old}}\\
        &\ll_\Gamma c^\sigma e^{(\delta_\Gamma+{\sigma_1})\height(x)}\nu_o(B(g^+,\eta)) &\text{by Lemma \ref{lem: effective doubling of density}} \\
        &\ll_\Gamma c^{\sigma_1} e^{(\delta_\Gamma+2{\sigma_1})\height(x)}\nu_o(B(g^+,\tilde{c}\inv e^{-\height(x)}\eta) &\text{by \eqref{eq: step in proving effective federer for ps}}\\
        &\ll_{\Gamma}  c^{\sigma_1} e^{(\delta_\Gamma+2{\sigma_1})\height(x)}\nu_o(\operatorname{Pr}_{g^-}(B_U(\eta)))&\text{by Corollary \ref{cor:proj to euclid old}}\\
        &\ll_{\Gamma} c^{\sigma_1} e^{2(\delta_\Gamma+{\sigma_1})\height(x)}\ps_g(B_U(\eta)) &\text{by \eqref{eq: lower bounded busemann for effective federer}},
    \end{align*} which completes the proof in this case.
    
    The case $0<c<1$ can be shown in a similar way using Lemma \ref{lem: effective small doubling of density}. 
    \end{proof}

	When $x \in \supp\bms,$ a flowing argument with $\{a_{-s}:s\ge 0\}$ allows us to remove the restriction that $\eta$ must be small in a way that depends on $\height(x).$ More precisely, we obtain:
	
		\begin{corollary}\label{cor:ps measure is doubling} If $\Gamma$ is geometrically finite and Zariski dense, then for any $x\in\supp\bms$, the measure $\ps_x$ is doubling, and the constants only depend on $\Gamma$. More precisely, there exist constants $\sigma_1 = \sigma_1(\Gamma)\ge \delta_\Gamma$, $\sigma_2 = \sigma_2(\Gamma)>0$ such that for every $c >0$, every $x \in \supp\bms$ and every $T>0$, \[\ps_x(B_U(cT))\ll_\Gamma \max\{c^{\sigma_1}, c^{\sigma_2}\}\ps_x(B_U(T)).\]	    
	\end{corollary}
	\begin{proof}
	    On a geometrically finite quotient, there exists a compact set $\Omega_0\subset X$ such that for every $x \in X$ with $x^-\in\Lambda_r(\Gamma)$, there exists a sequence $s_n \to \infty$ such that $a_{-s_n}x \in \Omega_0$. 

	    Because $\Omega_0$ depends only on $\Gamma$, the height of any point in $\Omega_0$ is bounded by a constant depending only on $\Gamma$. Thus, by Corollary \ref{cor: effective federer leafwise}, for all $x \in \Omega_0 \cap \supp\bms$ with $x^-\in\Lambda_r(\Gamma)$ and for all $\eta \ll_\Gamma c\inv,$ we have that \begin{equation}\ps_x(B_U(c\eta))\ll_\Gamma \max\{c^{\sigma_1}, c^{\sigma_2}\} \ps_x(B_U(\eta)).\label{eq:doubling proof in compact}\end{equation}
	    
	    Now, fix $x \in \supp\bms$ with $x^-\in\Lambda_r(\Gamma)$. Let $T \ge 0,$ and let $s>0$ be sufficiently large so that $e^{-s}T \ll_\Gamma c\inv$ and $a_{-s}x \in \Omega_0.$ Then \begin{align*}
	        \ps_x(B_U(cT)) &= e^{\delta_\Gamma s}\ps_{a_{-s}x}(B_U(ce^{-s}T))\\
	        &\ll_\Gamma \max\{c^{\sigma_1}, c^{\sigma_2}\} e^{\delta_\Gamma s}\ps_{a_{-s}x}(B_U(e^{-s} T)) &\text{by \eqref{eq:doubling proof in compact}}\\
	        &\ll_\Gamma \max\{c^{\sigma_1}, c^{\sigma_2}\} \ps_{a_{-s}x}(B_U(T)),
	    \end{align*} so the result holds for $x^- \in \Lambda_r(\Gamma).$
	
	Since $x\mapsto\ps_x$ is continuous (see Lemma \ref{lem; continuity g to psg}) and the set of $x$ with $x^-\in\Lambda_r(\Gamma)$ is dense in the set of points $y\in X$ which satisfy $y^-\in\Lambda(\Gamma)$, the result then follows for all $x\in\supp\bms.$
	\end{proof}

	\subsection{Non-planarity of the PS measure}
	
	For a subset $S\subseteq\R^{n-1}$ and $\xi>0$, let \[
	\mathcal{N}_U(S,\xi)=\{u_\t\in U\::\: \exists \textbf{s}\in S\text{ such that }\norm{\t-\textbf{s}}<\xi\}.\]
	
	In the following, we use the shadow lemma for $\nu_o$ to obtain a stronger version of nonplanarity than that in Definition \ref{defn:friendly}. From this, we will see that the PS measures when $\Gamma$ is geometrically finite, satisfies a non-planarity-like property. More specifically, the bound we get is independent of the hyperplane, but the size of $\eta$ must be restricted in a way that depends on $\height(x)$, and a factor of $\height(x)$ will appear.
	
	\begin{theorem}\label{thm:local bound ps density}
	There exist $\theta=\theta(\Gamma)\ge 1$, $\alpha=\alpha(\Gamma)>0$ which satisfy the following. For any $w \in \H^n$,  $\lambda\in\Lambda(\Gamma)$, $0<\eta\le 1$, and $\xi>0$, we have
	\[\nu_w(\mathcal{N}(L,\xi)\cap B(\lambda,\eta))\ll_{\Gamma} e^{2\delta_\Gamma d(o,w)}\frac{\xi^\alpha}{\eta^\theta}\nu_w(B(\lambda,\eta)). \] 
	\end{theorem}
	\begin{proof}
	First, we show the result for $o$.
	
	According to \cite[Lemma 3.8]{friendly} there exists $\beta>0$ such that for any $\eta>0$, and any affine hyperplane $L\subset \R^n$ we have 
	\begin{equation}\label{eq:nu upper bound}
	   \nu_o(\mathcal{N}(L,\eta))\ll_{\Gamma} \eta^\beta.
	\end{equation}
	
	For $\lambda\in\Lambda(\Gamma)$ and for $t\in\R$, let $\lambda_t$ be the unit speed geodesic ray from ${o}$ to $\lambda$.
	It follows from the shadow lemma for $\nu_o$ (see \cite[Theorem 2]{StratmannVelani}, also \cite[Theorem 3.2]{Schapira})
	that for any $\eta>0$, we have
	\begin{align*}
	   \nu_o\left(B_o\left(\lambda,\eta \right)\right)&\gg_{\Gamma}\eta^{\delta_\Gamma}e^{(k(\lambda_{-\log\eta})-\delta_\Gamma)d(o,\lambda_{-\log\eta})} 
	\end{align*}
	where $k(\lambda_{-\log\eta})$ is the rank of the cusp containing $\lambda_{-\log\eta}$ (see \S\ref{section: thick thin decomposition}). 
	It follows from the fact that $k(\lambda_{-\log\eta})\ge 0$ and $d({o},\lambda_{-\log\eta})\le -\log\eta$, that	\begin{align*}	   \nu_o\left(B_o\left(\lambda,\eta \right)\right)&\gg_{\Gamma}\eta^{2\delta_\Gamma}.  	\end{align*}
	Since $\nu_o$ is Federer (by Theorem \ref{thm:friendly}), using \eqref{eq:gromove to euclid} we arrive at the same bound for Euclidean balls (with the implied constant changing): 
	\begin{align}\label{eq:nu lower bound} 
	   \nu_o\left(B\left(\lambda,\eta \right)\right)&\gg_{\Gamma}\eta^{2\delta_\Gamma}.
	\end{align}

	Note that by the definition of $\norm{d_L}_{\nu_o,B(\lambda,\eta)}$, $$B(\lambda,\eta)\cap\supp\nu_o\subset \mathcal{N}\left(L,\norm{d_L}_{\nu_o,B(\lambda,\eta)}\right).$$
	It then follows from \eqref{eq:nu upper bound} and \eqref{eq:nu lower bound} that \[
	\eta^{\delta_\Gamma}\ll_{\Gamma}\left(\norm{d_L}_{\nu_o,B(\lambda,\eta)}\right)^\beta. \]
	Hence
	\begin{equation}\label{eq:dL bound}
	    \norm{d_L}_{\nu_o,B(\lambda,\eta)}\gg_{\Gamma}\eta^{2\delta_\Gamma/\beta}. 
	\end{equation}
	
	According to Theorem \ref{thm:friendly}, the PS density is friendly. In particular, it is decaying and nonplanar, 
	so there exists $\alpha> 0$ such that for all $\lambda\in \Lambda(\Gamma)$, $0<\eta\le1$, $\xi>0$, an affine hyperplane $L\subset \R^n$, and $B=B(\lambda,\eta)$, we have
	\begin{equation}\label{eq:decaying and nonplanar bound}
	    \nu_o(\mathcal{N}(L,\xi\norm{d_L}_{B})\cap B)\ll_\Gamma \xi^{\alpha}\nu_o(B). 
	\end{equation}
    The claim now follows for $o$ from \eqref{eq:dL bound} and \eqref{eq:decaying and nonplanar bound} by taking $\theta=2\delta_\Gamma/\beta$.
    
    Second, we show the result for a general $w\in \H^n$. Note that \[
    e^{-\delta_\Gamma d(o,w)}\ll_{\Gamma} e^{-\delta_\Gamma\beta_\lambda(w,o)}\ll_{\Gamma} e^{\delta_\Gamma  d(o,w)}. \]
    Thus, using this and the fact that $\{\nu_w\}_{w\in\H^n}$ is a conformal density satisfying (\ref{eq: conformal density def}), we arrive at
    \begin{align*}
        \nu_w(\mathcal{N}(L,\xi\eta^\theta)\cap B(\lambda,\eta))&\ll_{\Gamma} e^{\delta_\Gamma d(o,w)}\nu_o(\mathcal{N}(L,\xi\eta^\theta)\cap B(\lambda,\eta))\\
        &\ll_{\Gamma}e^{\delta_\Gamma d(o,w)}\xi^\alpha\nu_o(B(\lambda,\eta))\\
        &\ll_{\Gamma} e^{2\delta_\Gamma d(o,w)}\xi^\alpha\nu_w(B(\lambda,\eta))
    \end{align*}
    
    Last, note that by taking $\xi=\eta^{1-\theta}$, we conclude that $\theta\ge 1$.
\end{proof}

\begin{proposition}\label{prop: PS points friendliness intermed step with planes} Let $\Gamma$ be geometrically finite and Zariski dense. There exist constants $\alpha = \alpha(\Gamma)>0, \omega=\omega(\Gamma) \ge 0, $ and $\theta = \theta(\Gamma)>\alpha$ satisfying the following: for any $x \in G/\Gamma$ with $x^+\in \Lambda(\Gamma)$, and for every $\xi>0$ and  $0<\eta\ll_\Gamma e^{-\height(x)}$, we have that for every hyperplane $L$, \[\ps_x(\mathcal{N}_U(L,\xi)\cap B_U(\eta)) \ll_\Gamma e^{\omega \height(x)} \frac{\xi^\alpha}{\eta^\theta}\ps_x(B_U(\eta)).\]
\end{proposition}
\begin{proof}
Let $\alpha=\alpha(\Gamma),\theta=\theta(\Gamma)>0$ satisfy the conclusion of Theorem \ref{thm:local bound ps density}, and $c'>1$ satisfy the conclusion of Corollary \ref{cor:proj to euclid old}. Fix $g\in G$ which satisfies $x=g\Gamma$ and $\height(x)=d(\pi(\mathcal{C}_0),\pi(g)).$

By the same argument as in the proof of Corollary \ref{cor: effective federer leafwise} to bound the Busemann function when $\eta \le 1$, we obtain 
\begin{align*}
&e^{-\delta_\Gamma \height(x)}\nu_o\left(\operatorname{Pr}_{g^-}(\mathcal{N}(L,\xi)x\cap B_U(\eta)x)\right) \\
&\ll_\Gamma    \ps_g(\mathcal{N}(L,\xi)\cap B_U(\eta)) = \int_{\t \in \mathcal{N}(L,\xi)\cap B_U(\eta)} e^{\delta_{\Gamma}\beta_{(u_\t g)^+}(o,u_\t g(o))} d\nu_o((u_\t g)^+) \\
    &\ll_\Gamma e^{\delta_\Gamma \height(x)} \nu_o\left(\operatorname{Pr}_{g^-}(\mathcal{N}(L,\xi)x\cap B_U(\eta)x)\right)
\end{align*}

	Thus, for $\eta \ll_{\Gamma} e^{-\height(x)}$ (so that $ce^{d(o,\pi(x))}\eta \le 1$ below, and we stay within the injectivity radius at $x$, using \eqref{eq: inj radius and height}), we have that
	\begin{align*}
		&\ps_x(\mathcal{N}_U(L,\xi)\cap B_U(\eta))\\
		& \ll_{\Gamma}e^{\delta_\Gamma \height(x)}\nu_o\left(\operatorname{Pr}_{g^-}(\mathcal{N}(L,\xi)\cap B_U(\eta))\right)\\
		&\ll_{\Gamma}e^{\delta_\Gamma \height(x)}\nu_o\left(\mathcal{N}(L', ce^{d(o,\pi(x))}\xi)\cap B\left(g^+, c e^{d(o,\pi(x))}\eta\right)\right) &\text{by Corollary \ref{cor: proj to euclid for planes}}\\
		& \ll_{\Gamma}e^{\delta_\Gamma \height(x)} \left(\frac{\xi({c}e^{d(o,\pi(x))})^{1-\theta}}{\eta^\theta}\right)^{\alpha}\nu_o\left( B\left(g^+, ce^{d(o,\pi(x))}\eta\right)\right) & \text{by Theorem \ref{thm:local bound ps density}}\\
		&\ll_\Gamma e^{\delta_\Gamma \height(x)}\left(\frac{\xi(e^{d(o,\pi(x))})^{1-\theta}}{\eta^\theta}\right)^{\alpha} e^{d(o,\pi(x))\sigma}\nu_o(B(g^+,\eta)) &\text{by Lemma \ref{lem: effective doubling of density}}\\
		&\ll_\Gamma e^{\delta_\Gamma \height(x)}\left(\frac{\xi(e^{d(o,\pi(x))})^{1-\theta}}{\eta^\theta}\right)^{\alpha} e^{2d(o,\pi(x))\sigma}\nu_o(B(g^+,c\inv e^{-d(o,\pi(x)}\eta)) &\text{by Corollary \ref{cor:proj to euclid old}}\\
		&\ll_\Gamma e^{2\delta_\Gamma \height(x)+(\sigma+(1-\theta)\alpha) d(o,\pi(x))}\left(\frac{\xi}{\eta^\theta}\right)^{\alpha} \ps_x(B_U(\eta))\\
		&\ll_\Gamma e^{(2\delta_\Gamma + \sigma+(1-\theta)\alpha)\height(x)} \left(\frac{\xi}{\eta^\theta}\right)^\alpha \ps_x(B_U(\eta))\\
		&\ll_\Gamma e^{\omega\height(x)}\frac{\xi^\alpha}{\eta^{\theta'}} \ps_x(B_U(\eta)),
	\end{align*} where $$\omega = \max\{2\delta_\Gamma + \sigma+(1-\theta)\alpha,0\},\quad \theta'=\theta\alpha.$$
\end{proof}

\subsection{Absolutely Friendliness of the PS-measure}\label{section: appendix}

When all cusps are of maximal rank, the PS measure is \emph{absolutely friendly}, and stronger results hold. Note that if $\Gamma$ is convex cocompact, then there are no cusps, so this additional assumption is vacuously true.
	
	\begin{definition}\label{defn:absolutely friendly}
		Let $\mu$ be a measure defined on $\R^m$. 
		\begin{enumerate}
			\item $\mu$ is called \textbf{absolutely decaying} (respectively, \textbf{globally absolutely decaying}) if there exist $\alpha,c_2>0$ such that for all $v \in\supp\mu$, all $0<\xi<\eta\le1$ (respectively, $0<\xi<\eta$), and every affine hyperplane $L \subseteq \R^{n}$,   $$\mu(\mathcal{N}(L,\xi)\cap B(v,\eta)) \le c_2\left(\frac{\xi}{\eta}\right)^\alpha \mu (B(v,\eta)).$$
			\item $\mu$ is called \textbf{absolutely friendly} (respectively, \textbf{globally friendly}) if it is Federer (respectively, doubling) and absolutely decaying (respectively, globally absolutely decaying).
		\end{enumerate}
	\end{definition}
	
	It is easy to see that if a measure $\mu$ is globally friendly, then it is also absolutely friendly. 
	
	
According to \cite[Theorem 2]{friendly 2} if $\Gamma$ is convex cocompact or \cite[Theorem 1.12]{friendly} if $\Gamma$ is geometrically finite, $\nu_o$ is absolutely friendly if and only if all cusps have maximal rank.
	
	\begin{theorem} \label{thm: absolutely friendly for ps}  Assume that $\Gamma$ is Zariski dense and either convex cocompact or geometrically finite with all cusps having maximal rank. Then the PS-measures  $\left\{\ps_x\right\}_{x^-\in\Lambda(\Gamma)}$ are globally friendly, and the constants in Definition \ref{defn:absolutely friendly} only depend on $\Gamma$ (in particular, they do not depend on $x$).
	\end{theorem}

	This follows by a flowing argument, similar to the doubling results for $x \in \supp\bms$ proven before. The key difference is observed by contrasting Definition \ref{defn:absolutely friendly}(1) with Theorem \ref{thm:local bound ps density}: when the powers of $\xi, \eta$ match, a flowing argument may be used for BMS points. When they do not match, one introduces a power corresponding to how far one flows with $a_{-s}$.

	\begin{corollary}
	Assume that $\Gamma$ is Zariski dense and either convex cocompact or geometrically finite with all cusps having maximal rank. There exists $0<\alpha=\alpha(\Gamma)<1$ such that for any $x\in\supp\bms$, $T>0$, and $0<\xi\le T$, we have \[\frac{\ps_x(B_U(T+\xi))}{\ps_x(B_U(T))} -1 \ll_{\Gamma}\left(\frac{\xi}{T}\right)^{\alpha}.\]\end{corollary}
		
	\begin{proof}
		Let $c_1=c_1(\Gamma),c_2=c_2(\Gamma)>0$ and $\alpha=\alpha(\Gamma)>0$ satisfy the conclusion of Definition \ref{defn:absolutely friendly} for $\ps_x$ and $k=2$.
		
		It follows from the geometry of $B_U(\xi+\eta)x-B_U(\eta)x$ that there exist $L_1,\dots,L_m$, where $m$ only depends on $n$, such that\[ 
		B_U(\xi+T)x-B_U(T)x\subseteq\bigcup_{i=1}^{m} \mathcal{N}_U(L_i,2\xi).\]
		Then, by Definition \ref{defn:absolutely friendly}, we have 
		\begin{align*}
		\frac{\ps_x(B_U(\xi+T))}{\ps_x(B_U(T))} -1 & =\frac{\ps_x(B_U(\xi+T)-B_U(T))}{\ps_x(B_U(T))}\\
		&\le m c_2\left(\frac{\xi}{T}\right)^{\alpha}\frac{\ps_x\left(B_U\left(\xi+T\right)\right)}{\ps_x(B_U(T))}\\
		&\le m c_1 c_2\left(\frac{\xi}{T}\right)^{\alpha}.
		\end{align*}\end{proof}


\begin{thebibliography}{99}
		\bibitem{sobolev} T. Aubin, \emph{Nonlinear analysis on manifolds,} GM 252 Springer, 1982.
		\bibitem{bowditch} B. H. Bowditch, \emph{Geometrical finiteness for hyperbolic groups}, J. Funct. Anal., 113(2):245–317, 1993. 
		\bibitem{burger} M. Burger, \emph{Horocycle flow on geometrically finite surfaces,} Duke Math. J., \textbf{61} (1990), 779--803.
		\bibitem{DaniSmillie} S. G. Dani, J. Smillie, Uniform distribution of horocycle orbits for Fuchsian groups, Duke Math. J.,51(1):185194, 1984.
		\bibitem{gromov}T. Das, D. S. Simmons, and M. Urba\'{n}ski, \emph{Geometry and dynamics in Gromov hyperbolic metric spaces: with an emphasis on non-proper settings}, http://arxiv.org/abs/1409.2155, preprint 2014. 
		\bibitem{friendly} T. Das, L. Fishman, D. Simmons, M. Urba\'{n}ski,
		\emph{Extremality and dynamically defined measures, part II: Measures from conformal dynamical systems}, 1-38. doi:10.1017/etds.2020.46
		\bibitem{Edwards} S. C. Edwards, \emph{Effective Equidistribution of the Horocycle Flow on Geometrically Finite Hyperbolic Surfaces}, International Mathematics Research Notices, rnz263,
		\bibitem{EdwardsOh} S. Edwards, H. Oh, \emph{Spectral gap and exponential mixing on geometrically finite hyperbolic manifolds}, arXiv:2001.03377v1.
		\bibitem{FlaminioForni} L. Flaminio, G. Forni, \emph{Invariant distributions and time averages for horocycle flows}, Duke Math. J. 119 (2003), No. 3, pp. 465-526.
		\bibitem{Furstenberg} H. Furstenberg, \emph{The unique ergodicity of the horocycle  flow}, Recent advances in topological dynamics (Proc. Conf., Yale Univ., New Haven, Conn.,1972; in honor of Gustav Arnold Hedlund), pp. 95–115. Lecture Notes in Math.,318, Springer, Berlin, 1973.
		\bibitem{HillVelani} R. Hill, S. Velani, \emph{The Jarn\'{i}k-Besicovitch theorem for geometrically finite Kleinian groups}, Proceedings of the LMS, 77 (1998), pp. 524-550.
		\bibitem{Hirai}T. Hirai, \emph{On irreducible representations of the Lorentz group of $n$-th order}. Proc. Japan. Acad., Vol 38., 258-262, 1962.
		\bibitem{hormander} L. H\"ormander, \emph{The Analysis of Linear Partial Differential Operators, vol. I}, Distribution Theory and Fourier Analysis, Grundlehren Math. Wiss. vol. 256, 2nd edn. Springer, Berlin (1990).
		\bibitem{Kaimanovich}V. A. Kaimanovich, \emph{Invariant measures for the geodesic flow and measures at infinity on negatively curved manifolds}, Ann. I.H.P., Physique Th\'eorique 53, n. 4 (1990) 361-393.
		\bibitem{Katz} A. Katz, \emph{Quantitative disjointness of nilflows from horospherical flows}, preprint.
		\bibitem{KelmerOh} D. Kelmer, H. Oh, \emph{Exponential mixing and shrinking targets for geodesic flow on geometrically finite hyperbolic manifolds}, Preprint.
		\bibitem{KleinbockMargulis}D. Kleinbock, G. A. Margulis. \emph{Bounded orbits of nonquasiunipotent flows on homogeneous spaces}, Sinai’s Moscow Seminar on Dynamical Systems, Amer. Math. Soc. Transl. Ser. 2, 171, Amer. Math. Soc., Providence, RI (1996), pp. 141-172.
		\bibitem{Flows homogeneous spaces} D. Kleinbock, G. A. Margulis. \emph{Flows on homogeneous spaces and Diophantine approximation on manifolds}, Ann. of Math. (2) 148 (1998), no. 1, 339–360.
		\bibitem{KLW}D. Y. Kleinbock, E. Lindenstrauss, B. Weiss, \emph{On fractal measures and Diophantine approximation}, Selecta Math. 10 (2004), 479–523.
		
		\bibitem{MargulisThesis} G. Margulis, \emph{On some aspects of the theory of Anosov systems}, On Some Aspects of the Theory of Anosov Systems, p. 1–71. Springer, 2004.
		\bibitem{MauSchap}F. Maucourant, B. Schapira, \emph{Distribution of orbits in the plane of a finitely generated subgroup of $\operatorname{SL}(2,\mathbb{R})$}, Amer. J. of Math., 136 (2014), 1497–1542.
		\bibitem{McAdam}T. McAdam, \emph{Almost-primes in horospherical flows on the space of lattices}, arXiv 1802.08764, 2018.
		\bibitem{MelianPestana}M. V. Meli\'{a}n and D. Pestana, \emph{Geodesic excursions into cusps in finite volume hyperbolic manifolds}, Michigan Math. J. 40 (1993) 77–93.
		\bibitem{Matrix coefficients}A. Mohammadi, H. Oh, \emph{Matrix coefficients, Counting and Primes for orbits of geometrically finite groups}, J. European Math. Soc. 17, (2015), 837–897.
		\bibitem{joinings} A. Mohammadi, H. Oh, \emph{Classification of joinings for Kleinian groups,} Duke Math. J., \textbf{165} (2016), no. 11, 2155--2223.
		\bibitem{isolations} A. Mohammadi, H. Oh, \emph{Isolations of geodesic planes in the frame bundle of a hyperbolic 3-manifold}. arXiv: 2002.06579. 
		\bibitem{OhShah} H. Oh, N. Shah, \emph{Equidistribution and counting for orbits of geometrically finite hyperbolic groups,} J. Amer. Math. Soc. \textbf{26} (2013), 511--562.
		\bibitem{Patterson} S. J. Patterson, \emph{On a lattice-point problem in hyperbolic space and related
			questions in spectral theory}, Ark. Mat., Vol 26 (1988), p. 167--172.
		\bibitem{Raghunathan} M. Raghunathan,\emph{Discrete subgroups of Lie groups}. Math. Student 2007, Special Centenary Volume, 59–70 (2008). 
		\bibitem{Ratner} M. Ratner, \emph{Distribution rigidity for unipotent actions on homogeneous spaces}, Bull. Amer. Math. Soc. (N.S.), \textbf{24(2)} (1991), p.321–325.
		\bibitem{roblin} T. Roblin, \emph{Ergodicit\'e et \'equidistribution en courbure n\'egative.} M\'{e}m. Soc. Math. Fr. (N.S.), \textbf{95:vi+96} (2003).
		\bibitem{Sarnak} P. Sarnak, \emph{Asymptotic Behavior of Periodic Orbits of the Horocycle Flow and Eisenstein Series.} Comm. Pure Appl. Math., \textbf{34} (1981) p. 714-739.
		\bibitem{SarnakUbis} P. Sarnak, A. Ubis, \emph{The horocycle flow at prime times,} J. Math. Pure Appl. Vol. 103 (2015), p. 575-618.
		\bibitem{mixing}  P. Sarkar, D. Winter, \emph{Exponential mixing of frame flows for convex cocompact hyperbolic manifolds}, arXiv:2004.14551.
		\bibitem{Schapira} B. Schapira, \emph{Lemme de l’Ombre et non divergence des horosph\`eres d’une vari\'et\'e g\'eom\'etriquement finie}, Annales de l’Institut Fourier, 54 (2004), no. 4, 939-987.
		\bibitem{Shimizu} H. Shimizu, \emph{On discontinuous groups acting on the product of the upper half planes}, Ann. Math., II. Ser. 77 (1963) 33–71.
		\bibitem{friendly 2}B. Stratmann and M. Urba\'{n}ski, \emph{Diophantine extremality of the Patterson measure}, Math. Proc. Cambridge Philos. Soc. 140 (2006), 297–304.
		\bibitem{StratmannVelani} B. Stratmann, S. Velani, \emph{The Patterson measure for geometrically finite groups with parabolic elements, new and old}, Proc. of the London Math. Soc., Vol. s3-71, Issue 1 (1995), pp. 197-220.
		\bibitem{Strombergsson}A. Str\"ombergsson, \emph{On the Deviation of Ergodic Averages for Horocycle Flows}, J. Mod. Dyn., Vol. 7 (2013), pp. 291-328.
		\bibitem{sullivan} D. Sullivan, \emph{The density at infinity of a discrete group of hyperbolic motions}, Inst. Hautes Etudes Sci. Publ. Math., (50):171–202, 1979. 
		\bibitem{sullivan 2}D. Sullivan, \emph{Disjoint spheres, approximation by imaginary quadratic numbers and the logarthm law for geodesics}. Acta Math., 149: 215–237, 1982 1, 4, 12.
		\bibitem{distributions} N. Tamam, J. M. Warren, \emph{Distribution of orbits of geometrically finite groups acting on null vectors}, arXiv:2009.11968.
		\bibitem{winter} D. Winter, \emph{Mixing of frame flow for rank one locally symmetric manifolds and measure classification,} Israel J. Math., \textbf{210} (2015), 465--507.
	\end{thebibliography}
	\end{document}